\documentclass[a4paper,12pt]{amsart}
\usepackage[utf8]{inputenc}
\usepackage{mathtools}
\usepackage{amsmath}
\usepackage{amssymb}
\usepackage{graphicx}
\usepackage{changepage}
\graphicspath{ {images/} }
\usepackage{mathrsfs}
\usepackage{listings}
\usepackage{amsthm}
\usepackage[toc,page]{appendix}
\usepackage[margin=1.1in, headsep=10pt]{geometry}
\usepackage{framed,enumitem}
\usepackage{amsaddr}
\usepackage{dsfont}

\usepackage[colorlinks = true,
linkcolor = blue,
urlcolor  = blue,
citecolor = blue,
anchorcolor = blue]{hyperref}

\renewenvironment{abstract}
{\small
	\begin{center}
		\bfseries \abstractname\vspace{-.5em}\vspace{0pt}
	\end{center}
	\list{}{%
		\setlength{\leftmargin}{15mm}
		\setlength{\rightmargin}{\leftmargin}%
	}%
	\item\relax}
{\endlist}

\newtheorem{thm}{Theorem}[section]
\newtheorem{defn}[thm]{Definition}
\newtheorem{cor}[thm]{Corollary}

\newtheorem{prop}[thm]{Proposition}
\newtheorem{lem}[thm]{Lemma}
\newtheorem{rk}[thm]{Remark}

\title{Positive random walks and an identity for half-space SPDE's}

\author{Shalin Parekh}
\address{Columbia University}
\email{sp3577@columbia.edu}

\date{\today}


\begin{document}

\maketitle

\begin{abstract}\textbf{}
\\
The purpose of this article is threefold. First, we introduce a new type of boundary condition for the multiplicative-noise stochastic heat equation on the half space. This is essentially a Dirichlet boundary condition but with a nontrivial normalization near the boundary which leads to inhomogeneous transition densities (roughly, those of a Brownian \textit{meander}) within the associated chaos series. Secondly, we prove a new convergence result of the directed-polymer partition function in an octant to the multiplicative stochastic heat equation with this type of boundary condition, which in turn involves a detailed analysis of the aforementioned inhomogeneous Markov process. Thirdly, as a corollary, we prove a surprising equality-in-distribution for multiplicative-noise stochastic heat equations on the half space with \textit{different} boundary conditions. This identity may be seen as a precursor for proving Gaussian fluctuation behavior of supercritical half-space KPZ at the origin.
\end{abstract}

\tableofcontents

\newpage

\section{Introduction and Context}

The present work will focus on three related subject areas: uniform measures on collections of nearest-neighbor \textit{non-negative} paths (e.g., Brownian meander), intermediate-disorder directed polymers weighted by such measures, and multiplicative-noise SPDE's in a half-space.
\\
\\
We begin our discussion with multiplicative-noise SPDE's. The multiplicative noise stochastic heat equation has been a popular subject of research within stochastic analysis and mathematical physics in recent years. This equation arises naturally in the context of directed polymers and interacting particle systems, as a \textit{weak scaling limit}. In spatial dimension one, the multiplicative-noise stochastic heat equation is also related to the so-called KPZ equation via the Hopf-Cole transform, and may be solved by the classical Ito-Walsh construction \cite{Wal86} or by more modern techniques such as regularity structures \cite{HL18}. In the present article, we consider the stochastic heat equation with multiplicative noise on a \textit{half-line}: \begin{equation}\label{SHE}\partial_T Z = \frac12 \partial_X^2 Z + Z\xi, \;\;\;\;\;\;\;\;\;\;X \geq 0, T\geq 0, \tag{SHE}\end{equation} where $\xi$ is a Gaussian space-time white noise on $\Bbb R_+ \times \Bbb R_+$. We consider two different types of boundary conditions, Robin and Dirichlet. Let us first let us write the Robin boundary condition of parameter $A \in \Bbb R$: \begin{equation}\label{a}\partial_XZ(T,0) = AZ(T,0).\end{equation} This type of boundary condition has been considered in \cite{CS16, Par18, GPS17, BBCW18} in the context of interacting particle systems, and a robust solution theory has been developed in \cite{GH17} using techniques of \cite{Hai14}. This boundary condition transforms into a Neumann boundary condition for the half-space KPZ equation upon taking the logarithm. Next, we consider the Dirichlet boundary condition for half-space \eqref{SHE}: \begin{equation}\label{b}Z(T,0) = 0.\end{equation} This type of boundary condition was considered (for instance) in \cite{GLD12}, in the context of directed polymers near an absorbing wall. Again, one can make sense of the equation using classical techniques of \cite{Wal86} or more modern ones such as \cite{Hai14}. Our main result compares these two different types of boundary conditions:

\begin{thm}\label{mr} For $A \in \Bbb R$, let $Z_A(T,X)$ denote the solution of \eqref{SHE} with Robin-boundary parameter $A$ as in \eqref{a}, and delta initial data $Z_A(0,X)=\delta_0(X)$. Let $Z_{Dir}(T,X)$ be the solution to \eqref{SHE} with Dirichlet boundary condition \eqref{b}, with initial data $Z_{Dir}(0,X) = e^{B_X-(A+\frac12)X},$ where $B$ is a Brownian motion. Then we have the following equality of distributions: \begin{equation}\label{c} Z_A(T,0) \stackrel{d}{=} \lim_{X\to 0} \frac{Z_{Dir}(T,X)}{X}.
\end{equation}
\end{thm}
Let us now discuss the motivation for this result, the contexts in which it has arised, and the implications it holds. The main motivation for this theorem comes from an algebraic identity which is given in Theorem 8.1 of \cite{BBC18}. Specifically, that theorem is an identity-in-distribution for directed polymers with log-gamma weights, and our main goal was to take the SPDE limit of that identity in order to see whether useful information could be obtained in the context of half-space KPZ universality. In turn, this required us to prove a general convergence result for directed polymers (stated below as Theorem \ref{bsk}) which, as we will see, involved analyzing some interesting objects in their own right, such as the Brownian meander.
\\
\\
It is striking (although we will not use it) that right side of \eqref{c} is actually $(\partial_X Z_{Dir})(T,0)$. It is not even clear why the limit in the right side of \eqref{c} should exist, since the spatial regularity of $Z_{Dir}$ is far from differentiable. In Section 4 we prove that the mild form of the Dirichlet-boundary \eqref{SHE} actually exists and also that the limit in \eqref{c} is actually well defined (Corollary \ref{lim}).
\\
\\
Mathematically, we believe that Theorem \ref{mr} is interesting because it hints at an intriguing ``duality" between the initial data of a solution to the half-space \eqref{SHE} and the boundary conditions one imposes on it. This duality can in turn be exploited in order to obtain useful results on quantities of interest. For instance, Theorem \ref{mr} already provides a useful and nontrivial coupling between the Robin-boundary random variables $Z_A(T,X)$ for different values of $A \in \Bbb R.$ Specifically, it shows that for $A<A'$, the random variable $Z_{A'}(T,0)$ is stochastically dominated by $Z_A(T,0).$ Indeed, this follows from Theorem \ref{mr}, together with the fact that $e^{B_X-(A+1/2)X} \geq e^{B_X-(A'+1/2)X}$ for all $X$, and a domination result for the Dirichlet boundary \eqref{SHE} (which says that two solutions coupled to the same noise are dominated for all time if the respective initial data are dominated). Using this, one may potentially obtain useful information about the Neumann-boundary KPZ equation which was considered in \cite{CS16}. It was conjectured in \cite{Par18} that one has the almost-sure convergence $$\lim_{T \to \infty} \frac{1}{T} \log Z_A(T,0) = \begin{cases} -\frac{1}{24}, & A\geq -1/2 \\ (A+1/2)^2-\frac{1}{24}, & A\leq -1/2 \end{cases},$$ which would give the exact law of large numbers for Neumann-boundary KPZ. Unfortunately Theorem \ref{mr} alone is not enough to obtain this strong of a result since it is not entirely clear where the $-\frac1{24}$ term would come from. Nevertheless, some nontrivial quantitative information can be obtained by combining the aforementioned stochastic dominance property with (for instance) Theorem 1.1 of \cite{Par18}. Moreover, it is plausible and even hopeful that a clever use of Theorem \ref{mr} (perhaps combined with some new ideas, techniques, or generalizations) could lead to quantitative results which are close to the above expression. The reason that this is plausible is that a Feynman-Kac representation associated to the right-hand side of \eqref{c} was used in Section 1.3 of \cite{Par18} to obtain those conjectural values in the first place. More than just computing the above limit, we are also interested in computing the limiting distribution of the fluctuations around the mean value. These should be of order $T^{1/2}$ and Gaussian in the case when $A<-1/2$, and they should be of order $T^{1/3}$ and random-matrix theoretic otherwise (with separate cases when $A=-1/2$ and $A>-1/2)$, see for instance \cite{BBCW18, BBCS16, BBC16}.
\\
\\
This brings us to the method of proof of Theorem \ref{mr}. As suggested above, it will be proved using an approximation via \textit{directed polymers} with very specific weights, where a discrete version of this identity holds. This identity comes from the general exactly solvable framework of half-space Macdonald processes which was developed in \cite{BBC18}, which was in turn inspired by works of \cite{COSZ14, BC14, OSZ14, BR01} and much more.
\\
\\
Directed polymers are a natural probabilistic object which were first introduced in \cite{HH85,IS88}. They generalize directed first- and last-passage percolation, and have deep connections to statistical mechanics and stochastic analysis. Specifically, we consider an environment $\{\omega_{i,j}\}$ consisting of iid, mean-zero, finite-variance random variables. The standard deviation of the weights is referred to as \textit{inverse temperature}. One may define a partition function $Z^{\omega}(n,x)$ as a sum over all nearest-neighbor simple-random walk paths of length $n$ starting from $x$, of the product of all weights $e^{\omega_{i,j}}$ along the path. Using this partition function, one may also define random Markovian transition densities associated to this environment $\omega$, wherein a nearest-neighbor path is more likely to travel in a direction with higher weights. Then one may ask many natural questions, such as the existence of infinite-volume limits of these path measures, and their typical fluctuation scale as well as the typical height fluctuation scale of the associated partition function \cite{Com17}.
\\
\\
Many seminal results on these directed polymer systems have been proved, perhaps most notably that there is a phase transition which becomes apparent in high dimensions. Specifically, in spatial dimensions $\ge 3$, there is a strictly positive critical value of the inverse temperature below which weak disorder holds, meaning that the fluctuations of a typical polymer path look like Brownian motion and one may construct infinite-length path measures \cite{CY06, Com17}. Such polymers are said to exhibit \textit{weak disorder.} In contrast, lower-dimensional polymers at any finite inverse-temperature are now known to be characterized by \textit{strong disorder}, meaning that the path fluctuations are quite different and there is no sensible notion of an infinite volume Gibbs measure \cite{Com17}. The results of \cite{AKQ14a,AKQ14b} examined the partition function in a regime which lies in \textit{between} strong and weak disorder. Specifically, in spatial dimension one, they scaled the inverse temperature of the model like $O(n^{-1/4})$ and simultaneously applied diffusive scaling to the partition function, and there they observed that the fluctuations are governed by \eqref{SHE} and that the path measures themselves have a continuum analogue which is formally described by a Radon-Nikodym derivative with respect to Brownian motion, with drift given by the spatial derivative of the KPZ equation. Recent work of \cite{CD18, CSZ18} has investigated the same intermediate-disorder behavior in two spatial dimensions, where the scaling $O(n^{-1/4})$ is replaced by a more complicated logarithmic term. The paper \cite{Wu18} then extended the work of \cite{AKQ14a} to the case of half-space, in the case of \textit{Robin} boundary condition.
\\
\\
We will be interested in the analogous half-space question of intermediate-disorder fluctuations of the directed polymer partition function associated to \textit{uniform non-negative path measures}. Specifically, let
\begin{itemize}
\item $\mathbf P_x^n$ denote the uniform measure on the collection of all nearest-neighbor paths of length $n$ starting from $x$ and never going below $0$.
\item $\omega_{i,j}$ be iid mean-zero, variance-one random variables.
\item $f_n$ be a sequence of functions such that $f_n(n^{1/2}\;\cdot )$ converges (as $n \to \infty$) to some function $f(\cdot)$ in the Holder space $C^{\alpha}_{loc}(\Bbb R_+)$, for all $\alpha \in (0,1/2$).
\end{itemize}
Letting $\mathbf E_x^n$ denote the expectation with respect to $\mathbf P_x^n$, and letting $S$ denote the canonical process associated to $\mathbf P_x^n$, one defines the directed-polymer partition function as follows: 
$$Z_k^{\omega}(n,x):= \mathbf E_x^n\bigg[f_k(S_n)\prod_{i=0}^n (1+k^{-1/4}\omega_{i,S_i})\bigg]. $$ Note that the expectation is taken only with respect to the random walk, \textit{conditional} on the environment $\omega_{i,j}$ (which is always assumed to be independent of the walk). 
We then have the following result, which will be proved in Section 5 below.
\begin{thm}[Theorem \ref{thm2}]\label{bsk} Assume that the $\omega_{i,j}$ have two moments. Consider the rescaled partition function $$\mathscr Z_n(T,X):= Z_n^{\omega}(nT,n^{1/2}X), $$ where the quantity on the right side is defined by linear interpolation for non-integer values. Then $\mathscr Z_n$ converges in law to an SPDE on the half-sapce (this SPDE is explicit and given in mild form by \eqref{spde0} below). The convergence occurs in the sense of finite-dimensional distributions. If we assume that the $\omega_{i,j}$ have $p>8$ moments, then convergence actually occurs in the space $C(\Bbb R_+\times \Bbb R_+)$ with respect the topology of uniform convergence on compact sets.
\end{thm}

This theorem investigates the intermediate-disorder behavior of directed polymers weighted by a highly nontrivial random-walk measure: specifically one conditions the associated random walk to stay positive rather than reinforcing it (e.g. reflecting) at the boundary. This leads to inhomogeneous transition densities within the associated chaos series, which (as we will see) can be related to the derivative $\partial_XZ_{Dir}(T,0)$ near the boundary, i.e., the right-hand side in Theorem \ref{mr}.
\\
\\
We remark that the limiting SPDE \eqref{spde} is a multiplicative-noise heat equation on a half-space with a ``normalized Dirichlet" boundary condition. It has a formal Feynman-Kac interpretation which is given by considering the so-called \textit{Brownian meander} \cite{DIM77, DI77} on a finite-time-interval, and re-weighting it by its integral against a space-time white noise field. More specifically, if $\mathscr P_t^T(X,Y)$ denotes the inhomogeneous Markov transition density at time $t$ of Brownian motion started from $X$ and conditioned to stay positive until time $T\ge t,$ then the limit $\mathscr Z$ in Theorem \ref{bsk} is the solution to the SPDE given in Duhamel-form by
\begin{equation}\label{spde0}
    \mathscr{Z}(T,X) = \int_{\Bbb R^+} \mathscr{P}_T^T(X,Y)f(Y)dY + \int_0^T \int_{\Bbb R^+} \mathscr{P}^T_{T-S}(X,Y) \mathscr{Z}(S,Y) \xi(dYdS).
\end{equation}
An important step towards proving Theorem \ref{mr} will be to make sense of this expression at $X=0$, and then to show that this can in turn be related to the derivative of Dirichlet-boundary SHE at the origin.
\\
\\
It should be noted that we work with a simplified version of the partition function as opposed to much of the previous literature: \cite{AKQ14a, CSY03} and related works. There the partition function $Z_k^{\omega}(n,x)$ is defined with weights $e^{k^{-1/4}\omega_{i,S_i}}$ instead of the quantity $1+k^{-1/4}\omega_{i,S_i}$ which we have used above. The reason for this is that the latter object is mathematically simpler because it is already renormalized (in some sense), and hence leads to simpler proofs and less stringent moment restrictions. However, it should be noted that the exponential version is more natural from the physical point of view, and entire results such as \cite{DZ16} have been devoted to finding the correct renormalization and phase transition behavior for that version, as a function of the moment assumptions.
\\
\\
The proof of Theorem \ref{bsk} will lead to some new technical results related to the uniform measures $\mathbf P_x^n$, i.e., random walk conditioned to stay non-negative. These will be collected in an appendix at the end of the paper. Perhaps most interestingly, we will prove a coupling result for such random walks in the nearest-neighbor case, and then we will use that coupling to show the following concentration property: there exist constants $c,C>0$ (independent of $n,x \ge 0$) such that for all $u>0$ and all $k \leq n$ one has that $$\mathbf P_x^n\big(\sup_{0 \le i\le k} |S_i-x| >u \big) \le Ce^{-cu^2/k}.$$
We remind the reader that $S_i$ is the \textit{conditioned} walk. Such a result extends the work of many known results for random walks conditioned to stay positive. The study of such random walks started with the invariance principle of \cite{Ig74}, further generalized in \cite{Bol76}. Later, the study expanded considerably, with local limit theorems \cite{Car05} and expansions to heavy-tailed increments \cite{CC08}. We will see that some of the estimates we derive are similar in spirit to some of those works, but the intricate details are somewhat different. In the end, we will give original proofs of all of these technical results, because the highly specific estimates needed to prove Theorem \ref{bsk} were not found in those other references (since our random walk does not necessarily start at zero). However, the upshot is that all of the proofs in the appendix will be entirely elementary, using only classical techniques.
\\
\\
\textbf{Outline:} In Section 2, we will state more precise versions of the theorems stated in the introduction, and outline the basic idea of the proof. In Section 3, we will perform an intricate analysis of the transition densities associated to the measures $\mathbf P_x^n$ which will lead to very useful estimates. In Section 4, we will develop the existence and uniqueness theory of the limiting SPDE \eqref{spde} from Theorem \ref{bsk}, and as a corollary we prove that $\partial_X Z_{Dir}(T,0)$ exists. In Section 5, we prove Theorem \ref{bsk} by using the estimates developed in Section 3. In the appendix we derive some elementary but useful concentration bounds for the measures $\mathbf P_x^n$.
\\
\\
\textbf{Acknowledgements:} We thank Ivan Corwin for suggesting that interesting identities could potentially be obtained from Theorem 8.1 of \cite{BBC18}, and also for reading portions of this preliminary draft. The author was partially supported by the Fernholz Foundation's ``Summer Minerva Fellows" program, as well as summer support from Ivan Corwin's NSF grant DMS:1811143.

\section{Main Results}

In this section, we will reformulate and clarify the results stated in the introduction, and we will provide a rough outline of the proof which hopefully shows how those three theorems are intertwined. We always abbreviate $\Bbb R_+$ as nonnegative reals, and $\Bbb Z_{\ge 0}$ as non-negative integers.

\begin{defn}[Mild Solution] Define the Dirichlet-Boundary heat kernel 
\begin{equation}
    \label{dir} P_t^{Dir}(X,Y):= \frac{1}{\sqrt{2\pi t}} \big( e^{-(X-Y)^2/2t} - e^{-(X+Y)^2/2t} \big).
\end{equation}
Let $\xi$ be a space-time white noise defined on a probability space $(\Omega, \mathcal F, \Bbb P),$ and let $\mu$ be a random Borel measure on $\Bbb R_+$. A space-time process $Z_{Dir} = (Z_{Dir}(T,X))_{T,X \ge 0}$ is a mild solution of the Dirichlet-boundary \eqref{SHE} with initial data $\mu$ if $\Bbb P$-almost surely, for all $X,T \ge 0$ one has that $$Z_{Dir}(T,X) = \int_{\Bbb R_+} P_T^{Dir}(X,Y) \mu(dY) + \int_0^T \int_{\Bbb R_+} P_{T-S}^{Dir}(X,Y) Z_{Dir}(S,Y) \xi(dS,dY),$$ where the integral against $\xi$ is meant to be interpreted in the Ito-Walsh sense \cite{Wal86}.
\end{defn}

The definition of the Robin-boundary \eqref{SHE} $Z_A$ is very similar, but one replaces the Dirichlet heat-kernel with the Robin-boundary one throughout. We refer the reader to [Par18, Definition 4.1] for more details.
\\
\\
The proof of Theorem \ref{mr} will be obtained by approximating both $Z_{Dir}$ and $Z_A$ by the partition function of a directed polymer with log-gamma weights, and using a known identity for such directed polymers which allows to switch the boundary weights with those on the initial data without changing the distribution of the partition function along the boundary \cite{BBC18} (Theorem 8.1). The approximation argument will strongly emulate the arguments given in \cite{Wu18, AKQ14a} although there are numerous new challenges which make the convergence result rather difficult and technical. These additional difficulties are a byproduct of the inhomogeneous Markov transition densities for random walks conditioned to stay above zero.
\\
\\
Let us explicitly state the Dirichlet-boundary approximation result now. For each $n \in \Bbb N$, let $\omega^n = \{\omega^n_{i,j}\}_{i \geq j \geq 0}$ denote a random environment in the principal octant of $\Bbb Z^2$, with the following properties:

\begin{itemize}

    \item The random variables $\{\omega^n_{i,j}\}_{i \geq j \geq 1}$ are i.i.d. and so are the random variables $\{\omega^n_{i0}\}_{i \geq 0}$. These two collections are independent.
    
    \item For $j>0$, the $\omega^n_{i,j}$ have finite second moment. Furthermore, one has $\Bbb E[\omega^n_{i,j}] = 0$ and $\Bbb E[(\omega^n_{i,j})^2] = 1+o(1)$ as $n \to \infty$.
    
    \item For $j=0$, $\log(1+n^{-1/4}\omega^n_{i,j})$ has finite second moment; moreover there exist $\mu,\sigma \in \Bbb R$ such that $\Bbb E[\omega^n_{i,j}] = \mu n^{-1/4} +o(n^{-1/4}), $ and $var(\omega^n_{i,j}) = \sigma^2+o(1)$ as $n \to \infty$.
    
\end{itemize}

The following result is the main technical contribution of our paper.

\begin{thm}\label{thm2} Let $\omega^n$ be defined as above. Define the random partition function $$Z_n(p,q):=\sum_{\pi:(0,0) \to (p,q)}2^{-\#\{i\leq p+q\;:\;\pi(i)\neq 0\}} \prod_{i=0}^{p+q} (1+n^{-1/4} \omega^n_{i\pi(i)}),$$ where the sum is taken over all up-right paths from $(0,0)$ to $(p,q)$ which stay in the octant $\{(i,j): i \geq j \geq 0\}$. Let $\Phi$ denote the cdf of a standard normal distribution. We define the rescaled processes $$\mathscr Z_n(T,X) := \frac{1}{2\Phi\big(\frac{X+n^{-1/2}}{\sqrt{T}}\big)-1} \cdot Z_n( nT+n^{1/2}X, nT),\;\;\;\;\;\;\;\;\;\;\;\;T,X \ge 0$$ where we interpolate linearly between integer values of $Z_n$. Assuming that all weights $\omega_{i,j}$ have $p>8$ moments bounded independently of $n$, the processes $\mathscr Z_n$ converges in distribution (as $n \to \infty$, with respect to the locally uniform topology on $C(\Bbb R_+ \times \Bbb R_+)$) to the space-time process \begin{equation}\label{d} \frac{Z_{Dir}(T,X)}{2 \Phi(X/\sqrt{T})-1},\;\;\;\;\;\;\;\;\;\;\;\;T,X \ge 0\end{equation} where $Z_{Dir}$ solves \eqref{SHE} with Dirichlet boundary condition \eqref{b}, and has initial data $Z_{Dir}(0,X)=e^{B_X+(\mu-\frac12 \sigma^2)X}$ where $B$ is a standard Brownian motion. If we only assume that the weights have two moments (not $p>8$), one still has convergence of finite-dimensional marginals.
\end{thm}

\begin{rk} It is very important to note that Theorem \ref{thm2} is also valid at $X=0$, but one needs to replace equation \eqref{d} by the limit $\lim_{X \to 0} \frac{Z_{Dir}(T,X)}{2\Phi(X/\sqrt{T})-1}.$ The fact that this limit actually exists will be deduced from Corollary \ref{lim}. Therefore, there are really two different regimes in which one should interpret Theorem \ref{thm2}. One regime is $X>0$, where the result merely says that $Z_n( nT+n^{1/2}X,  nT)$ converges to $Z_{Dir}(T,X)$. The other (more interesting) regime is the case where $X=0$, in which case the theorem says that $(\pi n T/2)^{1/2} Z_n( nT,  nT)$ converges in law to $\lim_{X \to 0} \frac{Z_{Dir}(T,X)}{2\Phi(X/\sqrt{T})-1},$ which is equivalent to $$n^{1/2}Z_n(nT, nT) \to \lim_{X\to 0} \frac{Z_{Dir}(T,X)}{X}.$$ This $X=0$ case is where the true power (and difficulty) of Theorem \ref{thm2} lies. The nice thing about our approach will be that the proof will simultaneously cover both regimes at once, without considering separate cases. In fact, we will see that convergence even takes place in a parabolic Holder space of the appropriate regularity (assuming there are more than eight moments).
\end{rk}

We now combine this result with the Robin-boundary result of \cite{Wu18} and the log-gamma identities of \cite{BBC18} in order to obtain the following corollary, which clearly implies Theorem \ref{mr}. In what follows, we denote by $ \Gamma^{-1}(\theta,c)$ the inverse-gamma distribution of shape parameter $\theta$ and scale parameter $c$, i.e., the law of the random variable $cX$, where $X$ has pdf given by $$f(x) = \frac{x^{-\theta-1}}{\Gamma(\theta)} e^{-1/x},\;\;\;\;\;\;\; x>0. $$ We will also write $\Bbb E[\Gamma^{-1}(\theta,c)]$ to denote the expectation of such a random variable.

\begin{cor}\label{cor} For $n \in \Bbb N$, let $\zeta_1^n = \{\zeta_1^n(i,j)\}_{i \geq j \geq 0}$ and $\zeta_2^n = \{\zeta_2^n(i,j)\}_{i \geq j \geq 0}$ be fields of independent random variables with the following distributions $$\zeta^1_n(i,j) \sim \begin{cases} \Gamma^{-1}(2\sqrt{n}, \;\frac12 \Bbb E[\Gamma^{-1}(2\sqrt{n})]^{-1}),& i \neq j \\ \\ \Gamma^{-1}(\sqrt{n}+A+\frac12,\;\frac12 \Bbb E[\Gamma^{-1}(2\sqrt{n})]^{-1}),& i=j \end{cases}$$
$$\zeta^2_n(i,j) \sim \begin{cases} \Gamma^{-1}(2\sqrt{n},\;\frac12 \Bbb E[\Gamma^{-1}(2\sqrt{n})]^{-1}),& j \neq 0 \\ \\ \Gamma^{-1}(\sqrt{n}+A+\frac12,\; \frac12 \Bbb E[\Gamma^{-1}(2\sqrt{n})]^{-1}),& j=0.  \end{cases}$$
Let $Z_n^1$ and $Z_n^2$ denote the associated partition functions, i.e., $$Z^{\alpha}_n:= \sum_{\pi:(0,0) \to (\lfloor nT \rfloor,\lfloor nT\rfloor)} \prod_{i=0}^{2\lfloor nT\rfloor } \zeta^{\alpha}_n(i,\pi(i)), \;\;\;\;\;\;\;\;\text{ for } \alpha \in \{1,2\}.$$ Here the sum is taken over all up-right paths from $(0,0)$ to $(n,n)$ which stay in the octant $\{(i,j): i \geq j \geq 0\}$. Then the following are true:

\begin{enumerate}
    \item $\sqrt{n}Z_n^1$ converges in distribution as $n \to \infty$ to the left-hand side of \eqref{c}.
    
    \item $\sqrt{n}Z_n^2$ converges in distribution as $n \to \infty$ to the right-hand side of \eqref{c}.
    
    \item For every $n$, one has $Z^1_n \stackrel{d}{=} Z^2_n.$
\end{enumerate}

\end{cor}

\begin{figure}[h]
	\centering
	\includegraphics[scale=0.85]{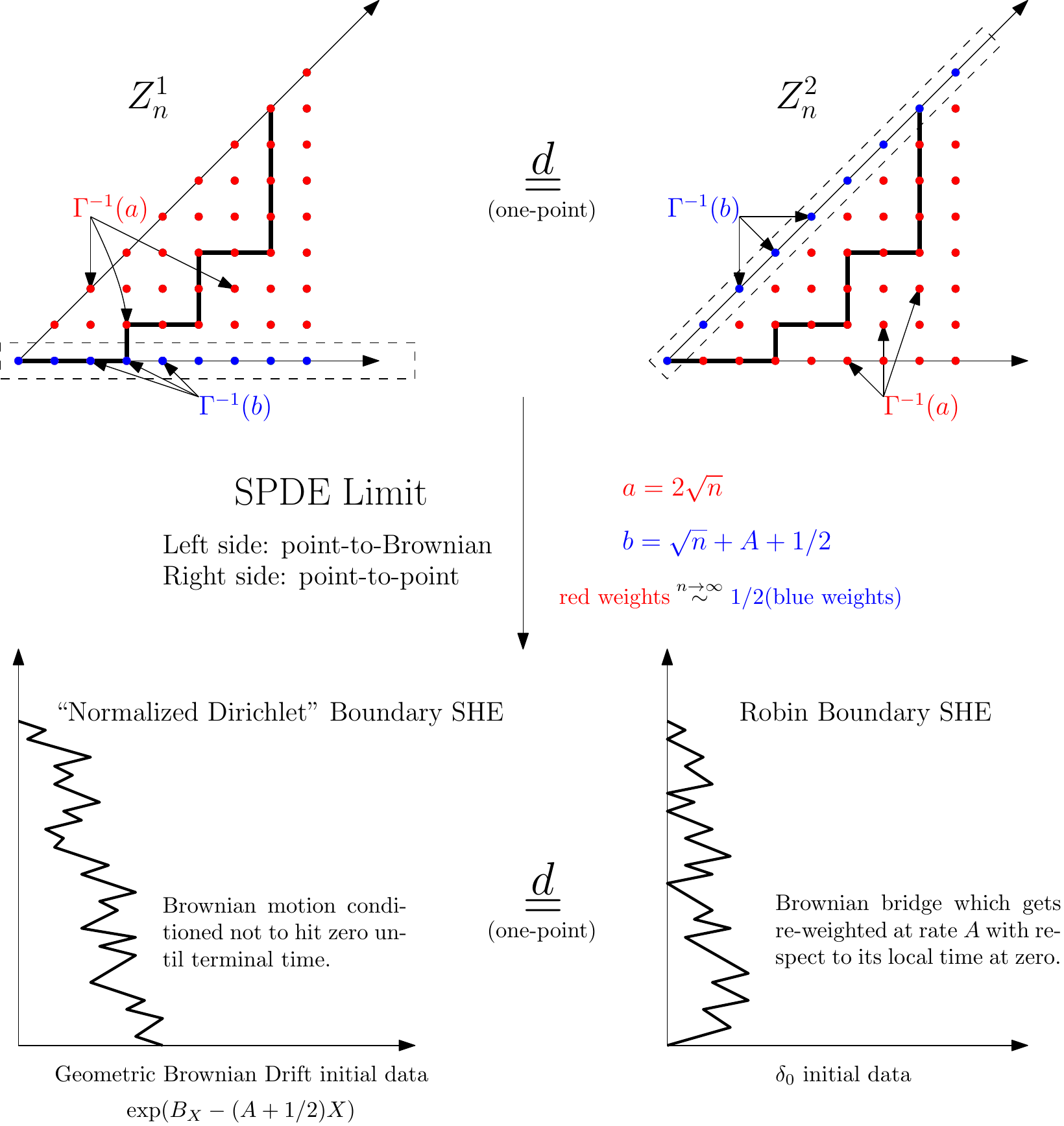}
	\caption{Depiction of the polymer approximation in Corollary \ref{cor}. The weight of a given path is the product of the weights along it. The partition function $Z^{\alpha}_n$ is given by summing the weights of all up-right paths from $(0,0)$ to $(n,n)$ which stay in the octant. We have given the respective Feynman-Kac representations for the limiting SPDEs.}
\end{figure}

\begin{proof} Item \textit{(1)} is proved as Theorem 5.1(B) of \cite{Wu18} using techniques from \cite{AKQ14a}. Item \textit{(3)} is proved in Theorem 8.1 of \cite{BBC18} by developing the theory of half-space Macdonald processes. Thus we only need to prove Item \textit{(2)} and this will be done using Theorem \ref{thm2}, in the special case where $X=0$. As in Theorem 4.5 of \cite{AKQ14a}, we define a family of independent weights $\omega^n = \{ \omega^n_{i,j} \}_{i \geq j \geq 0}$ according to the rule: $$2\zeta^2_n(i,j) = 1+ (4n)^{-1/4} \omega^n_{i,j},\;\;\;\;\;\;\;\;\;j>0,$$ $$\zeta^2_n(i,0) = 1+ n^{-1/4} \omega^n_{i0}.\;\;\;\;\;\;\;\;\;\;\;\;\;\;\;\;\;\;$$ There are now three things to verify, corresponding to the three bullet points preceding Theorem \ref{thm2}. Using the fact that $$\Bbb E[\Gamma^{-1}(\theta) ] = \frac{1}{\theta-1},\;\;\;\;\;\;\;\;\; var(\Gamma^{-1}(\theta)) = \frac{1}{(\theta-1)^2(\theta-2)},$$ one gets the desired asymptotics on $\Bbb E[\omega_{i,j}]$ and on $\Bbb E[(\omega_{i,j})^2]$, with $\mu = -A$ and $\sigma^2=1$. This proves the corollary (and thus also Theorem \ref{mr}).
\end{proof}

Thus by using Corollary \ref{cor} and the results quoted therein, we have reduced the proof of Theorem \ref{mr} to that of Theorem \ref{thm2} and this is what we will focus on now.
\\
\\
Since the sum defining the partition function in the preceding results is over all upright paths which stay in the principal octant of $\Bbb Z^2$, it is natural to try to relate those quantities to reflecting random walk measures. However, if one does asymptotics in Corollary \ref{cor}, one may verify that $\zeta_n^2(j,j) \sim 1/2$ as $n \to \infty$. What this means is that instead of pure reflection, our random walk path loses mass by a factor of $1/2$ each time it hits zero. Hence, it is clear that the analysis in proving Theorem \ref{thm2} will involve taking a close look at such random walk measures, as well as directed polymers weighted by such measures, as suggested in the introduction.
\\
\\
More precisely, fix some $x \in \Bbb Z_{\geq 0}$, and define a sample space of non-negative random walk trajectories by $$\Omega_x^n := \{ (s_0,...,s_{n}) \in \Bbb Z^{n}: |s_{i+1}-s_i|=1, s_i \geq 0, s_0=x\}.$$ Define a sub-probability measure $\mu_x^n$ and a probability measure $\mathbf{P}_x^n$ on $\Omega_x^n$ by $$\mu_x^n(S) := 2^{-n}, \;\;\;\;\;\; \mathbf{P}_x^n(S) := \frac{1}{\# \Omega_x^N}=\frac{\mu_x^n(S)}{\mu_x^n(\Omega_x^n)} ,\;\;\;\;\;\;\;\text{ for all } S \in \Omega_x^n.$$ As an intermediate step in proving Theorem \ref{thm2}, we obtain the following result.

\begin{thm}\label{rw}
With the above notation, the following are true.
\begin{enumerate}
    \item (Markov Property) Fix $n,x \geq 0$. Let $S = (S_k)_{k=0}^{n}$ denote the coordinate process associated to $\mathbf{P}_x^n$, i.e., $S$ is a $\Omega_x^n$-valued random variable with law $\mathbf{P}_x^n$. Then $(S_k)_{k=0}^{n}$ is a time-inhomogeneous Markov process, in fact conditionally on $(S_k)_{k \leq K}$ with $K<n$, the process $(S_{k+K})_{k=0}^{n-K}$ is distributed according to $\mathbf{P}_{S_{K}}^{n-K}.$ One has explicit transition densities for $0\leq i_1 <...<i_k\leq n$: $$\mathbf{P}_x^N( S_{i_1} = s_1 ,..., S_{i_k} = s_k) = \mathfrak{p}_{i_1}^N(x,s_1) \mathfrak{p}_{i_2-i_1}^{N-i_1}(s_1,s_2) \cdots \mathfrak{p}_{i_k-i_{k-1}}^{N-i_{k-1}}(s_{k-1},s_k).$$ where $\mathfrak{p}_n^N$ is given in Definition \ref{fuk} below.
    
    \item (Mass) For every $x \in \Bbb Z_{\geq 0}$, the total mass of $\mu_x^n$ is asymptotically $(x+1)\sqrt{\frac{2}{\pi n}}$: $$\lim_{n \to \infty} n^{1/2} \mu_x^n(\Omega_x^n) = (x+1)\sqrt{2/\pi}.$$
    
    \item (Concentration) There exist $C,c>0$ such that for every $x \ge 0$, every $0 \le m\le k \le n$, and every $u>0$ one has that $$\mathbf P_x^n\big(\sup_{m \le i\le k} |S_i-S_m| >u \big) \le Ce^{-cu^2/(k-m)}.$$ 
    
    \item (Convergence of transition densities) Let $\mathfrak{p}_n^N$ be as in item (1). One has the convergence $$\epsilon^{-1} \mathfrak{p}_{\epsilon^{-2}t}^{\epsilon^{-2}T}(\epsilon^{-1}X, \epsilon^{-1}Y) \longrightarrow \mathscr{P}_t^T(X,Y),$$ where $\mathscr{P}_t^T$ is the transition probability for a certain (inhomogeneous) Markov process defined in Definition \ref{sht} below. Moreover, for fixed $(t,T,X)$ the convergence in the $Y$-variable occurs in $L^p(\Bbb R_+,e^{aY}dY)$ for every $p \in [1,\infty)$.
    
    \end{enumerate}
    
\end{thm}

The first part of the theorem is quite elementary, and the last part is a more local version of the results of \cite{Ig74, Bol76}. The third part is new (as far as we know), and the second part will actually just follow from the local central limit theorem. All proofs may be found in the appendix (except (\textit{4}), which is proved in Section 3).

\begin{rk} One can actually formulate an invariance principle for this family of measures, which was done in greater generality in \cite{Ig74,Bol76}. Fix $X,T\ge 0$. For each $x,N \ge 0$, let $(S^{x,N}_n)_{n=0}^{N}$ be distributed according to $\mathbf P_x^N$. Then the processes $(N^{-1/2}S_{Nt}^{N^{1/2}X,NT})_{t \in [0,T]}$ converge in law (with respect to the uniform topology on $C[0,T]$, as $N \to \infty$) to a time-inhomogeneous Markov process $B$ on $[0,T]$ whose transition densities $\mathscr{P}_t^T(X,Y)$ are given by the limit in item (3). This limiting process $B$ may be interpreted as Brownian motion conditioned to stay positive until time T, see Proposition \ref{mp}. We will give an elementary discussion of how to prove this at the very end of the appendix, but it will not be needed for the results above.
\end{rk}

Let us now discuss the basic idea of the proof of Theorem \ref{thm2}. We only consider the special case when $(X,T) = (0,1)$ because this is enough to give the main idea. Denote by $\mathbf{E}_{KRW}$ the expectation with respect to a reflected random walk of length $2n$ that is killed at the origin with probability $1/2$ (i.e., the one whose transition density is equal to $p_n^{(1/2)}$ which is defined in section 3 below). One rewrites the partition function appearing in Theorem \ref{thm2} as a discrete Feynman-Kac formula for this killed random walk:
\begin{align*}Z_n &= \sum_{\pi:(0,0) \to (n,n)} 2^{-\#\{i \leq 2n \;:\; \pi(i)>0 \}} \prod_{i=0}^{2n} (1+n^{-1/4} \omega_{i\pi(i)}^n) \\ &= \mathbf{E}_{KRW} \bigg[ z^n_0(S_{T_n})\prod_{i=0}^{T_n-1} (1+n^{-1/4} \hat{\omega}_{iS_i}) \cdot 1_{\{\text{survival}\}} \bigg],
\end{align*}
where
\begin{itemize}

    \item $\hat{\omega}^n_{ij}$ is defined to be $\omega_{(n-i)(n-j)}^n$ for all $i,j$. 
    
    \item The expectation $\mathbf{E}_{KRW}$ is taken \textit{only} with respect to the random walk $S$, i.e., \textbf{conditional} on the $\omega^n_{i,j}$ (which are always assumed to be independent of $S$).
    
    \item $T_n$ is the first time that $S$ hits the diagonal line $\{(2n-i,i):0\leq i \leq 2n\}.$
    
    \item $z^n_0(x) := \prod_{i=0}^x (1+n^{-1/4}\omega^n_{i0})$ can be thought of as a sort of ``initial data" for the above discrete Feynman-Kac representation.
    
    \item $\{$survival$\}$ is the event that the random walk actually survives up to time $2n$ (or equivalently, up to time $T_n$).
\end{itemize}

Now, using Theorem \ref{rw}(2), one finds that $\mathbf{P}_{KRW}($survival$) \approx \sqrt{2/\pi n}$. Moreover, we can make the approximation $T_n \approx 2n$ for reasons justified later (see Proposition \ref{oqr}). Combining this with the above gives \begin{align*}
    \sqrt{\frac{\pi n}{2}} Z_n &\approx \mathbf{E}_{KRW} \bigg[  z^n_0(S_{2n})\prod_{i=0}^{2n} (1+n^{-1/4} \hat{\omega}_{iS_i})\bigg| \text{ survival }\bigg] \\ &= \mathbf{E}_{KRW} \bigg[ z^n_0(S_{2n}) \sum_{k=0}^{2n} n^{-k/4} \sum_{i_1<...<i_k} \prod_{j=1}^k \hat{\omega}_{i_jS_{i_j}} \;\bigg| \text{ survival }\bigg]
\end{align*} 
In the notation of Theorem \ref{rw}, the killed random walk conditioned to survive has law $\mathbf{P}^n_x$ and the associated Markov process has transition densities $\mathfrak{p}_n^N$. Thus (using theorem \ref{rw}(1)) the expectation in the preceding expression may be expanded as \begin{equation}\label{chaos}\sum_{k=0}^{2n} n^{-k/4} \sum_{0\leq i_1<...<i_k\leq 2n} \sum_{(x_1,...,x_{k+1}) \in \Bbb Z_{\geq 0}^{k+1}} z^n_0(x_{k+1})\prod_{j=1}^{k+1} \mathfrak{p}_{i_j-i_{j-1}}^{2n-i_{j-1}}(x_{j-1},x_j)\prod_{j=1}^k\hat{\omega}_{i_jx_j},\end{equation} with $x_0:=0$, $i_0:=0$, and $i_{k+1}:=2n$. Recall that $\log(1+u) \sim u-\frac12 u^2$, so by writing 
\begin{align*}z^n_0(x) = e^{\sum_0^x \log(1+n^{-1/4} \omega^n_{i0})} &\approx e^{\sum_0^x\big( n^{-1/4}\omega^n_{i0} -\frac12 n^{-1/2}(\omega^n_{i0})^2\big)} \\&= e^{n^{-1/4}\sum_0^x (\omega^n_{i0}-n^{-1/4}\mu) +n^{-1/2}\mu x -\frac{1}{2n^{1/2}} \sum_0^x (\omega^n_{i0})^2},\end{align*} one may convince herself (using Donsker's principle and the law of large numbers together with the third bullet point preceding Theorem \ref{thm2}) that as $n \to \infty$, $$\big(z_0^n(n^{1/2}X)\big)_{X\geq 0} \stackrel{d}{\longrightarrow} \big(e^{B_X+(\mu-\frac12 \sigma^2)X}\big)_{X\geq 0}$$
for a Brownian motion $B$. Then taking the limit of \eqref{chaos} as $n \to \infty$ by using Theorem \ref{rw}(3) (with some uniformity estimates), one obtains the Wiener-Ito chaos series $$\sum_{k=0}^{\infty} \int_{0\leq t_1 <...<t_k\leq 1} \int_{\Bbb R_+^{k+1}} e^{B_{x_{k+1}}+(\mu-\frac12 \sigma^2)x_{k+1}} \prod_{j=1}^{k+1} \mathscr{P}_{t_j-t_{j-1}}^{1-t_{j-1}} (x_{j-1},x_j) dx_{k+1}\xi(dx_k,dt_k) \cdots \xi(dx_1,dt_1),$$ with the convention $x_0=0$, $t_0=0$, $t_{k+1}=1$, and where the $\mathscr{P}_t^T$ are the conditional heat kernels from the limit in Theorem \ref{rw}(3), and $\xi$ is space-time white noise. But (as we will see in Proposition \ref{eq} below) this chaos series is precisely equal to $$\lim_{X \to 0} \frac{Z_{Dir}(1,X)}{2\Phi(X)-1}=\sqrt{\pi/2}\lim_{X \to 0} \frac{Z_{Dir}(1,X)}{X},$$ where the initial data is $e^{B_X+(\mu-\frac12\sigma^2)X}$ (and $\Phi$ is the cdf of a standard normal). This will complete the argument.

\section{Uniform measures on collections of positive paths}

In this section we will introduce the inhomogeneous heat kernels $\mathfrak{p}_n^N$ associated to random walks conditioned to stay positive. We begin with an elementary discussion of the basic properties of these measures, and later (in subsection 3.1) we will prove technical estimates about these measures which will be very useful in later sections.
\\
\\
For the ensuing discussion, we fix a number $\alpha \in [0,1]$. We define $p_n^{(\alpha)}(x,y)$ to be the probability of a random walk (started from $x \geq 0)$ being at position $y$ at time $n$, with reflection at the origin and probability $\alpha$ of death each time it hits $0$. Specifically, $$p_n^{(\alpha)}(x,y) = Q_x(|S_n| = y, T_{death} > n),$$ where $Q_x$ is the law of a simple symmetric random walk on $\Bbb Z$ started from $x \geq 0$, $S_n$ is the canonical process associated to $Q_x$, and $T_{death}$ is the time of death, given that death occurs with probability $\alpha$ (independently) upon each return to site zero.
\\
\\
Taking a more analytical perspective, we see that for each $x \geq 0$, the function from $\Bbb Z_{\geq 0} \times \Bbb Z_{\geq 0} \to \Bbb R$ given by $(n,y) \mapsto p_n^{(\alpha)}(x,y)$ is the unique solution to the discrete-time, discrete-space heat equation given by \begin{equation}\label{1} \phi(n+1,y)-\phi(n,y) = \frac12 \big( \phi(n,y-1) + \phi(n,y+1) \big), \;\;\;\;\; y \geq 1, n\geq 0,\end{equation} with initial condition which is Dirac mass at $x$: \begin{equation}\label{2} \phi(0,y) = \delta_x(y), \;\;\;\;\; y \geq 0,\end{equation} and with the Dirichlet-type boundary condition \begin{equation}\label{3} \phi(n+1,0) = \alpha \;\phi(n,1), \;\;\;\;\; n\geq 0.\end{equation}
Our first proposition gives an explicit expression for this $\alpha$-boundary heat kernel when $\alpha = 1/2$, in terms of the standard heat kernel on the whole line.

\begin{prop}
For $n \in \Bbb Z_{\geq 0}$ and $x \in \Bbb Z$, let $p_n(x)$ denote the standard heat kernel on $\Bbb Z$ (i.e., the transition function for a discrete-time simple symmetric random walk started from zero). Then one has $$p_n^{(1/2)}(x,y) = p_n(x-y) - p_n(x+y+2).$$

\end{prop}

\begin{proof}One may verify directly that \eqref{1}, \eqref{2}, and \eqref{3} hold, with $\phi$ given by $(n,y) \mapsto p_n^{(1/2)}(x,y)$. The only nontrivial thing to check is \eqref{3} with $\alpha =1/2$, for which one uses the identity $p_{n+1}(x) = \frac12 \big(p_n(x-1) + p_n(x+1) \big).$
\end{proof}

\begin{rk}
More generally one may use an image method to derive the explicit expression $$p_n^{(\alpha)}(x,y) = p_n(x-y) + (2\alpha -1) p_n(x+y) - 4\alpha(1-\alpha) \sum_{k=1}^{\infty} (2\alpha-1)^{k-1} p_n(x+y+2k).$$
The three special cases of this for which the series expansion actually terminates are $\alpha \in \{0,1/2,1\}$ $$p_n^{(1)}(x,y) = p_n(x-y) + p_n(x+y),$$  $$p_n^{(0)}(x,y) = p_n(x-y) - p_n(x+y).$$ $$p_n^{(1/2)}(x,y) = p_n(x-y) - p_n(x+y+2).$$ These correspond to Neumann, Dirichlet, and mixed boundary conditions, respectively. This formula for general $\alpha$ will not be needed, but it would be interesting to extend the results of Theorem \ref{rw} and Theorem \ref{thm2} to the case when the diagonal weights are asymptotically $\alpha \neq 1/2$.
\end{rk}

\begin{defn}\label{fuk}
We define the following quantity for integers $0 \leq n \leq N$ $$\mathfrak{p}_{n}^N(x,y):= p_n^{(1/2)}(x,y) \frac{\psi(y;N-n)}{\psi(x;N)},$$ where $$\psi(x;n) := \sum_{y \geq 0} p_n^{(1/2)}(x,y).$$
\end{defn}

The relevance of these $\mathfrak{p}_n^N$ are as follows: as in Theorem \ref{rw}, let $$\Omega_x^N := \{ (s_0,...,s_N) \in \Bbb Z_{\geq 0}^N: |s_{i+1}-s_i|=1, s_0=x\}.$$ Then denote by $\mathbf{P}_x^N$ the uniform measure on $\Omega^N_x$ and let $S$ denote the coordinate process associated to this measure (e.g., $S$ can be the identity map on $\Omega_x^N$). In more practical terms, $S$ is none other than a simple symmetric random walk conditioned to stay positive.

\begin{prop}\label{pt} $S$ is an inhomogeneous Markov process on $\{0,...,N\}$. In fact, for $0\leq i_1 < ...<i_n \leq N$ one has \begin{align*}\mathbf{P}_x^N( S_{i_1} = s_1 ,..., S_{i_n} = s_n) &= \mathfrak{p}_{i_1}^N(x,s_1) \mathfrak{p}_{i_2-i_1}^{N-i_1}(s_1,s_2) \cdots \mathfrak{p}_{i_n-i_{n-1}}^{N-i_{n-1}}(s_{n-1},s_n) \\ &= p_{i_1}^{(1/2)}(x,s_1) p_{i_2-i_1}^{(1/2)} (s_1,s_2) \cdots p_{i_n-i_{n-1}}^{(1/2)}(s_{n-1},s_n)\frac{\psi(s_n, N-i_n)}{\psi(x,N)}.\end{align*} In particular, for $M<N$ the conditional law of $(S_{M+k})_{k=0}^{N-M}$ given $(S_k)_{k=0}^M$ is distributed according to $\mathbf{P}_{S_M}^{N-M}.$

\end{prop}

Note that this proves Theorem \ref{rw}(1). It also shows that the $\mathfrak{p}_n^N(x,\cdot)$ are probability measures.

\begin{proof} Write $S_{[0,M]}$ for the restriction of $S$ to $\{0,1,...,M-1\}$, and write $S_{[M,N]}$ for the restriction of $S$ to $\{M,...,N\}$ shifted by $M$ places (so $S_{[M,N]}$ is defined on $\{0,...,M-N\}$). For nearest-neighbor paths $s_1$ and $s_2$ of lengths $M$ and $N-M$, respectively, such that $s_1(M) = s_2(0)$ one computes that $$\mathbf{P}_x^n(S_{[M,N]} = s_2 | S_{[0,M]} = s_1) = \frac{\mathbf{P}_x^N(S = s_1*s_2)}{\mathbf{P}_x^{N}(S_{[0,M]} =s_1)} = \frac{\frac{1}{\# \Omega_x^N}}{\frac{\#\{\pi \in \Omega_x^N:\; \pi|_{[0,M]} =s_1\}}{\# \Omega_x^N}} $$$$= \frac{1}{\#\{\pi \in \Omega_x^N: \pi|_{[0,M]} =s_1\}} = \frac{1}{\#\Omega_{s_1(M)}^{M-N}}= \mathbf{P}_{s_1(M)}^{M-N}(S=s_2),$$ where $s_1*s_2$ denotes the concatenation of paths. This immediately implies that given $(S_k)_{k=0}^M$, the law of $(S_{M+k})_{k=0}^{N-M}$ is distributed according to $\mathbf{P}_{S_M}^{N-M}.$ This also implies that $(S_k)_{k=0}^M$ and $(S_{M+k})_{k=0}^{N-M}$ are conditionally independent given $S_M$. Therefore, in order to prove the given formula for transition densities, it suffices to prove the claim for $n=1$; then the claim for general $n$ follows from the conditional independence and induction (recall that $n$ is the number of indices $0\leq i_1 < ...< i_n \leq N$ appearing in the transition formula).
\\
\\
To prove the formula for $n=1$ it suffices by conditional independence to assume that $i_n=N$, just note that $\mathbf{P}_x^N$ is just the probability associated to the killed random walk conditioned to survive, so that $$\mathbf{P}_x^N(S_N = s) = \frac{p_N^{(1/2)}(x,s)}{\sum_{y \geq 0} p_N^{(1/2)}(x,y) } = p_N^{(1/2)}(x,s) \frac{1}{\psi(x,N)},$$ which proves the claim.
\end{proof}

Next we introduce the continuum analogues of the previously introduced measures. We will generally use capital letters to distinguish macroscopic variables from (lowercase) microscopic ones.

\begin{defn}\label{sht} Let $P_t(X):= e^{-X^2/2t}/\sqrt{2\pi t}$ denote the standard heat kernel on the whole line $\Bbb R$. Recall from Section 2 the Dirichlet boundary heat-kernel $$P_t^{Dir}(X,Y): = P_t(X-Y) - P_t(X+Y).$$ We then define the inhomogeneous kernel for $0 \leq t \leq T$ and $X,Y > 0:$  $$\mathscr{P}_t^T(X,Y):= \begin{cases} P_t^{Dir}(X,Y) \frac{2\Phi(Y/\sqrt{T-t}) -1}{2\Phi(X/\sqrt{T})-1} & t<T \\ P_T^{Dir}(X,Y) \frac{1}{2\Phi(X/\sqrt{T})-1}&  t=T\end{cases},$$ where $\Phi(x) = \frac{1}{\sqrt{2\pi}}\int_{-\infty}^{x} e^{-u^2/2}du$ is the cdf of a standard normal. For $X = 0$, one analogously defines the quantity for $Y>0$ and $T \geq t \geq 0$: $$\mathscr{P}_t^T(0,Y) = \begin{cases} Y(T/t^3)^{1/2}e^{-Y^2/2t} \big( 2\Phi(Y/\sqrt{T-t})-1\big) & t<T \\ (Y/T) e^{-Y^2/2T} & t=T \end{cases},$$ which is the limit of the previously defined $\mathscr{P}_t^T(X,Y)$ as $X \to 0$.
\end{defn}

We now discuss the relevance of these kernels as Markov transition densities. Specifically, for $X>0$ define $\mathbf{W}_X^T$ to be the probability measure on $C([0,T],C(\Bbb R_+))$ obtained by conditioning Brownian motion on $[0,T]$ started from $X$ to stay strictly positive until time $T$. We define $B$ to be the canonical process associated to $\mathbf{W}_X^T$. One can also define $\mathbf{W}_0^T$ as the weak limit of the $\mathbf{W}_X^T$ as $X \to 0$. The fact that this limiting measure actually exists is not entirely trivial. It is actually called the Brownian meander, and has been studied extensively in \cite{DIM77, DI77, CM81, Ig74}, and subsequent papers on the subject.

\begin{prop}\label{mp} Fix some $T,X>0$ and let $\mathbf W_{X}^T$ be as defined above, and let $B$ denote the associated canonical process. Consider the kernels $\mathscr{P}_t^T$ defined before. Then for $0 \leq t_1 <...<t_n \leq T$ and $Y_1,...,Y_n>0$, $$\mathbf{W}_X^T(B_{t_1} \in dY_1,\;..., \;B_{t_n} \in dY_n)\;\;\;\;\;\;\;\;\;\;\;\;\;\;\;\;\;\;\;\;\;\;\;\;\;\;\;\;\;\;\;\;\;\;\;\;\;\;\;\;\;\;\;\;\;\;\;\;\;\;\;\;\;\;\;\;\;\;\;\;\;\;\;\;\;\;\;\;\;\;\;\;\;\;\;\;\;\;\;\;\;\;\;\;\;\;\; $$$$= \mathscr{P}_{t_1}^T(X,Y_1) \mathscr{P}_{t_2-t_1}^{T-t_1}(Y_1,Y_2) \cdots \mathscr{P}_{t_n-t_{n-1}}^{T-t_{n-1}}(Y_{n-1},Y_n) \; dY_1 \dots dY_n$$ In particular, if $T<S$ then the conditional law of $(B_{t+S})_{t \in [0,T-S]}$ given $(B_t)_{t \in [0,S]}$ is equal to $\mathbf{W}_{B_S}^{T-S}$. The same statements hold true for $X=0$.
\end{prop}

Before moving onto the proof, we remark that when $X \neq 0$ and $t_n \neq T$, the above formula for transition densities reduces to $$P_{t_1}^{Dir}(X,Y_1) P_{t_2-t_1}^{Dir}(Y_1,Y_2) \cdots P_{t_n-t_{n-1}}^{Dir}(Y_{n-1},Y_n) \frac{2 \Phi(Y_n/\sqrt{T-t_n})-1}{2\Phi(X/\sqrt{T})-1}\; dY_1 \dots dY_n.$$ When $t_n=T$ the numerator $2 \Phi(Y_n/\sqrt{T-t_n})-1$ should be interpreted as just $1$. When $X=0$ this expression becomes $0/0$, and one needs to take the limit, which gives the formula stated in the theorem.

\begin{proof}
Assuming $X>0$ the proof is analogous to that of Proposition \ref{pt}. Basically one first shows that if $S<T$ then the conditional law of $(B_{t+S})_{t \in [0,T-S]}$ given $(B_t)_{t \in [0,S]}$ is equal to $\mathbf{W}_{B_S}^{T-S}$, and furthermore that $(B_{t+S})_{t \in [0,T-S]}$ and $(B_t)_{t \in [0,S]}$ are conditionally independent given $B_S$. This may be proven by a single computation using the basic properties of standard Brownian motion.
\\
\\
As in the proof of Proposition \ref{pt}, this then reduces the claim to proving the formula for $n=1$ and $t_n = T$. In turn, this follows by noticing that $\mathbf{W}_X^T$ is the same as Brownian motion killed at zero, but conditioned to survive. Hence one finds that $$\mathbf{W}_X^T(B_T \in dY) = \frac{P_T^{Dir}(X,Y)dY}{\int_0^{\infty}P_T^{Dir}(X,Z)dZ } = \frac{P_T^{Dir}(X,Y)dY}{2\Phi(X/\sqrt{T})-1},$$which proves the claim.
\end{proof}

This concludes the introductory material on the subject, and we now move on to more technical aspects of the exposition of these random walk measures.

\subsection{Heat kernel estimates}

We now move onto proving various useful estimates for the heat kernels $\mathfrak{p}_n^N$ and $\mathscr{P}_t^T$ defined earlier in this section. Not much motivation will be given here, but the utility of these estimates will become clear later. The methods used in proving these bounds will be brute-force analysis. We will freely use results from the appendix in this section, so the reader may be interested in taking a brief look at so of the main results from there (Propositions \ref{ohgod} and \ref{coup}, and Theorem \ref{conc}).

\begin{prop}\label{macky} There exists constants $C,K>0$ such that for all $x\ge 0$, all $N\ge n\ge 0$, and all $a\ge 0$ one has that $$\sum_{y \geq 0} \mathfrak{p}_n^N(x,y)e^{ay} \leq C e^{ax+Ka^2n}.$$
\end{prop}

\begin{proof} Let us write $$\sum_{y \geq 0} \mathfrak{p}_n^N(x,y)e^{ay} = \mathbf E_x^N[e^{aS_n}] = 1+\int_0^{\infty} ae^{au} \mathbf P_x^N(S_n>u)dy.$$ Now we split the integral as $ \int_0^x + \int_x^{\infty}$. For the integral over $[0,x]$ we use the crude bound $\mathbf P_x^N(S_n>u) \leq 1$. For the integral over $[x,\infty)$, we use the result of Theorem \ref{conc}. This will give $$\mathbf E_x^N[e^{aS_n}] \leq e^{ax} + C\int_x^{\infty} ae^{au} e^{-c(u-x)^2/n}du \leq e^{ax} + C\cdot an^{1/2}e^{ax+\frac{a^2n}{4c}}.$$ Since $an^{1/2} \leq e^{a^2n}$, this gives the result, with $K:=1+\frac{1}{4c}.$
\end{proof}

We remark that $c = 1/32$ from the proof of Theorem \ref{conc}, so we actually obtain $K=9$ in the preceding proposition. Conjecturally, the optimal value of $K$ should be $1/2$, as is the case for simple random walk (as seen from $\cosh(a) \leq e^{\frac12 a^2})$.

\begin{lem}\label{wconc}
Fix $b>0$. There exists $C=C(b)>0$ such that for all $x\ge 0$ and all $N\ge n\ge 0$, one has that $$\mathfrak p_n^N(x,y) \leq  \frac{C}{\sqrt{n+1}}e^{-b|x-y|/\sqrt{n}}.$$
\end{lem}

We remark that this bound is quite strong and quite natural. Many of our estimates could have been derived from this result rather than from the concentration theorem, but only in a much weaker form (because the decay is merely exponential rather than Gaussian).

\begin{proof}
We consider four different cases.
\\
\\
\textit{Case 1.} $x \ge \sqrt{N}$. Then, one has $\frac{\psi(y,N-n)}{\psi(x,N)} \leq \frac1{\psi(x,N)} \le C,$ by Lemma \ref{mass}. Thus it holds that $\mathfrak{p}_n^N(x,y) \leq Cp_n^{(1/2)}(x-y) \leq C(n+1)^{-1/2}e^{-b|x-y|/\sqrt{n}}$. The final inequality comes from the second bound of Lemma \ref{ohgod}.
\\
\\
\textit{Case 2.} $n<N/2$ and $y \leq x$. Then one has \begin{align*}
    \mathfrak{p}_n^N(x,y) & \leq p_n^{(1/2)}(x,y) \bigg[\frac{x+1+\sqrt{N}}{x+1} \bigg] \bigg[ \frac{y+1}{y+1+\sqrt{N-n}}\bigg] \\ &\leq C(n+1)^{-1/2}e^{-b|x-y|/\sqrt{n}} \bigg[\frac{x+1+\sqrt{N}}{x+1} \bigg] \bigg[ \frac{x+1}{x+1+\sqrt{N-n}}\bigg] \\ &= C(n+1)^{-1/2}e^{-b|x-y|/\sqrt{n}} \bigg[\frac{x+1+\sqrt{N}}{x+1+\sqrt{N-n}}\bigg] \\ &\leq C(n+1)^{-1/2}e^{-b|x-y|/\sqrt{n}} \bigg[\frac{N}{N-n}\bigg]^{1/2}
\end{align*} 
We used Lemma \ref{mass} in the first line, we used Lemma \ref{ohgod} and that $y \mapsto \frac{y+1}{y+1+\sqrt{N-n}}$ is monotone increasing in the second line, and we used the fact that $x \mapsto \frac{x+1+\sqrt{N}}{x+1+\sqrt{N-n}}$ is monotone decreasing in the last line. Since $n<N/2$ it follows that $\big[\frac{N}{N-n}\big]^{1/2} \leq 2^{1/2}$ so that term may be absorbed into $C$. Since $n<N/2$ one has $\frac{N}{N-n} \leq 2$, so we are done.
\\
\\
\textit{Case 3.} $n<N/2$ and $y \ge x$. then \begin{align} \notag
    \mathfrak{p}_n^N(x,y) &\leq C p_n^{(1/2)}(x,y) \bigg[\frac{x+1+\sqrt{N}}{x+1} \bigg] \bigg[ \frac{y+1}{y+1+\sqrt{N-n}}\bigg] \\ \notag &\leq  C p_n^{(1/2)}(x,y) \bigg[\frac{x+1+\sqrt{N}}{x+1} \bigg] \bigg[ \frac{y+1}{x+1+\sqrt{N-n}}\bigg] \\ \notag &\leq C\bigg[\frac{N}{N-n}\bigg]^{1/2} p_n^{(1/2)}(x,y) \frac{y+1}{x+1} \\ \notag &=  C\bigg[\frac{N}{N-n}\bigg]^{1/2}p_n^{(1/2)}(x,y) \bigg[\frac{y-x}{x+1} \;\; + \;\; 1\bigg]\\ \label{maff} &\leq C\bigg[\frac{y-x}{n+1} + C(n+1)^{-1/2}\bigg]e^{-b|x-y|/\sqrt{n}}.
\end{align} 
Here we noted $y \ge x$ in the second line, and we used the fact that $x \mapsto \frac{x+1+\sqrt{N}}{x+1+\sqrt{N-n}}$ is monotone decreasing in the third line. In the final line, we used $\big[\frac{N}{N-n}\big]^{1/2}\le 2^{1/2}$ (since $n<N/2$) and we also used both bounds of Lemma \ref{ohgod}. Now, we know that the bound \eqref{maff} is true for all $b$, in particular it is true with $b$ replaced by $b+1$ (after perhaps making the constant bigger). Thus we see that $$\frac{|x-y|}{n+1} e^{-(b+1)|x-y|/\sqrt{n}} \le \frac{1}{\sqrt{n+1}} e^{-b|x-y|/\sqrt{n}}\bigg[\frac{|x-y|}{\sqrt{n}}e^{-|x-y|/\sqrt{n}}\bigg] . $$ $$\leq \frac{1}{\sqrt{n+1}}e^{-b|x-y|/\sqrt{n}} \sup_{u>0} ue^{-u} = \frac{C}{\sqrt{n+1}}e^{-b|x-y|/\sqrt{n}}.$$
\\
\\
\textit{Case 4.} $x \leq \sqrt{N}$ and $n>N/2$. Since $x \leq \sqrt{N} \leq \sqrt{2n}$, we can apply Lemmas \ref{mass} and \ref{ohgod} to see that $$\mathfrak{p}_n^N(x,y) \leq Cp_n^{(1/2)}(x,y) \frac{x+1+\sqrt{N}}{x+1} $$$$\leq C\frac{x+1}{n+1}e^{-b|x-y|/\sqrt{n}}\cdot \frac{2\sqrt{2n}+1}{x+1} \leq C(n+1)^{-1/2}e^{-b|x-y|/\sqrt{n}}.$$This completes the proof of all cases.
\end{proof}

\begin{prop}\label{fetiz}
There exists constants $C,K>0$ such that for all $x\ge 0$, all $N\ge n\ge 0$, all $a\ge 0$, and all $p \ge 1$ one has that $$\sum_{y \geq 0} \mathfrak{p}_n^N(x,y)^p e^{ay} \leq C^p(n+1)^{-(p-1)/2}e^{ax+Ka^2n}.$$
\end{prop}

\begin{proof}Using Proposition \ref{wconc} with $b=0$, one finds that $$\mathfrak p_n^N(x,y)^p =\mathfrak p_n^N(x,y)^{p-1}\mathfrak p_n^N(x,y) \leq \frac{C^{p-1}}{(n+1)^{(p-1)/2}} \mathfrak p_n^N(x,y).$$ Then the claim follows immediately from Proposition \ref{macky}.
\end{proof}

We now bound space-time differences of the heat kernels $\mathfrak p_n^N$.

\begin{lem}\label{dack}
There exists a constant $C>0$ such that for all $x,y,z \geq 0$ one has that $$\big| \mathfrak p_n^N(x,y) - \mathfrak p_n^N(x,z) \big| \leq \frac{C}{n+1} \bigg[ \frac{N+1}{N-n+1}\bigg]^{1/2} |y-z|.$$
\end{lem}
We are not sure if this estimate is sharp. It may or may not be possible to get rid of the term $\big[\frac{N+1}{N-n+1}\big]^{1/2}$ using more clever arguments. However, it will be inconsequential for us because this crude estimate will suffice to prove tightness of the rescaled partition function.

\begin{proof} Without loss of generality, assume $y \ge z$. It suffices to prove the bound in the case $y=z+1$. In the general case, one simply adds the bound one simply adds the bound $y-z$ times. Let us write $$\big| \mathfrak p_n^N(x,z+1) - \mathfrak p_n^N(x,z) \big| = \bigg|\frac{p_n^{(1/2)}(x,z+1) \psi(z+1,N-n) - p_n^{(1/2)}(x,z)\psi(z,N-n)}{\psi(z,N)}\bigg| $$$$\leq |p_n^{(1/2)}(x,z+1) - p_n^{(1/2)}(x,z)|\frac{\psi(z+1,N-n)}{\psi(x,N)}+p_n^{(1/2)}(x,z) \frac{|\psi(z+1,N-n)-\psi(z,N-n)|}{\psi(x,N)}.$$
Let us call the two terms of the last expression as $I_1,I_2$ respectively. As in the proof of Lemma \ref{wconc} we now need to consider several cases.
\\
\\
\textit{Case 1.} $x \geq \sqrt{N}.$ First we bound $I_1$. Since $x \geq \sqrt{N}$ it holds from Lemma \ref{mass} that $$ \frac{\psi(z+1,N-n)}{\psi(x,N)} \leq \frac{1}{C},$$ and furthermore the second bound from Lemma \ref{ohgod} gives $$|p_n^{(1/2)}(x,z+1) - p_n^{(1/2)}(x,z)| \leq \frac{C}{n+1}.$$ The preceding two expressions already prove the desired bound on $I_1$. Next we need to bound $I_2$. As before we know that $ \frac{\psi(z+1,N-n)}{\psi(x,N)} \leq \frac{1}{C}.$ We know from lemma \ref{ohgod} that $p_n^{(1/2)}(x,z) \leq C\frac{z+1}{n+1}.$ Furthermore, we know (see the proof of Lemma \ref{mass}) that $$\psi(z+1,N-n)-\psi(z,N-n)= p_{N-n}(z+1) + p_{N-n}(z+2),$$ where $p_n$ is the standard kernel on the whole line. This can be bounded above by $\frac{C}{\sqrt{N-n}}e^{-(z+1)/\sqrt{N-n}}$ (see, for instance, the proof of Lemma \ref{ohgod}). Thus, we find that $$I_2 \leq \frac{C}{n+1} \frac{z+1}{\sqrt{N-n}}e^{-(z+1)/\sqrt{N-n}} \leq \frac{C}{n+1} \sup_{u \ge 0}ue^{-u}, $$ This proves the desired bound on $I_2$, since the supremum may be absorbed into the constant.
\\
\\
\textit{Case 2.} $x \leq \sqrt{N}$. First let us bound $I_1$. We know from Lemma \ref{ohgod} that \begin{equation}\label{babb}|p_n^{(1/2)}(x,z+1) - p_n^{(1/2)}(x,z)| \leq C\bigg[\frac{1}{n+1}\wedge \frac{x+1}{(n+1)^{3/2}}\bigg] e^{-|z-x|/\sqrt{n}}.\end{equation} Moreover, we know from lemma \ref{mass} that $$\frac{\psi(z,N-n)}{\psi(x,N)} \leq C \frac{z+1}{z+1+\sqrt{N-n}}\frac{x+1+\sqrt{N}}{x+1}.$$ Now we consider two sub-cases, $z \ge x$ and $z \le x$. If $z \le x$, then $$\frac{z+1}{z+1+\sqrt{N-n}}\frac{x+1+\sqrt{N}}{x+1} \le \frac{z+1}{z+1+\sqrt{N-n}}\frac{z+1+\sqrt{N}}{z+1}$$ $$=\frac{z+1+\sqrt{N}}{z+1+\sqrt{N-n}} \leq C\sqrt{\frac{N+1}{N-n+1}}.$$ In the first bound, we used that $x \mapsto \frac{x+1+\sqrt{N}}{x+1}$ is decreasing. In the last bound, we used that $z \mapsto \frac{z+1+\sqrt{N}}{z+1+\sqrt{N-n}}$ is decreasing, and that $\frac{1+\sqrt{N}}{1+\sqrt{N-n}}\leq C\sqrt{\frac{N+1}{N-n+1}}$. With \eqref{babb}, this already gives the required bound on $I_1$ (when $z \ge x$). If $z \ge x$, then one sees that $$\frac{z+1}{z+1+\sqrt{N-n}}\frac{x+1+\sqrt{N}}{x+1} \le \frac{z+1}{z+1+\sqrt{N-n}}\frac{z+1+\sqrt{N}}{x+1} $$$$ = \frac{z+1+\sqrt{N}}{z+1+\sqrt{N-n}} \frac{z+1}{x+1}\leq C\sqrt{\frac{N+1}{N-n+1}} \frac{z+1}{x+1},$$  where the last bound again holds because of the same reasons as when considering the previous sub-case ($z \le x$). Next, we write $\frac{z+1}{x+1} = \frac{z-x}{x+1}+1$, and then we combine the previous bound with \eqref{babb} to see that $$I_1 \leq C \bigg[\frac{N+1}{N-n+1}\bigg]^{1/2} \bigg[ \frac{x+1}{(n+1)^{3/2}} e^{-|z-x|/\sqrt{n}} \frac{|z-x|}{x+1} + \frac{1}{n+1} \bigg]. $$ now we just notice that $$\frac{x+1}{(n+1)^{3/2}} e^{-|z-x|/\sqrt{n}} \frac{|z-x|}{x+1} = \frac{1}{n+1}\bigg[ \frac{|z-x|}{(n+1)^{1/2}} e^{-|z-x|/n^{1/2}}\bigg] \leq \frac{1}{n+1} \sup_{u \ge 0} ue^{-u}.$$ This completes the proof of the required bound on $I_1$, since the supremum is absorbed into the constant.
\\
\\
Now we just need to obtain the required bound on $I_2$ in the case that $x \le \sqrt{N}$. For this, one first notes from Lemma \ref{ohgod} that $p_n^{(1/2)}(x,z) \le \frac{x+1}{n+1}.$ Then one notes from Lemma \ref{mass} that $\frac{1}{\psi(x,N)} \leq \frac{x+1+\sqrt{N}}{x+1} \leq 2\frac{1+\sqrt{N}}{x+1}$ (since $x \le \sqrt{N}).$ Finally, one notes that $$\psi(z+1,N-n)-\psi(z,N-n)= p_{N-n}(z+1) + p_{N-n}(z+2)\leq \frac{C}{\sqrt{N-n+1}}.$$ Hence we find that $$I_2 \leq C \frac{x+1}{n+1}\frac{1}{\sqrt{N-n+1}} \frac{1+\sqrt{N}}{x+1},$$ which is enough since the $x+1$ factors cancel. This completes the proof.
\end{proof}

\begin{prop}\label{spatem} Fix $p \ge 1$.
There exists a constant $C=C(p)>0$ such that for all $x,y \ge 0$, all $N \ge n\ge m \ge 0$, and all $a\ge 0$ one has that
\begin{align}\sum_{z \ge 0} \big( \mathfrak p_n^N(x,z) - \mathfrak p_n^N(y,z) \big)^{2p} e^{az} &\leq Ce^{a(x+y)+Ka^2n} \big(n^{\frac{1}{2}-\frac32p}+a^pn^{\frac12-p}\big) |x-y|^p, \label{spat1}\\ 
\sum_{z \ge 0} \big(\mathfrak p_n^{N-n+m}(x,z) - \mathfrak p_m^N(x,z) \big)^{2p}e^{az} &\le Ce^{a(x+y)+Ka^2n} \big(m^{\frac{1}{2}-\frac32p}+a^pm^{\frac12-p}\big)|n-m|^{p/2}.\label{tem1}
\end{align}
In the spatial bound \eqref{spat1}, the constant $C$ grows at worst exponentially in $p$.
\end{prop}

We remark that in the special case that $p=2$ and $a \leq Cn^{-1/2}$, one has that $n^{\frac{1}{2}-\frac32p}+a^pn^{\frac12-p} \le Cn^{-1}$ and similarly for $m$. This is the case in which this bound will be most useful.

\begin{proof} We first start out by proving an auxiliary bound which will be very useful: \begin{equation}\label{use}\;\;\;\sum_{z \ge 0} \big( \mathfrak p_n^N(x,z) - \mathfrak p_n^N(y,z) \big)^2 e^{az} \leq Ce^{a(x+y)+Ka^2n} \big(n^{-1}+an^{-1/2}\big) \bigg[\frac{N+1}{N-n+1}\bigg]^{1/2}|x-y|,
\end{equation}
Let us prove this. The Coupling Lemma (\ref{coup}) and the preceding Lemma will be key here. First, by the Coupling Lemma, we know that $\mathbf P_x^N$ and $\mathbf P_y^N$ may be coupled in such a way so that the respective coordinate processes (call them $(S^x_n)_{n=0}^N$ and $(S^y_n)_{n=0}^N$) are never distance more than distance $|y-x|$ apart (i.e., $\sup_{n\le N} |S^x_n-S^y_n| \le |x-y|$ a.s.). Now, we may write $$\sum_{z \ge 0} \big( \mathfrak p_n^N(x,z) - \mathfrak p_n^N(y,z) \big)^2 e^{az} $$ $$\;\;\;\;\;\;\;\;\;\;\;\;= \mathbf E_x^N[(\mathfrak p_n^N(x,S_n)-\mathfrak p_n^N(y,S_n))e^{aS_n}]-\mathbf E_y^N[(\mathfrak p_n^N(x,S_n)-\mathfrak p_n^N(y,S_n))e^{aS_n}] $$ $$\;\;\;\;\;\;\;\;=E[(\mathfrak p_n^N(x,S_n^x)-\mathfrak p_n^N(y,S_n^x))e^{aS_n^x}]-E[(\mathfrak p_n^N(x,S_n^y)-\mathfrak p_n^N(y,S_n^y))e^{aS_n^y}]$$ 
\begin{align*}=E[&(\mathfrak p_n^N(x,S_n^x) -\mathfrak p_n^N(x,S_n^y)) e^{aS_n^x}] + E[ \mathfrak p_n^N(x,S_n^y)(e^{aS_n^x}-e^{aS_n^y})] \\ &+ E[(\mathfrak p_n^N(y,S_n^y)-\mathfrak p_n^N(y,S_n^x))e^{aS_n^y} ] + E[ \mathfrak p_n^N(y,S_n^x) (e^{aS_n^y}-e^{aS_n^x})].
\end{align*}
Let us call the terms in the last expression as $J_1,J_2,J_3,J_4$, respectively. Since $J_1$ and $J_3$ occupy symmetric roles, it suffices to bound $J_1$ and then the analogous bound for $J_3$ automatically follows. The same thing happens for $J_2$ and $J_4$. With this understanding, we will only prove the desired bound for $J_1$ and $J_2$.
\\
\\
Let us start by bounding $J_1$. By Lemma \ref{dack}, we see that \begin{align*}|\mathfrak p_n^N(x,S_n^x) -\mathfrak p_n^N(x,S_n^y)| &\leq \frac{C}{n+1} \bigg[ \frac{N+1}{N-n+1} \bigg]^{1/2} |S_n^x-S_n^y| \\ &\leq \frac{C}{n+1} \bigg[ \frac{N+1}{N-n+1} \bigg]^{1/2} |x-y|.\end{align*}
Applying the definition of $J_1$ and then Proposition \ref{macky}, we therefore obtain that $$J_1 \leq \frac{C}{n+1} \bigg[ \frac{N+1}{N-n+1} \bigg]^{1/2} |x-y| E[e^{aS_n^x}] \leq \frac{C}{n+1} \bigg[ \frac{N+1}{N-n+1} \bigg]^{1/2} |x-y|e^{ax+Ka^2n}.$$ 
This already gives the desired bound on $J_1$. As discussed, the analogous bound on $J_3$ is obtained in an identical fashion, but one will get $e^{ay}$ instead of $e^{ax}.$ The final bound on $J_1+J_3$ is then obtained by noting that $e^{ax}+e^{ay} \leq 2e^{a(x+y)}.$
\\
\\
Now we bound $J_2$. First note that $|e^u-e^v| \leq |u-v|e^{u\vee v}$ for all $u,v \in \Bbb R.$ Thus $|e^{aS_n^y} - e^{aS_n^x}| \leq a|S_n^y-S_n^x| e^{a(S_n^y\vee S_n^x)} \leq a|y-x|e^{a (S_n^y+ S_n^x)}.$ By Cauchy-Schwarz, we in turn bound $E[e^{a(S_n^y+S_n^x)}] \leq \mathbf E_x^N[e^{2aS_n}]^{1/2}\mathbf E_y^N[e^{2aS_n}]^{1/2} \leq Ce^{a(x+y)+Ka^2n},$ by Proposition \ref{fetiz}. Now, we also know from Lemma \ref{wconc} that $\mathfrak p_n^N(x,S_n^y) \leq C(n+1)^{-1/2}$. Using these facts, we find that $$J_2 \leq C a|y-x|E[\mathfrak p_n^N(x,S_n^y)e^{a(S_n^y+S_n^x)}] \leq Can^{-1/2}|x-y|e^{a(y+x)+Ka^2n}. $$ 
Already this proves the required bound on $J_2$. The analogous bound on $J_4$ follows immediately. This completes the proof of \eqref{use}.
\\
\\
Now let us prove the spatial estimate \eqref{spat1}. For $m \leq n$, we use the semigroup property to write $\mathfrak p_n^N(x,z)=\sum_{y \ge 0} \mathfrak p_m^N(x,y)\mathfrak p_{n-m}^{N-m}(y,z)$ and then using Jensen's inequality, we find that
\begin{align*}
    \big( \mathfrak p_n^N(x,z) - \mathfrak p_n^N( y,z)\big)^{2p} &= \bigg( \sum_{w \geq 0} \big( \mathfrak p_m^N(x,w) - \mathfrak p_m^N(y,w)\big) \mathfrak p_{n-m}^{N-m}(w,z) \bigg)^{2p} \\ &\leq \bigg(\sum_{w \geq 0} \big(\mathfrak p_m^N(x,w) - \mathfrak p_m^N(y,w) \big)^2 \mathfrak p_{n-m}^{N-m}(w,z) \bigg)^p
\end{align*}
Denoting by $I$ the left-hand side of \eqref{spat1}, we then find by Minkowski's inequality that 
\begin{align*}
    I^{1/p} &\leq \Bigg( \sum_{z \ge 0} \bigg[ \sum_{w \ge 0} \big( \mathfrak p_m^N(x,w) -\mathfrak p_m^N(y,w)\big)^2 \mathfrak p_{n-m}^{N-m}(w,z)e^{az/p}\bigg]^p \Bigg)^{1/p} \\ &\stackrel{\text{Minkowski}}{\leq} \sum_{w \ge 0} \bigg[ \sum_{z \ge 0} \big( \mathfrak p_m^N(x,w) - \mathfrak p_m^N(y,w) \big)^{2p} \mathfrak p_{n-m}^{N-m}(w,z)^p e^{az} \bigg]^{1/p} \\ &= \sum_{w \ge 0} \big(\mathfrak p_m^N(x,w) - \mathfrak p_m^N(y,w) \big)^2 \bigg[ \sum_{ z \geq 0} \mathfrak p_{n-m}^{N-m}(w,z)^pe^{az} \bigg]^{1/p} \\& \stackrel{Prop. \ref{fetiz}}{\le} C\sum_{w \ge 0} \big( \mathfrak p_m^N(x,w) - \mathfrak p_m^N(y,w)\big)^2 (n-m)^{-\frac{1-p}{2p}} e^{(aw+Ka^2(n-m))/p} \\ &\stackrel{\eqref{use}}{\le} C(n-m)^{-\frac{1-p}{2p}} \big(m^{-1}+am^{-1/2}\big) \bigg[\frac{N+1}{N-m+1}\bigg]^{1/2}|x-y| e^{\big(a(x+y)+Ka^2n\big)/p}.
\end{align*}
Setting $m:=n/2$ then gives \eqref{spat1}, because $\big[\frac{N+1}{N-\frac12n+1}\big]^{1/2} \leq \big[ \frac{N+1}{\frac12N+1}\big]^{1/2} \leq 2^{1/2}.$ Note that the constant $C$ does not depend on $p$, which also proves the final sentence given in the theorem statement (after perhaps noting that $\big(n^{-1}+an^{-1/2}\big)^p \leq 2^p(n^{-p} + a^pn^{-p/2}$).
\\
\\
We now move onto the temporal estimate \eqref{tem1}. The main idea is to use Jensen's inequality together with the spatial estimate. Specifically, we start off by writing 
\begin{align*}
    \big( \mathfrak p_m^{N-n+m}(x,z) - \mathfrak p_n^N(x,z) \big)^{2p} &= \bigg(\mathfrak p_n^{N-n+m}(x,z) - \sum_{y \ge 0} \mathfrak p_{n-m}^N(x,y) \mathfrak p_m^{N-n+m}(y,z) \bigg)^{2p} \\ &= \bigg( \sum_{y \ge 0} \mathfrak p_{n-m}^N(x,y) \big( \mathfrak p_m^{N-n+m}(x,z) - \mathfrak p_m^{N-n+m}(y,z) \big) \bigg)^{2p} \\ & \stackrel{\text{Jensen}}{\leq} \sum_{y \ge 0} \mathfrak p_{n-m}^N(x,y) \big( \mathfrak p_m^{N-n+m}(x,z) - \mathfrak p_m^{N-n+M}(y,z) \big)^{2p}
\end{align*}
Next, we multiply by $e^{az}$, then sum over $z$, and interchange the sum over $z$ with the sum over $y$. Letting $J$ denote the left-hand side of \eqref{tem1}, this gives
\begin{align*}
    J &\le \sum_{y \ge 0} \mathfrak p_{n-m}^N(x,y) \sum_{z \ge 0} \big( \mathfrak p_m^{N-n+m}(x,z) - \mathfrak p_m^{N-n+m}(y,z) \big)^{2p} e^{az} \\ &\le C^p \sum_{y \ge 0} \mathfrak p_{n-m}^N(x,y) e^{a(x+y)+Ka^2m} \big(m^{\frac{1}{2}-\frac32p}+a^pm^{\frac12-p}\big) |y-x|^p \\ &= C^pe^{ax+Ka^2m}\big(m^{\frac{1}{2}-\frac32p}+a^pm^{\frac12-p}\big) \mathbf E_x^N[|S_{n-m}-x|^pe^{aS_{n-m}}]. 
\end{align*}
All that is left to do is to show that one has $\mathbf E_x^N[|S_{n-m}-x|^pe^{aS_{n-m}}] \leq Ce^{ax} |n-m|^{p/2}.$ This is an easy consequence of the concentration theorem. Indeed, for any $k \leq N$ one may write 
$$\mathbf E_x^N[|S_k-x|^pe^{aS_k}] \leq \mathbf E_x^N[|S_k-x|^{2p}]^{1/2} \mathbf E_x^N[e^{2aS_k}]^{1/2},$$ and then the claim follows immediately from Propositions \ref{macky} and Corollary \ref{pmoments}.
\end{proof}

\begin{cor}[Spatial/Temporal Estimates]
There exists $C>0$ such that for all $x,y,z \ge 0$ and all $N \ge n \ge m \ge 0$ one has that \begin{align}\label{mam}\big|\mathfrak p_n^N(x,z) - \mathfrak p_n^N(y,z) \big| &\leq C(n+1)^{-3/4} |x-y|^{1/2},\\
\label{mom}\big|\mathfrak p_m^{N-n+m}(x,z) - \mathfrak p_n^N(x,z)\big| &\leq C(m+1)^{-3/4}|n-m|^{1/4}.
\end{align}
\end{cor}

These pointwise bounds are quite useful, in the sense that that the exponents (despite not being sharp) are ones which actually give meaningful information. However, we will not actually need this estimate, but it could potentially be useful if one wanted to develop the results of Section 4 with Dirac initial data (for instance).

\begin{proof}
Let $I(2p)$ denote the left-hand side of \eqref{spat1} with $a=0$. Then $I(2p)^{\frac1{2p}}$ is bounded above by $C (n+1)^{\frac1{4p}-\frac34} |x-y|^{1/2}$, where the constant $C$ is independent of $p$ (by the final sentence in the statement of Proposition \ref{spatem}). Letting $p\to \infty$ already proves the first bound (since $\ell^p$ norms converge to the $\ell^{\infty}$ norm).
\\
\\
For the second bound, we cannot do the same thing, since the constant in \eqref{tem1} could (in principle) have worse-than-exponential dependence on $p$. However, we can use the semigroup property to write $$\big|\mathfrak p_m^{N-n+m}(x,z) - \mathfrak p_n^N(x,z)\big| \leq \sum_{y \ge 0} \mathfrak p_{n-m}^N(x,y) \big| \mathfrak p_m^{N-n+m}(x,z) - \mathfrak p_m^{N-n+m}(y,z)\big|,$$ and then one may use the spatial bound \eqref{mam} with Corollary \ref{pmoments} to obtain the result.
\end{proof}

Next we prove a strong convergence result for the discrete kernels $\mathfrak{p}_n^N$ to the continuous ones $\mathscr{P}_t^T$, which will be quite useful for the polymer convergence result in Section 5. In the case of Brownian meander at terminal time ($X=0$ and $t=T)$, it is weaker than the local convergence result of \cite{Car05}, but we actually need it for all time so we give an original and detailed proof.

\begin{prop}\label{gence}
Fix $\tau \geq 0$. Then for $n\geq 0$, define $$\mathscr{P}_n(t,T;X,Y) := (n/2)^{1/2} \mathfrak{p}_{2\lfloor tn \rfloor}^{2\lfloor Tn\rfloor}(2\lfloor n^{1/2}X/\sqrt 2 \rfloor,2\lfloor n^{1/2}Y/\sqrt 2\rfloor).$$ Then for each fixed $X,T,t \geq 0$, the map $Y \mapsto \mathscr{P}_n(t,T;X,Y)$ converges pointwise and in $L^p(\Bbb R_+,e^{aY}dY)$ to $\mathscr{P}_t^T(X,Y)$ for all $p \ge 1$ and $a \ge 0$ (as $n \to \infty$).
\\
\\
Furthermore, for all $X,T \ge 0$, the map $(t,Y) \mapsto \mathscr{P}_n(t,T;X,Y)$ converges pointwise and in $L^p(dt \otimes e^{aY}dY)$ to $\mathscr{P}_t^T(X,Y)$ for all $p \in [1,3)$ and $a \ge 0$ (as $n \to \infty$).
\end{prop}
From now on, we will abbreviate quantities such as $\mathfrak{p}_{2\lfloor tn \rfloor}^{2\lfloor Tn\rfloor}(2\lfloor n^{1/2}X/\sqrt 2 \rfloor,2\lfloor n^{1/2}Y/\sqrt 2\rfloor)$ by just writing $\mathfrak p_{2nt}^{2nT}((2n)^{1/2}X,(2n)^{1/2}Y)$ instead. This abuse of notation will hopefully not cause any confusion, but in reality one should keep in mind that all quantities are only defined with \textit{even} integers. The reason for this is the periodicity of the simple random walk: $ \mathfrak p_n^N(x,y)$ vanishes if $n$ and $x-y$ have different parity. If it were not for this parity consideration, we could actually take a limit of the simpler quantity $n^{1/2} \mathfrak p_{\lfloor nt \rfloor}^{\lfloor nT \rfloor}(\lfloor n^{1/2}X \rfloor, \lfloor n^{1/2} Y \rfloor).$

\begin{proof}
First, let us prove pointwise convergence. Letting $p_n$ denote the standard heat kernel on all $\Bbb Z$, we recall that $$p_n^{(1/2)}(x,y) = p_n(x-y) - p_n(x+y+2).$$ $$\psi(x,n) = p_n(0)+p_n(x+1) + 2\sum_{1\le y \le x} p_n(y) = \sum_{-x\leq y \leq x+1}p_n(y).$$ Let $F_n$ denote the cdf associated to $p_n$, so that $\psi(x,n) = F_n(x+1)-F_n(-x) = F_n(x)+F_n(x+1)-1.$ By uniformity of convergence of cdf's in the central limit theorem we know that $F_n(n^{1/2}x)$ converges uniformly (on $\Bbb R$) to $\Phi(x)$, where $\Phi$ is the cdf of a standard normal. From this is is clear that $\psi(n,n^{1/2}x) = F_n(n^{1/2}x)+F_n(n^{1/2}x+1)-1$ converges uniformly to $2 \Phi(x) - 1$ (because $\Phi$ has no atoms). From this, one deduces that $\psi(2nT,(2n)^{1/2}X)=\psi(2nT,(2nT)^{1/2}X/\sqrt{T})$ converges to $2\Phi(X/\sqrt{T})-1$.
\\
\\
Now, from the local central limit theorem, it is immediate that $p_{2n}(2x) = \frac1{\sqrt{\pi n}} e^{-x^2/n}+o(n^{-1/2})$ where the error is uniform in $x$. Next, we notice that $|p_n(x+y+2) - p_n(x+y)| \leq Cn^{-1}$, for a constant $C$ independent of $x,y$ (by Lemma \ref{ohgod} with $b=0$). Consequently, $(n/2)^{1/2}|p_{2nt}((2n)^{1/2}(X+Y)) - p_{2nt}((2n)^{1/2}(X+Y)+2)| \to 0$ as $n \to \infty$. Then it follows immediately that $$(n/2)^{1/2}p_{2nt}^{(1/2)}((2n)^{1/2}X,(2n)^{1/2}Y) \to \frac{1}{\sqrt{2\pi t}}(e^{-(X-Y)^2/2t}-e^{-(X+Y)^2/2t}) = P_t^{Dir}(X,Y).$$
Combining the results of the last two paragraphs, we find that if $X \neq 0$, then $$(n/2)^{1/2}\mathfrak p_{2nt}^{2nT}((2n)^{1/2}X,(2n)^{1/2}Y) \to P_t^{Dir(X,Y)} \frac{2\Phi(Y/\sqrt{T-t})-1}{2\Phi(X/\sqrt{T})-1} = \mathscr P_t^T(X,Y).$$ 
This proves pointwise convergence for $X \ne 0$. When $X=0$, we need to give a separate proof. For this we again invoke the local central limit theorem. Specifically, we have
\begin{align*}(n/2)^{1/2}p_{2nt}^{(1/2)}(0,(2n)^{1/2}Y) &= (n/2)^{1/2}(p_{2nt}((2n)^{1/2}Y)-p_{2nt}((2n)^{1/2}Y+2)) \\ &= (n/2)^{1/2} \frac{1}{\sqrt{\pi n t}}( e^{-Y^2/2t}- e^{-(Y+(n/2)^{-1/2})^2/2t}+o(n^{-1/2})) \\ &= \frac{1}{\sqrt{\pi n t}} \big( -\partial_Y e^{-Y^2/2t} + o(1)) \\ &= \frac{1}{\sqrt{\pi n t}}(\frac{Y}{t}e^{-Y^2/2t}+o(1)). 
\end{align*} 
On the other hand, we also have that $$\frac{1}{\psi(0,2nT)} = \frac{1}{p_{2nT}(0)+p_{2nT}(1)} = \sqrt{\pi n T}(1 +o(1)),$$ because $p_{2nT}(1)=0$. In the end we find that $$ \frac{(n/2)^{1/2}p_{2nt}^{(1/2)}(0,(2n)^{1/2}Y)}{\psi(0,2nT)}\to \sqrt{\pi T} \frac{1}{\sqrt{\pi t}} \cdot \frac{Y}{t}e^{-Y^2/2t} = Y(T/t^3)^{1/2} e^{-Y^2/2t}$$ 
Multiplying by $\psi(n^{1/2}Y,n(T-t)),$ which (as noted earlier in the proof) converges to $2\Phi(Y/\sqrt{T-t})-1$ (interpreted as $1$ if $T=t$), we get that $$(n/2)^{1/2} \mathfrak p_{2nt}^{2nT}(0,(2n)^{1/2}Y) \to Y (T/t^3)^{1/2} e^{-Y^2/2t} (2\Phi(Y/\sqrt{T-t})-1) = \mathscr P_t^T(0,Y). $$ 
This completes the proof of pointwise convergence. Now we will fix $t,T,X$, and we will address convergence in $L^p(\Bbb R_+, e^{aY}dY).$ The main idea is simply to use dominated convergence in conjunction with Lemma \ref{wconc}. Specifically, that lemma (applied with $b=1$) tells us that \begin{equation}\label{butk}\mathscr P_n(t,T;X,Y) \leq C t^{-1/2} e^{-t^{-1/2}|X-Y|}.\end{equation} 
Here $C$ is a constant independent of $X,Y,T,t$. Letting $p \ge 1$, it is then clear that for fixed $X,T,t$, the sequence of maps $$Y \mapsto \mathscr P_n(t,T;X,Y)^p e^{aY}$$ is dominated (uniformly in $n$) by a function which is integrable on $\Bbb R_+$. This is enough to imply uniform integrability of this sequence of functions of $Y$, which is in turn enough to guarantee [Zit, Theorem 11.5] that $$ \int_{\Bbb R_+} |\mathscr P(t,T;X,Y) - \mathscr P_t^T(X,Y)|^pe^{aY}dY \to 0.$$ Similarly, one uses $\eqref{butk}$ in conjunction with the dominated convergence theorem to obtain convergence in $L^p(\Bbb R_+\times \Bbb R_+, dt \otimes e^{aY}dY)$ of $(Y,t) \mapsto \mathscr P_{n}(t,T;X,Y)$. This argument only works for $p\in [1,3)$, since the singularity of $\int_{\Bbb R_+} t^{-p/2} e^{-pt^{-1/2}|X-Y|}dY \sim t^{-(p-1)/2}$ fails to be absolutely integrable near $t=0$, if $p \ge 3$.
\end{proof}

\begin{prop}Let $a,\tau >0$ and let $\mathscr{P}_t^T$ be the kernels from Definition \ref{sht}. Then there exists a constant $C=C(\tau, a)$ such that for all $X,Y \geq 0$ and $s\le t \le T \leq \tau$ one has the following \begin{align} \label{easy}\int_{\Bbb R_+} \mathscr{P}_t^T(X,Z) e^{aZ} dZ &\leq C e^{aX}, \\ \label{dik}\int_{\Bbb R_+} \mathscr{P}_t^T(X,Z)^2 e^{aZ} dZ &\leq Ct^{-1/2} e^{aX}, \\ \label{spat} \int_{\Bbb R_+} \big( \mathscr P_t^T(X,Z) - \mathscr P_t^T(Y,Z) \big)^2 e^{aZ} dZ &\leq Ct^{-1/2} e^{a(X+Y)} |X-Y|, \\ \label{tem}
\int_{\Bbb R_+} \big( \mathscr P_s^{T-t+s}(X,Z) - \mathscr P_t^T(X,Z)\big)^2e^{aZ} dZ &\leq Cs^{-1/2} e^{2aX} |t-s|^{1/2}\end{align}
\end{prop}

We remark that these bounds will be the key behind the proofs of Section 4 below.

\begin{proof} The claims follow from the $L^1$ and $L^2$ convergence in Theorem \ref{gence}. More specifically, \eqref{easy} follows Proposition \ref{macky} and convergence in $L^1(\Bbb R_+,e^{aY}dY)$. Next, \eqref{dik} follows from Proposition \ref{fetiz} and convergence in $L^2(e^{aY}dY).$ Similarly, \eqref{spat} and \eqref{tem} follow immediately from Proposition \ref{spatem} and convergence in $L^2(e^{aY}dY).$
\end{proof}

\section{Existence of the right-derivative of Dirichlet-SHE}

In this section we prove existence of the mild solution $Z_{Dir}$ of Dirichlet-boundary \eqref{SHE}, and we also prove existence of the limit $\lim_{X \to 0} \frac{Z_{Dir}(0,X)}{X}$ (started from any reasonable initial data). These will both be done in one single step, by showing that for $X,T\geq 0$ the chaos series $$\sum_{k=0}^{\infty} \int_{0\leq t_1 <...<t_k\leq T} \int_{\Bbb R_+^{k+1}} f(X_{k+1}) \prod_{j=1}^{k+1} \mathscr{P}_{t_j-t_{j-1}}^{T-t_{j-1}} (X_{j-1},X_j) dX_{k+1}\xi(dX_{k},dt_k) \cdots \xi(dX_1,dt_1),$$ converges (uniformly over compact subsets in $\Bbb R_+ \times \Bbb R_+$) , with $t_0:=0$, and $f$ is some random initial data with subexponential growth at infinity. Then we will show that when $X,T>0$ this chaos series is nothing but $$\frac{Z_{Dir}(T,X)}{2\Phi(X/\sqrt{T})-1},$$ where $\Phi$ is the cdf of a standard normal. This would simultaneously prove existence of $Z_{Dir}$ and the desired limit. This is because we know the above chaos series extends continuously to $X=0$, which means $\lim_{X\to 0} \frac{Z_{Dir}(X,T)}{2\Phi(X/\sqrt{T})-1}$ exists, which is equivalent to showing that $\lim_{X \to 0} \frac{Z_{Dir}(T,X)}{X}$ exists.
\\
\\
With this motivation, we move onto the main results of this section. Given some (possibly random) initial data $f:\Bbb R_+ \to \Bbb R_+$, consider the following Duhamel-form SPDE: \begin{equation}\label{spde}\mathscr{Z}(T,X) = \int_{\Bbb R^+} \mathscr{P}_T^T(X,Y)f(Y)dY + \int_0^T \int_{\Bbb R^+} \mathscr{P}^T_{T-S}(X,Y) \mathscr{Z}(S,Y) \xi(dYdS), \end{equation} where $\xi$ is space-time white noise (so the above should be interpreted as an Itô integral), and $\mathscr P$ was defined in Section 3. Since $\mathscr{Z}$ appears on both sides of this relation, it is not clear that a solution would even exist. Thus we have the following.

\begin{thm}\label{exist} Fix $a, \tau >0$ and suppose that we have some (random) function-valued initial data $f$ satisfying $$\sup_{X\ge 0} e^{-aX}\Bbb E[ f(X)^2] <\infty.$$
Then, a unique solution to the SPDE \eqref{spde} with initial data $f$ exists in the class of space-time functions $\mathscr{Z}(T,X)$ which satisfy $$\sup_{\substack{X\ge 0 \\ T \in [0,\tau]} } e^{-aX} \Bbb E[ \mathscr{Z}(T,X)^2] < \infty.$$ Furthermore, the solution $\mathscr Z$ may be constructed in such a way so that its law is supported on the space of functions which are Holder-continuous of exponent $1/2-\epsilon$ in the $X$ variable and $1/4-\epsilon$ in the time variable, on any compact subset of $\Bbb R_+\times \Bbb R_+$ for any $\epsilon>0$.
\end{thm}

\begin{proof}This is adapted from the proofs given in [Par18, Section 4]. Informally, one argues as follows: define the following sequence of iterates: $$u_0(T,X) = \int_{\Bbb R_+} \mathscr{P}_T^T(X,Y)f(Y)dY,$$ $$u_{n+1}(T,X) = \int_0^T \int_{\Bbb R^+} \mathscr{P}_{T-S}^T(X,Y) u_n(S,Y) \xi(dYdS).$$ In other words, $u_n$ is just the $n^{th}$ term of a chaos series given by the expansion of \eqref{spde}. Thus it is clear that the desired solution to \eqref{spde} should be given by $\sum_{n \geq 0} u_n$. Hence, in order to formalize these ideas, we will show that the series $\sum u_n$ converges in the appropriate Banach space of random space-time functions.
\\
\\
To this end, let us define a Banach space $\mathcal B$ of $C(\Bbb R_+)$-valued processes $u = (u(T,\cdot))_{T \in [0,\tau]}$ which are adapted to the natural filtration of $\xi$, with norm given by $$\|u \|_{\mathcal B}^2 := \sup_{\substack{X\ge 0 \\ T \in [0,\tau]}} e^{-aX} \Bbb E[u(T,X)^2].$$ 
Then define a sequence of functions $F_n:[0,\tau] \to \Bbb R$ for $n \geq 0$ by $$F_n(T) := \sup_{\substack{X\ge 0\\ S \in [0,T]}} e^{-aX}\Bbb E[ u_n(S,X)^2],$$ where $u_n$ are the iterates defined above. By Ito isometry, it is clear that \begin{align}\label{muck}
\Bbb E[ u_{n+1}(T,X)^2] &= \int_0^T \int_{\Bbb R_+} \mathscr{P}_{T-S}^T(X,Y)^2 \Bbb E[ u_n(S,Y)^2] dYdS \notag \\ &\leq \int_0^T\bigg[ \int_{\Bbb R_+} \mathscr{P}^T_{T-S}(X,Y)^2 e^{aY} dY\bigg] F_n(S)dS.\end{align} Now by \eqref{dik} we have that \begin{equation}\label{turd}\int_{\Bbb R_+} \mathscr{P}_{T-S}^T(X,Y)^2 e^{aY}dY \leq C(T-S)^{-1/2}  e^{aX},\;\;\;\;\;\;\;\forall T\in[0,\tau],X\ge 0,\end{equation} where $C$ may depend on $a$ and $\tau$. Furthermore, one notes that the $F_n$ are increasing functions of $T$, and therefore $T \mapsto \int_0^T (T-S)^{-1/2}F_n(S)dS$ is also increasing (which may be verified by making the substitution $S=TU$). Combining this fact with \eqref{muck} and \eqref{turd}, one obtains 
\begin{equation}\label{btch}F_{n+1}(T) \leq C \int_0^T (T-S)^{-1/2} F_n(S) dS,\end{equation} where $C$ does not depend on $n$. Now, we claim that $F_0(T) \leq C$ (with $C=C(a,\tau)$). Indeed, by Jensen's inequality and Fubini's theorem, one has
$$\Bbb E[u_0(T,X)^2] = \Bbb E \bigg[ \bigg( \int_{\Bbb R_+} \mathscr{P}_T^T(X,Y) f(Y) dY \bigg)^2 \bigg] \leq \int_{\Bbb R_+} \mathscr{P}_T^T(X,Y) \Bbb E[f(Y)^2] dY \leq Ce^{aX},$$
where in the last inequality we used \eqref{easy} together with the assumption that $\Bbb E[f(X)^2] \leq Ce^{aX}$. This proves that $F_0 \leq C$, which means that one may iterate \eqref{btch} to obtain $$F_n(S) \lesssim C^n T^{n/2} /(n/2)!,$$ which implies that $\sum_n \|u_n\|_{\mathcal{B}}<\infty.$ This completes the proof of existence.
\\
\\
The proof of uniqueness is essentially the same. Indeed, if $\mathscr{Z}$ and $\mathscr{Z'}$ were two solutions in $\mathcal{B}$ which are started from the same initial data $f$, then an application of Ito's isometry reveals that $$\Bbb E\big[ (\mathscr{Z}(T,X) - \mathscr{Z'}(T,X))^2 ] = \int_0^T \int_{\Bbb R_+} \mathscr{P}_{T-S}^T(X,Y)^2 \Bbb E[(\mathscr{Z}(S,Y) - \mathscr{Z'}(S,Y))^2 ] dSdY. $$ Then one iterates as above and one may obtain that the left-hand side is bounded above (uniformly in $T,X$) by $C^nT^{n/2}/(n/2)!$, and by letting $n \to \infty$ this tends to zero.
\\
\\
Now we address the Holder regularity. Let $u_n$ be the iterates defined above. We know that $u_0$ is a smooth function of $(X,T) \in [0,\infty) \times (0,\infty)$ because it is the solution to the deterministic (i.e., noiseless) version of SPDE \eqref{spde} which is just an inhomogeneous heat equation (e.g., one may simply differentiate $u_0$ under the integral sign). Thus, it suffices to prove that the function $\mathscr Z_0:= \mathscr Z-u_0=\sum_{n \geq 1} u_n$ has the required Holder regularity, so this is what we will do.
\\
\\
For the spatial regularity, one computes that \begin{align*}\mathbb E[ (u_{n+1}(T,X)-&u_{n+1}(T,Y))^2 ] = \int_0^T \int_{\Bbb R_+} \big( \mathscr P_{T-S}^T(X,Z) - \mathscr P_{T}^{T-S}(Y,Z) \big)^2 \mathbb E[u_n(S,Z)^2]dZ dS \\ & \leq \int_0^T\bigg[ \int_{\Bbb R_+} \big(\mathscr P_{T-S}^T(X,Z) - \mathscr P_{T-S}^{T}(Y,Z) \big)^2 e^{aZ} dZ\bigg] F_n(S)dS \\ & \leq C\int_0^T (T-S)^{-1/2} (T/S)^{1/2} |X-Y|e^{a(X+Y)} F_n(S)dS \\ &\leq Ce^{a(X+Y)}|X-Y|\int_0^T (T-S)^{-1/2} \frac{C^nS^{n/2}}{(n/2)!} dS \\ &\lesssim C^{n+1}e^{a(X+Y)} |X-Y| T^{(n+1)/2}/(n/2)!,
\end{align*}
where we made a substitution $S=TU$ in the final inequality, and we applied estimate \eqref{spat} in the third line. Using hypercontractivity of the Ornstein-Uhlenbeck semigroup associated to the Gaussian noise $\xi$, we can actually bound the $p^{th}$ moments of elements of the homogeneous Wiener chaoses in terms of their second moments. Specifically, if $p \ge 2$ then [Hai16, Equation (7.2)] says that: \begin{align*}\Bbb E \big[\big|u_{n+1}(T,X)-u_{n+1}(T,Y)\big|^p\big]^{1/p} &\leq (p-1)^{(n+1)/2} \Bbb E \big[(u_{n+1}(T,X)-u_{n+1}(T,Y))^2 \big]^{1/2} \\ &\lesssim C^{(n+1)/2} (2p)^{(n+1)/2} e^{a(X+Y)/2} \frac{T^{(n+1)/4}}{\sqrt{(n/2)!}} |X-Y|^{1/2}.\end{align*} Using Minkowski's inequality and summing over all $n$, we then obtain $$\Bbb E\big[ \big|\mathscr Z_0(T,X) - \mathscr Z_0(T,Y)\big|^p \big]^{1/p} \leq \sum_{n \geq 1} \Bbb E \big[\big|u_{n}(T,X)-u_{n}(T,Y)\big|^p\big]^{1/p} \leq C(p,T)e^{a(X+Y)/2}|X-Y|^{1/2}.$$ Here $C(p,T):= \sum_n\frac{ (2CpT^{1/2})^{(n+1)/2}}{\sqrt{(n/2)!}}$, which is independent of $X,Y$ and increasing as a function of $T$. This is enough (by Kolmogorov's criterion) to ensure that $\mathscr Z_0$ is Holder continuous of exponent $1/2-\epsilon$ (on compact sets) in the spatial variable.
\\
\\
For the temporal regularity, one computes
\begin{align}\notag&\mathbb E[ (u_{n+1}(T,X)^2-u_{n+1}(S,X))^2 ] \\ \notag&= \Bbb E \bigg[ \bigg( \int_0^T \int_{\Bbb R_+} \mathscr P_{T-U}^T(X,Z) u_n(U,Z) \xi(dZdU) - \int_0^S\int_{\Bbb R_+} \mathscr P_{S-U}^S(X,Z) u_n(U,Z) \xi(dZdU)\bigg)^2\bigg] \\ \notag&= \int_0^S \int_{\Bbb R_+} \big( \mathscr P_{T-U}^T(X,Z) - \mathscr P_{S-U}^S(X,Z) \big)^2 \Bbb E[u_n(U,Z)^2]dZdU \\\notag &\;\;\;\;\;\;\;\;\;\;\;\;\;\;\;\;\;\;\;\;\;\;\;\;\;\;\;\;\;\;\;+ \int_S^T \int_{\Bbb R_+} \mathscr P_{T-U}^T(X,Z)^2 \Bbb E[u_n(U,Z)^2]dZdU ,\end{align}
Let us call the integrals in the last expression as $I_1,I_2$ respectively. As before, one has $\Bbb E[u_n(U,Z)^2] \leq e^{aZ}F_n(U) \leq e^{aZ}\frac{C^nU^{n/2}}{(n/2)!}$. Then one uses \eqref{tem} to bound the inner integral of $I_1$ by $$\int_{\Bbb R_+}\big( \mathscr P_{T-U}^T(X,Z) - \mathscr P_{S-U}^S(X,Z) \big)^2e^{aZ}dZ \leq Ce^{2aX}(S-U)^{-1/2} |T-S|^{1/2}, $$ and one also uses \eqref{dik} to to bound the inner integral of $I_2$ as $$\int_{\Bbb R_+} \mathscr P_{T-U}^T(X,Z)^2 e^{aZ} dZ \leq C(T-U)^{-1/2}e^{aX}.$$ Then one finally performs the integral over $U$ on the respective domains, and one can obtain that $I_1+I_2\le C^{n+1}e^{2aX}T^{(n+1)/2} |T-S|^{1/2} $. Then one uses hypercontractivity and sums over $n$ (exactly as in the spatial case), to get that $$\Bbb E\big[ \big|\mathscr Z_0(T,X) - \mathscr Z_0(S,X)\big|^p \big]^{1/p} \leq C(p,T) e^{aX} |T-S|^{1/4}.$$ Here $C(p,T)$ is an increasing function of $T$ (same as before) so it can be bounded from above on any compact set of $(X,T)$'s. This is enough to give Holder regularity of $1/4-\epsilon$ in time, by Kolmogorov's criterion.
\end{proof}

Next, we discuss the relationship of the $\mathscr{Z}$ we have constructed in Theorem \ref{exist} with the Dirichlet-boundary \eqref{SHE}.

\begin{prop}\label{eq} Any solution of the SPDE \eqref{spde} must a.s. satisfy the following relation for all $T,X>0$  $$\mathscr{Z}(T,X)\big( 2 \Phi(X/\sqrt{T})-1 \big) = Z_{Dir}(T,X)$$ where $Z_{Dir}$ solves the Dirichlet-boundary \eqref{SHE} as in Definition 2....., with the same initial data $f$.
\end{prop}

\begin{proof} One notes the following relation for $X>0$, which is immediate from Definition \ref{sht}: \begin{equation}\label{r}\mathscr{P}_t^T(X,Y) \big( 2\Phi(X/\sqrt{T})-1 \big) =\begin{cases} P_t^{Dir}(X,Y) \big( 2 \Phi(Y/\sqrt{T-t})-1\big), & t<T \\ P_T^{Dir}(X,Y), & t=T\end{cases}.\end{equation} So suppose $\mathscr{Z}$ solves \eqref{spde}, and define $$A(X,T):= \mathscr{Z}(T,X) \big( 2\Phi(X/\sqrt{T})-1 \big).$$
By multiplying both sides of \eqref{spde} by $2\Phi(X/\sqrt{T})-1$ and applying \eqref{r}, one has the relation
\begin{align*}
    A(T,X) &= \int_{\Bbb R_+} P_T^{Dir}(X,Y) f(Y) dY + \int_0^T \int_{\Bbb R_+} P_{T-S}^{Dir}(X,Y) \bigg[ \mathscr{Z}(S,Y) \big( 2\Phi(Y/\sqrt{S})-1 \big)\bigg] \xi(dY,dS) \\ &= \int_{\Bbb R_+} P_T^{Dir}(X,Y) f(Y) dY + \int_0^T \int_{\Bbb R_+} P_{T-S}^{Dir}(X,Y) A(S,Y) \xi(dY,dS), 
\end{align*}
so that $A$ is indeed a mild solution to the Dirichlet-boundary \eqref{SHE}.
\end{proof}

One thing we have not addressed is the \textit{uniqueness} of solutions to the Dirichlet-boundary \eqref{SHE} in some large-enough class of random space-time functions. This can be obtained from Theorem \ref{exist} with minimal work, and with the same conditions on the initial data $f$, one can in fact obtain existence/uniqueness in the space of $\xi$-adapted space-time functions $A$ satisfying $\sup_{T\leq \tau,\; X\geq 0} \Bbb E[ A(T,X)^2 ]<\infty$.

\begin{cor}\label{lim}
Consider any solution $Z_{Dir}$ of Dirichlet-boundary \eqref{SHE}, started from any initial data $f$ satisfying the assumptions of theorem \eqref{exist}. Then almost surely, for every $T>0$ the $\lim_{X\to 0} \frac{Z_{Dir}(T,X)}{X}$ exists.
\end{cor}

It is somewhat fascinating that we are able to obtain such a result via an intricate probabilistic analysis of seemingly unrelated uniform random-walk measures!

\begin{proof}
Consider the solution $\mathscr{Z}$ to \eqref{spde} started from initial data $f$. By the preceding Proposition, we can couple this with the solution to the Dirichlet-boundary \eqref{SHE} in such a way so that $$\mathscr{Z}(X,T) = \frac{Z_{Dir}(T,X)}{2\Phi(X/\sqrt{T})-1}$$ for all $X>0$ and $T\geq 0$. But we know that $\mathscr{Z}$ extends continuously to $X=0$ by Theorem \ref{exist}, hence we know that $$\lim_{X\to 0} \frac{Z_{Dir}(T,X)}{2\Phi(X/\sqrt{T})-1}$$ exists, and since $2\Phi(X/\sqrt{T})-1$ has nonzero derivative at $X=0$, the claim follows.
\end{proof}

\section{Convergence of the partition function to SHE}

In this section we use a discrete chaos expansion together with the methods of \cite{AKQ14a, CSZ17a} and the heat kernel estimates of the previous sections in order to prove Theorem \ref{thm2}. The first step (subsection 5.1) is to simplify the geometry of the region where our directed polymer lives, and then (in subsection 5.2) we will prove the convergence result in the simpler domain.

\subsection{Reduction from octant to quadrant}

In this subsection, we reduce the technicality of working with the partition function in an octant to working with it in a quadrant, which simplifies many computations. The dichotomy here is that the quadrant has a simple geometry which makes polymer-convergence results of the desired type quite straightforward; on the other hand, the octant has the advantage that one has nice identities such as those of Corollary \ref{cor}(3) which fail for a quadrant. Hence, one viewpoint is simpler for technical computations while the other is well-adapted for exact solvability. The results of this section are specific to the case of our positive random walk measures; however, the general outline and arguments which will be given may be easily modified for other random walk measures (such as the reflecting walk) as long as the analogous heat kernel bounds hold. Thus, this section may prove useful to other works of a similar flavor.
\\
\\
In what follows, we fix a sequence $\omega^n = \{\omega^n_{i,j}\}_{i,j \geq 0}$ of i.i.d. random environments with $n \in \Bbb N$. As always, we denote by $\Bbb E$ (resp. $\mathbb P$) the expectation (resp. probability) with respect to the environment $\omega^n_{i,j}$ and we denote by $\mathbf{E}_x^n$ (resp. $\mathbf{P}_x^n$) the expectation (resp. probability) with respect to the positive random walk measures of Section 3 (also mentioned briefly in Section 2, just before Proposition \ref{pt}). Furthermore, $T_n$ will denote the first time that this random walk $(S_n)$, started from $x \geq 0$, hits the diagonal line $\{(i,2n-i): i \geq 0\}$. 

\begin{lem}\label{killme}
Let $\mathfrak p_n^N(x,y)$ be the positive random walk transition probabilities defined at the beginning of Section 3. Then there exist constants $B,C,K>0$ such that for all $x ,n,k\geq 0$ and $a \ge 0$, $$\sum_{\substack{0\leq i_1<...<i_k \leq n \\ (x_1,...,x_k) \in \Bbb Z_{\geq 0}^k}} \mathfrak{p}_{i_1}^{n}(x,x_1)^2\mathfrak{p}_{i_2-i_1}^{n-i_1}(x_1,x_2)^2 \cdots \mathfrak{p}_{i_k-i_{k-1}}^{n-i_{k-1}}(x_{k-1},x_k)^2e^{ax_k} \leq B e^{ax+Ka^2n} C^k n^{k/2} / (k/2)!.$$
\end{lem} 

\begin{proof} By inducting (with respect to the variable $k$) on the bound in Proposition \ref{fetiz} (and noting that the constant there is independent of $x,n,N,m$), one sees that $$\sum_{(x_1,...,x_k) \in \Bbb Z_{\geq 0}^k} \mathfrak{p}_{i_1}^{n}(x,x_1)^2\mathfrak{p}_{i_2-i_1}^{n-i_1}(x_1,x_2)^2 \cdots \mathfrak{p}_{i_k-i_{k-1}}^{n-i_{k-1}}(x_{k-1},x_k)^2e^{ax_k}$$$$ \leq C^ke^{ax+Ka^2n} (i_1+1)^{-1/2} (i_2-i_1+1)^{-1/2}\cdots (i_k-i_{k-1}+1)^{-1/2} .$$
Thus the desired sum is bounded above by $$e^{ax+Ka^2n}\sum_{1 \leq i_1 <...< i_k \leq n+1} i_1^{-1/2}(i_2-i_1)^{-1/2} \cdots (i_k-i_{k-1})^{-1/2}. $$ Now one recognizes that $$n^{-k/2} \sum_{1 \leq i_1 <...<i_k \leq n+1} i_1^{-1/2}(i_2-i_1)^{-1/2} \cdots (i_k-i_{k-1})^{-1/2} $$$$= \frac{1}{n^k} \sum_{1 \leq i_1 <...<i_k \leq n+1} \big(\frac{i_1}{n}\big)^{-1/2}\big(\frac{i_2}{n}-\frac{i_1}{n}\big)^{-1/2} \cdots \big(\frac{i_k}{n}-\frac{i_{k-1}}{n}\big)^{-1/2}, $$ which (as a Riemann sum approximation) is bounded above by twice $$\int_{0 \leq t_1 <...<t_k \leq 1}  t_1^{-1/2}(t_2-t_1)^{-1/2} \cdots (t_k-t_{k-1})^{-1/2}dt_1 \cdots dt_k \leq B/(k/2)!,$$ where $B>0$. Hence the lemma is proved.
\end{proof}

\begin{lem}\label{omfg} Take a sequence $\omega^n = \{\omega^n_{i,j}\}$ of random environments satisfying the assumptions of ......... Furthermore, let $\{z_0^n(x)\}_{x \geq 0}$ be some sequence of non-negative stochastic processes with the property that $\Bbb E[z_0^n(x)^2] \leq K e^{an^{-1/2}x}$ for some constants $K,a$ which are independent of $n$ and $x$. Then there exists a constant $C$ such that for all $n,x \geq 0$ one has that $$\Bbb E \Bigg[ \sup_{0 \leq k \leq n} \mathbf{E}_x^n \bigg[z_0^n(S_n)\prod_{i=0}^{k} (1 + n^{-1/4} \omega_{i,S_i}^n) \bigg]^2 \Bigg] \leq Ce^{an^{-1/2}x}. $$
\end{lem}

\begin{proof} First we fix some $n \in \Bbb N$, and we note that the process $$M_k^n := \mathbf{E}_x^n \bigg[ z_0^n(S_n) \prod_{i=1}^k (1+n^{-1/4}\omega_{i,S_i}^n) \bigg]$$ is a $\mathbb{P}$-martingale in the $k$ variable (with respect to the filtration $(\mathcal F_k^n)_{k \geq 0}$, where $\mathcal F_k^n$ is generated by $z_0^n$ and $\{\omega^n_{i,j}\}_{0\leq j \leq i \leq k})$. Therefore by Doob's inequality, it is clear that $\Bbb E[ \sup_{0\leq k \leq n} (M_k^n)^2] \leq 4 \Bbb E [(M^n_n)^2].$ This reduces our work to proving the claim without the supremum inside the expectation (and replacing $k$ by $n$ in the product). To do this, we set $x_0:=x$ and we write 
\begin{align*}
    &\Bbb E \Bigg[ \mathbf{E}_x^n \bigg[z_0^n(S_n)\prod_{i=1}^n (1 + n^{-1/4} \omega_{i,S_i}^n) \bigg]^2 \Bigg] \\ & = \Bbb E \bigg[\bigg( \sum_{k=1}^{n} n^{-k/4} \sum_{0\leq i_1 <...<i_k \leq n} \sum_{x \in \Bbb Z_{\geq 0}^{k+1}} z_0^n(x_{k+1}) \prod_{i=1}^k \mathfrak{p}_{i_j-i_{j-1}}^{2n-i_{j-1}} (x_{j-1},x_j)\omega_{i_jx_j} \cdot \mathfrak{p}^{2n-i_k}_{2n-i_k}(x_k,x_{k+1}) \bigg)^2 \bigg] \\ &= \sum_{k=1}^n n^{-k/2} \sum_{0\leq i_1 <...<i_k \leq n} \sum_{(x_1,...,x_{k}) \in \Bbb Z_{\geq 0}^{k}} \prod_{i=1}^k \mathfrak{p}_{i_j-i_{j-1}}^{2n-i_{j-1}} (x_{j-1},x_j)^2 \bigg[ \sum_{x_{k+1}\in \Bbb Z_{\ge 0}} \Bbb E[z_0^n(x_{k+1})] \mathfrak p_{2n-i_k}^{2n-i_k}(x_k,x_{k+1})\bigg]^2.
\end{align*}
We know by assumption that $\Bbb E[z_0^n(x_{k+1})^2] \leq e^{an^{-1/2}x_{k+1}}$. By Jensen we then have that
$$\bigg[ \sum_{x_{k+1}\in \Bbb Z} \Bbb E[z_0^n(x_{k+1})] \mathfrak p_{2n-i_k}^{2n-i_k}(x_k,x_{k+1})\bigg]^2 \leq \sum_{x_{k+1}\ge 0} \Bbb E[z_0^n(x_{k+1})^2]\mathfrak p_{2n-i_k}^{2n-i_k}(x_k,x_{k+1}) \leq Ce^{an^{-1/2}x_k} ,$$
where we applied Lemma \ref{macky} in the last bound. Thus we have \begin{align*}
    \Bbb E \Bigg[ \mathbf{E}_x^n \bigg[\prod_{i=1}^n (1 + n^{-1/4} \omega_{i,S_i}^n) \bigg]^2 \Bigg] &\leq \sum_{k=1}^n n^{-k/2} \sum_{0\leq i_1 <...<i_k \leq n}
    \sum_{(x_1,...,x_{k}) \in \Bbb Z_{\geq 0}^{k}} \prod_{i=1}^k \mathfrak{p}_{i_j-i_{j-1}}^{2n-i_{j-1}} (x_{j-1},x_j)^2e^{an^{-1/2}x_k} \\ &\leq \sum_{k=1}^n n^{-k/2} BC^ke^{an^{-1/2}x}n^{k/2}/(k/2)!\\& \leq Be^{an^{-1/2}x} \sum_{k=0}^{\infty} C^k/(k/2)!.
    \end{align*}
    This completes the proof.
\end{proof}

The key estimate of this section is as follows:

\begin{thm}[Key estimate]\label{ke} Fix $\alpha \in (0,1)$. Suppose that $(z_0^n(x))_{x \in \Bbb Z_{\ge 0}}$ is a family of non-negative, continuous random processes. Assume that \begin{itemize}
    \item $\Bbb E[|z_0^n(x)-z_0^n(y)|^p] \leq Cn^{-p/4}|x-y|^{p/2}e^{an^{-1/2}(x+y)}$ for some constants $C,a,p$ independent of $n,x,y$. Further assume $p>2$.
    
    \item with the same $a$, there exist square integrable random variables $D(n)$ such that $\sup_n \Bbb E[D(n)^2]<\infty$ and $z_0^n(x) \leq D(n) e^{an^{-1/2}x}$ for all $n,x$ almost surely.
    
    \item for each $n$, the process $z_0^n$ is independent of the environment $\{\omega^n_{i,j}\}_{i,j \geq 0}$.
    
\end{itemize}

Define the ``error" random variable $$\mathcal E (x,n):= \sup_{k \in [2n-n^{\alpha},2n]} \Bigg|\mathbf E_{x}^{2n} \bigg[ z_0^n(S_{2n})\prod_{i=1}^{2n} \big(1+ n^{-1/4} \omega_{i,S_i}^n \big)- z_0^n(S_k) \prod_{i=1}^{k}\big(1+n^{-1/4}\omega_{i,S_i}^n\big) \bigg]\Bigg|.$$ Then $\sup_{x \geq 0} e^{-3an^{-1/2}x} \Bbb E|\mathcal E(x,n)| \to 0$ as $n \to \infty$.
\end{thm}

\begin{proof} By the triangle inequality, we have $\mathcal E(n) \le \mathcal E_1(n)+\mathcal E_2(n),$ where $$\mathcal E_1(n) := \sup_{k \in [2n-n^{\alpha},2n]} \Bigg|\mathbf E_{x}^{2n} \bigg[ z_0^n(S_{2n})\bigg(\prod_{i=1}^{2n} \big(1+ n^{-1/4} \omega_{i,S_i}^n \big)-  \prod_{i=1}^{k}\big(1+n^{-1/4}\omega_{i,S_i}^n\big) \bigg) \bigg]\Bigg|,$$ $$\mathcal E_2(n) := \sup_{k \in [2n-n^{\alpha},2n]} \Bigg|\mathbf E_{x}^{2n} \bigg[ \big(z_0^n(S_{2n})- z_0^n(S_k) \big)\prod_{i=1}^{k}\big(1+n^{-1/4}\omega_{i,S_i}^n\big) \bigg]\Bigg|.\;\;\;\;\;\;\;\;\;\;\;\;\;\;\;\;\;\;\;\;\;$$ We separately show that both of these satisfy the desired bound.
\\
\\
First we consider $\mathcal E_1$. For now let us fix some $n\in \Bbb N$. Let us define a martingale $$M^n_k:= \mathbf E_{x}^{2n} \bigg[ z_0^n(S_{2n}) \bigg( \prod_{i=2n-k}^{2n} (1+n^{-1/4}\omega_{i,S_i}^n) \;\;-\;\;1 \bigg)\bigg].$$
This is a $\mathbb P$-martingale in the $k$ variable, for fixed $n \in \Bbb N$ (with respect to the filtration $(\mathcal F_k^n)_{k \geq 0}$, where $\mathcal F_k^n$ is generated by $z_0^n$ and $\{\omega^n_{i,j}\}_{0\leq j \leq i \leq k})$. Consequently, Doob tells us \begin{equation}\label{doob1}\Bbb E \big| \sup_{k \in [0,n^{\alpha}]} M_k \big|^2 \leq 4 \Bbb E|M^n_{n^{\alpha}}|^2.\end{equation} Computing the right-hand side, one gets \begin{align*}&\Bbb E[(M_{n^{\alpha}}^n)^2 ] = \Bbb E \Bigg[ \mathbf E_{x}^{2n}\bigg[ z_0^n(S_{2n}) \bigg( \prod_{i=2n-n^{\alpha}}^{2n} (1+n^{-1/4}\omega_{i,S_i}^n ) \;- \;1\bigg) \bigg]^2 \Bigg] \\ &= \sum_{\substack{1 \leq k \leq n^{\alpha}\\0\leq i_1 <...<i_k\leq  n^{\alpha} \\ (x_1,...,x_{k})\in \Bbb Z_{\ge 0}^{k}}} n^{-k/2}\cdot \mathfrak p^{2n}_{2n-n^{\alpha}+i_1}(x,x_1)^2 \prod_{j=1}^{k-1} \mathfrak p_{i_j-i_{j-1}}^{n^{\alpha}-i_{j-1}} (x_j,x_{j+1})^2 \bigg[\sum_{x_{k+1} \in \Bbb Z_{\ge 0}} \mathfrak p_{n^{\alpha}-i_k}^{n^{\alpha}-i_k}(x_k,x_{k+1}) \Bbb E[z_0^n(x_{k+1})] \bigg]^2.\end{align*}By Jensen we compute that $$\bigg[\sum_{x_{k+1} \in \Bbb Z_{\ge 0}} \mathfrak p_{n^{\alpha}-i_k}^{n^{\alpha}-i_k}(x_k,x_{k+1}) \Bbb E[z_0^n(x_{k+1})] \bigg]^2 \leq \sum_{x_{k+1} \in \Bbb Z_{\ge 0}} \mathfrak p_{n^{\alpha}-i_k}^{n^{\alpha}-i_k}(x_k,x_{k+1}) \Bbb E[z_0^n(x_{k+1})^2] \leq Ce^{an^{-1/2}x_k},$$ where we used the given condition on $z_0$ and Proposition \ref{macky} in the last inequality. Combining this with the previous expression, we see that $$\Bbb E[(M_{n^{\alpha}}^n)^2] \leq \sum_{k=1}^{n^{\alpha}} n^{-k/2} \sum_{\substack{0\leq i_1 <...<i_k\leq  n^{\alpha} \\ (x_1,...,x_{k})\in \Bbb Z_{\ge 0}^{k}}} \mathfrak p^{2n}_{2n-n^{\alpha}+i_1}(x,x_1)^2 \prod_{j=1}^{k-1} \mathfrak p_{i_j-i_{j-1}}^{n^{\alpha}-i_{j-1}} (x_j,x_{j+1})^2e^{an^{-1/2}x_k}.$$
By repeatedly applying Proposition \ref{fetiz}, this is in turn bounded above by $$\sum_{k=1}^{n^{\alpha}} n^{-k/2}C^k e^{an^{-1/2}x}\sum_{1 \leq i_1<...<i_k\le n^{\alpha}+1} (2n-n^{\alpha}+i_1)^{-1/2}(i_2-i_1)^{-1/2}\cdots (i_k-i_{k-1})^{-1/2}. $$ 
We use the bound $(2n-n^{\alpha}+i_1)^{-1/2} \leq n^{-1/2}$ and then (by viewing it as a Riemann sum as we did in the proof of Lemma \ref{killme}) the sum over $i_1<...<i_k\le n^{\alpha}$ can be bounded above by $Cn^{(k+1)\alpha/2}/(k/2)!.$
Hence the entire sum is bounded above by $$ e^{an^{-1/2}x} \sum_{k=1}^{n^{\alpha}}n^{-(1-\alpha)k/2} C^k/(k/2)! \leq n^{-(1-\alpha)/2} e^{an^{-1/2}x} \sum_1^{\infty} C^k/(k/2)!  = Ce^{an^{-1/2}x} n^{-(1-\alpha)/2},$$ which implies the desired result on $I_1$.
\\
\\
Now we consider $\mathcal E_2(n)$. Recall the given condition that $\Bbb E[|z_0^n(x)-z_0^n(y)|^p] \leq Cn^{-p/4}|x-y|^{p/2}e^{an^{-1/2}(x+y)}$ Let $\gamma:=\frac12-\frac1p$. Define a ``microscopic modulus of Holder continuity" as follows: $$C(M,n): = \sup_{\substack{x \neq y\\ 0\le x,y\leq M}} \frac{|z_0^n(x)-z_0^n(y)|}{n^{\gamma/2}|x-y|^{\gamma}},$$ then the given condition is enough to show (by a stronger form of Kolmogorov's criterion, see for instance the Garsia-Rodemich-Rumsey inequality) that $\sup_n \Bbb E[C(M,n)^2] \leq Ce^{an^{-1/2}M}$, for a constant $C$ independent of $M,a$. With this in mind, we now bound $\mathcal E_2$.
\\
\\
Let us write $z_0^n(S_{2n})-z_0^n(S_k)=:A_k^n$ and $\prod_{i=1}^k (1+n^{-1/4}\omega^n_{i,S_i})=:B_k^n$. Define the $\sigma$-algebra $\mathcal G_k:=\sigma(S_1,...,S_k)$ (note that $\mathcal G_k$ implicitly depends on $x,n$ since the measures $\mathbf P_x^n$ are varying). Then we may write $$\mathbf E_x^{2n}[A^n_kB^n_k] = \mathbf E_x^{2n} \bigg[ \mathbf E_x^{2n}[A_k\;|\;\mathcal G_k] \; B_k^n \bigg].$$ 
By the Markov property we see that $\mathbf E_x^{2n}[(z_0^n(S_{2n})-z_0^n(S_k))\;|\;\mathcal G_k] = f(S_k,2n,n-k)$, where $f(x,n,i):=\mathbf E_x^i[(z_0^n(S_i)-z_0^n(x))].$ Then we write $$f(x,n,i) \le \mathbf E_x^i[|z_0^n(S_i)-z_0^n(x)|\cdot 1_{\{S_i \ge M+x\}} ]+ \mathbf E_x^i[|z_0^n(S_i)-z_0^n(x)|\cdot 1_{\{S_i < M+x\}} ] .$$ Let us call the terms on the right side as $I_1(M,x,n,i),I_2(M,x,n,i)$, respectively. We will individually bound the expectation relevant to both of these, starting with $I_2$. Since $S_i<M+x$ implies that $S_i\vee x \leq S_i+x \leq M+2x$, it is clear that $$I_2(M,x,n,i) \leq C(M+2x,n) \mathbf E_x^i[n^{-\alpha/2}|S_i-x|^{\alpha}] \leq C(M+2x,n) n^{-\gamma/2}i^{\gamma/2}.$$ The final bound is by Proposition \ref{pmoments}. Letting $C(M):= \sup_n \Bbb E[C(M,n)^2]^{1/2}$ (which is bounded above by $Ce^{an^{-1/2}M}$ as discussed above), we may conclude that 
\begin{align*}
\Bbb E\bigg[&\sup_{k\in[2n-n^{\alpha},2n]}\mathbf E_x^{2n}\big[ I_2(M,S_k,2n,n-k) B_k^n\big]\bigg] \\&\leq \Bbb E\bigg[n^{-\gamma/2}C(M+2x,n) \sup_{k\in [2n-n^{\alpha},2n]} (n-k)^{\gamma/2}\mathbf E_x^{2n}[B_k^n]\bigg] \\
&\leq n^{-\gamma/2} \Bbb E[C(M+2x,n)^2]^{1/2} \Bbb E\bigg[ \sup_{k\in [2n-n^{\alpha},2n]} (n^{\alpha})^{\gamma/2}\mathbf  E_x^{2n}[B_k^n]^2\bigg]^{1/2} \\ &\leq n^{-(1-\alpha)\gamma/2} C(M+2x) \Bbb E[\sup_{k \leq 2n} \mathbf E_x^{2n}[B_k^n]^2] \\ &\leq Ce^{an^{-1/2}(M+2x)} n^{-(1-\alpha)\gamma/2}e^{an^{-1/2}x}.
\end{align*}
In the last line, we applied Lemma \ref{omfg} to conclude that $\Bbb E[\sup_{k \leq 2n} \mathbf E_x^{2n}[B_k^n]^2] \leq Ce^{-an^{-1/2}x}$, where $C$ is independent of $x,n,M$. Next, we need to bound the same quantity with $I_2$ replaced by $I_1$. To bound $I_1$, first note (by Cauchy-Schwarz and then the concentration theorem) that 
\begin{align*}I_1(M,x,n,i) &\leq D(n) \mathbf E_x^i [ (e^{an^{-1/2}x}+e^{an^{-1/2}S_i}) \cdot 1_{\{S_i>M+x\}}] \\ &\leq D(n) \mathbf E_x^i[(e^{an^{-1/2}x}+e^{bn^{-1/2}S_i})^2]^{1/2} \mathbf P_x^i(S_i>M+x)^{1/2} \\ &\leq C D(n)e^{an^{-1/2}x} e^{-cM^2/i}, \end{align*}
where $D(n)$ are the $L^2$ random variables satisfying the conditions of the theorem statement. Thus we find that \begin{align*}
\Bbb E\bigg[&\sup_{k\in[2n-n^{\alpha},2n]}\mathbf E_x^{2n}\big[ I_2(M,S_k,2n,n-k) B_k^n\big]\bigg] \\&\leq C \Bbb E\bigg[D(n) \sup_{k\in [2n-n^{\alpha},2n]}e^{-cM^2/(n-k)} \mathbf E_x^{2n}[e^{an^{1/2}S_k}B_k^n]\bigg] \\ &\leq Ce^{-cM^2/n^{\alpha}} \Bbb E[D(n)^2]^{1/2} \Bbb E\bigg[ \sup_{k\leq n} \mathbf E_x^{2n}[e^{an^{1/2}S_k}B_k^n]^2\bigg].
\end{align*}
Now, $e^{an^{1/2}S_k}$ is a $\mathbf P_x^n$-submartingale by \eqref{subgale}, thus we have $$\mathbf E_x^{2n}[e^{an^{1/2}S_k}B^n_k] \leq \mathbf E_x^{2n}[\mathbf E_x^{2n}[e^{an^{1/2}S_n}|\mathcal G_k]B_k^n] = \mathbf E_x^{2n}[e^{an^{1/2}S_n}B^n_k],$$ since $B^n_k$ is $\mathcal G_k$ measurable. This means that $$\Bbb E\bigg[ \sup_{k\leq n} \mathbf E_x^{2n}[e^{an^{1/2}S_k}B_k^n]^2\bigg] = \Bbb E[ \mathbf E_x^{2n}[ e^{an^{1/2}S_{n}}B^n_k]^2]\leq Ce^{an^{-1/2}x},$$ where we used Lemma \ref{omfg} in the last bound. Summarizing our progress so far, we combine the bounds on $I_1$ and $I_2$ to see that $$\Bbb E|\mathcal E_2(n)| \leq Ce^{an^{-1/2}(M+3x)}n^{-(1-\alpha)\gamma/2} +Ce^{-cM^2/n^{\alpha}}e^{an^{-1/2}x}.$$ Take $M=\sqrt n$ and multiply both sides by $e^{-3an^{1/2}x}$; then let $n \to \infty$ and the result approaches $0$.
\end{proof}

We will now prove that in our situation, the conditions of the preceding theorem actually apply. First we have a lemma which will be useful in extracting the random variables $D(n)$ stated in the conditions of Theorem \ref{ke}.

\begin{lem}\label{supergale}
Let $(X_n)_{n \ge 0}$ be a non-negative $L^1$ supermartingale. Then $$\Bbb P\big(\sup_n X_n>a\big) \leq \frac{\Bbb E[X_0]}{a}.$$
\end{lem}

\begin{proof}We apply Doob-Meyer decomposition to write $X=M-A$, where $M$ is a martingale with $M_0=X_0$, and $A_0$ is a non-decreasing process with $A_0=0$. Then $M$ is a positive martingale and $X \leq M$. Doob's inequality then shows that $$\Bbb P\big( \sup_{n\leq N} X_n > a\big) \leq \Bbb P\big( \sup_{n\leq N} M_n > a \big) \leq \frac{\Bbb E[M_N]}{a} = \frac{\Bbb E[M_0]}{a}.$$ Since $M_0=X_0$, letting $N \to \infty$ gives the claim, because the right side does not depend on $N$ and the left side approaches $\Bbb P\big( \sup_n X_n >a\big)$ by monotone convergence.
\end{proof}

\begin{prop}\label{nun}
For each $n \in \Bbb N$, let $\{\omega^n_{i0}\}_{i \geq 1}$ be a family of iid random variables such that $\omega^n_{i0}$ has finite $p^{th}$ moment, with $p>2$. Also assume that $1+n^{-1/4}\omega_{i0}^n>0$ a.s. and that $\sup_n \Bbb E[|\omega^n_{10}|^p] <\infty$. Furthermore, assume that $\Bbb E[\omega^n_{i0}] = \mu n^{-1/4} +o(n^{-1/4})$ and var$(\omega^n_{i0}) = \sigma^2+o(1)$ as $n \to \infty$. Define $z_0^n(x):= \prod_{i=1}^x (1+n^{-1/4}\omega^n_{i0}).$ Then $z_0^n$ satisfies the conditions of the preceding theorem:
\begin{itemize}
    \item $\Bbb E[|z_0^n(x)-z_0^n(y)|^p] \leq Cn^{-p/4}|x-y|^{p/2}e^{an^{-1/2}(x+y)}$ for some constants $C,a$ independent of $n,x,y$. 
    
    \item with the same $a$, there exist square integrable random variables $D(n)$ such that $\sup_n \Bbb E[D(n)^2]<\infty$ and $z_0^n(x) \leq D(n) e^{an^{-1/2}x}$ for all $n,x$ almost surely.
    
\end{itemize}
\end{prop}

\begin{proof} Before proving either bullet point, we prove a preliminary bound which is useful. Using $|1+n^{-1/4}\omega^n_{i0}|^{p} = 1+pn^{-1/4}\omega^n_{i0} + \frac12(p^2-p)n^{-1/2}(\omega^n_{i0})^2+ o(n^{-1/2})$ which has expectation roughly $1+n^{-1/2}(p\mu+\frac{p^2-p}{2}\sigma^2)+o(n^{-1/2}) \leq 1+an^{-1/2},$ (for some $a=a(p)$ we see that \begin{align}\label{pee}\Bbb E[z_0^n(x)^p] &= \prod_{i=1}^x \Bbb E[(1+n^{-1/4}\omega^n_{i0})^p] \leq  (1+n^{-1/2}a)^{x} \leq e^{an^{-1/2}x},
\end{align}
since $1+v\leq e^v$. With this preliminary bound in mind, we proceed to the proof of the first bullet point. It suffices to prove the claim when $y=0$ (i.e., $z_0^n(y)=1$), by independence of the multiplicative increments of $z_0^n$. Let us begin by writing \begin{align*}
    \Bbb E[|z_0^n(x)-1|^p] \leq 2^p \bigg( \Bbb E\bigg[\bigg|z_0^n(x)-\frac{z_0^n(x)}{\Bbb E[z_0^n(x)]}\bigg|^p\bigg] + \Bbb E\bigg[\bigg|\frac{z_0^n(x)}{\Bbb E[z_0^n(x)]}-1\bigg|^p\bigg]\bigg).
\end{align*}
Let us call these expectations on the right side as $E_1$ and $E_2$, respectively. We bound each of these separately. For $E_1$, one notes by using \eqref{pee} that 
\begin{align*}
    E_1 & = \Bbb E[z_0^n(x)^p ] \bigg| 1- \frac1{\Bbb E[z_0^n(x)]}\bigg|^p \leq e^{an^{-1/2}x} \big| 1- e^{-an^{1/2}x}\big|^p \\ &\leq e^{an^{-1/2}x} \big(1-e^{-an^{-1/2}x}\big)^2\leq e^{an^{-1/2}x} \big( an^{-1/2}x \big)^2 = ae^{an^{-1/2}x}n^{-1}x^2,
\end{align*}
where we used $\Bbb E[z_0^n(x)] \leq \Bbb E[z_0^n(x)^p]^{1/p} \leq e^{an^{-1/2}x}$ (by \eqref{pee}) in the first inequality, and we used $1-e^{-v}\leq v$ in the third one. This already gives the desired bound on $E_1$.
\\
\\
Now we bound $E_2$. This is the difficult part, and one needs to somehow exploit cancellations which occur at the quadratic scale (e.g., via a Burkholder-type inequality). To do this, first note that the process $M^n_x:=\frac{z_0^n(x)}{\Bbb E[z_0^n(x)]}$ is a martingale in the $x$-variable (for fixed $n$). Define $\zeta^n_i:= \frac{1+n^{-1/4} \omega^n_{i0}}{\Bbb E[1+n^{-1/4}\omega_{i0}]}.$ Then Burkholder-Davis-Gundy says \begin{align}\label{bdg1}
    E_2 & \leq C\Bbb E\bigg[ \bigg(\sum_{i=1}^x (M_i^n-M_{i-1}^n)^2 \bigg)^{p/2} \bigg] = C\Bbb E \bigg[ \bigg( \sum_{i=1}^x (\zeta^n_1)^2\cdots (\zeta_{i-1}^n)^2 ( \zeta^n_i-1)^2\bigg)^{p/2}\bigg]
\end{align}
Now, using the given conditions, $|\zeta^n_i-1|$ is easily seen to be bounded above by $C(n^{-1/4}|\omega_{i0}^n|+n^{-1/2}),$ so the square is bounded by $C(n^{-1/2}(\omega^n_{i0})^2 +n^{-1}).$ Writing $\|A\|_2:= \Bbb E[A^2]^{1/2},$ we then notice by triangle inequality and independence of $\zeta^n_i$ that
\begin{align*}
    \bigg\|\sum_1^x (\zeta^n_1)^2\cdots (\zeta_{i-1}^n)^2 ( \zeta^n_i-1)^2 \bigg\|_{p/2} \leq Cn^{-1/2} \sum_1^x \|(\zeta^n_1)^2\|_{p/2}\cdots \|(\zeta^n_{i-1})^2\|_{p/2} \|(\omega^n_{i0})^2+n^{-1/2}\|_{p/2}.
\end{align*}
Now, it holds that $\|(\zeta^n_1)^2\|_{p/2} \leq e^{2an^{-1/2}/p},$ by \eqref{pee} (with $x=1)$. Hence each term of the sum may be bounded above by $e^{2an^{-1/2}x/p}.$ The contribution of the $n^{-1/2}$ term next to $(\omega^n_{i0})^2$ is then seen to be negligible so we disregard it. Hence the the entire sum may be bounded by $C n^{-1/2}x e^{2an^{-1/2}x/p},$ which (combined with \eqref{bdg1} and the fact that $\|(\omega^n_{i0})^2\|_{p/2}$ is bounded independently of $n$ by assumption) completes the proof.
\\
\\
Now we prove the second bullet point. Note that $\frac{z_0^n(x)^p}{\Bbb E[z_0^n(x)^p]}$ is a positive martingale in the $x$-variable (for fixed $n$). Let $D(n):=\sup_{x\geq 0} z_0^n(x)/\Bbb E[z_0^n(x)^p]^{1/p}$. Then it is clear from Lemma \ref{supergale} that $\Bbb P(D(n)^p >a) \leq a^{-1},$ so that $\Bbb P(D(n)>a) \leq a^{-p}.$ If $p>2$, then this easily implies that $\sup_n \Bbb E[D(n)^2] <\infty.$ But \eqref{pee} tells us that $\Bbb E[z_0^n(x)^p]^{1/p} \leq Ce^{an^{-1/2}x}$ so we are done.
\end{proof}

Next, we finally, prove the octant-quadrant reduction theorem, i.e., that we can actually replace $T_n$ with $2n$ as discussed in section 2.

\begin{prop}[Octant-Quadrant Reduction]\label{oqr}
Let $\omega^n_{i,j}$, $\mathbf E_x^n$, $S_n$, and $T_n$ be as defined in Section 2. Let
$$\mathscr E(x,n):=\mathbf{E}_{x}^{2n} \bigg[  z^n_0(S_{2n})\prod_{i=0}^{2n} (1+n^{-1/4} \hat{\omega}_{iS_i})\bigg] - \mathbf{E}_{x}^{2n} \bigg[  z^n_0(S_{T_n})\prod_{i=0}^{T_n} (1+n^{-1/4} \hat{\omega}_{iS_i})\bigg]. $$ Let $x_n$ be a sequence of non-negative integers such that $x_n \leq Cn^{1/2}$ for some $C>0$. Then $\mathscr E(x_n,n) \to 0$ in probability.
\end{prop}

\begin{proof}
First we will show that $\sum_n \mathbf P_{x_n}^{2n}(T_n \leq 2n-n^{2/3}) < \infty$. By Borel-Cantelli, this would imply that all $\mathbf P_{x_n}^{2n}$ may be coupled onto the same probability space in such a way so that one almost surely has $T_n > 2n-n^{2/3}$ for large enough $n$. Then the result follows immediately by applying Theorem \ref{ke} with $\alpha = 2/3$ (there is nothing special about $2/3$: one can take any $\alpha >1/2$).
\\
\\
To prove that $\sum_n \mathbf P_{x_n}^{2n}(T_n \leq 2n-n^{2/3}) < \infty$, one first notes that the event $\{T_n \leq 2n-n^{2/3}\}$ can only happen if $\sup_{i \leq n} S_i \geq n^{2/3}$. But by the concentration (Theorem \ref{conc}), we know that $$\mathbf P_{x_n}^{2n}\big(\sup_{i\leq n} S_i \geq n^{2/3}\big) \leq Ce^{-c(n^{2/3}-x_n)^2/n} \leq Ce^{-c(n^{2/3}-Cn^{1/2})^2/n} \leq Ce^{-c'n^{1/3}}.$$ The right side is summable as a function of $n$, completing the proof.
\end{proof}

\subsection{Convergence in a quadrant}

With the reduction (Proposition \ref{oqr}) finished, we may simply consider a modified partition function \begin{align}\notag Z_k(n,x) &:= \mathbf E_{x}^{2n} \bigg[ z_0^k(S_{2n}) \prod_{i=0}^{2n-1} (1+k^{-1/4}\omega^k_{iS_i})\bigg] \\ \label{dof} &= \sum_{r=0}^{2n}k^{-r/4} \sum_{\substack{0\leq i_1 < ... <i_r <2 n\\(x_1,...,x_{r+1})\in \Bbb Z_{\ge 0}^r}}\prod_{i=1}^r \mathfrak p_{i_j-i_{j-1}}^{2n-i_{j-1}}(x_{j-1},x_j) \omega_{i_jx_j}^n \cdot \big(z_0^n(x_{r+1}) \mathfrak p_{2n-i_r}^{2n-i_r}(x_r,x_{r+1})\big),
\end{align}
with $i_0:=0$ and $x_0:=x$. We are now going to show that the rescaled processes \begin{equation}\label{resc}
\mathscr Z_n(T,X):= Z_n(nT,n^{1/2}X)
\end{equation}converge in law (as $n \to \infty$, with respect to the topology of uniform convergence on compact subsets of $\Bbb R_+ \times \Bbb R_+$) to the solution of \eqref{spde}. The first step for doing this is proving tightness in the appropriate Holder space.

\begin{prop}[Tightness] Let $\mathscr Z_n$ be defined as in \eqref{resc}, and assume that (for each $k$), the iid weights $\{\omega^k_{ij}\}_{i,j}$ have $p_0>8$ moments, bounded independently of $k$. Also let $\|X\|_p:= \Bbb E[|X|^p]^{1/p}.$ Then for every $p \in [1,p_0], a \ge 0$, $\theta \in [0,1)$, and compact set $K \subset [0,\infty)^2$ there exists $C=C(a,p,\theta,K)>0$ such that one has the following estimates uniformly over all pairs of space-time points $(T,X),(S,Y)\in K$:
\begin{align}
    \|\mathscr Z_n(T,X) \|_p &\leq C,\label{unif}\\ \|\mathscr Z_n(T,X)-\mathscr Z_n(T,Y)\|_p &\leq C|X-Y|^{\theta/2},\label{spatial}\\ \|\mathscr Z_n(T,X) - \mathscr Z_n(S,X)\|_p &\leq C|T-S|^{\theta/4}.\label{temporal}
\end{align}
In particular, the laws of the $Z_n$ are tight with respect to the topology of uniform convergence on compact subsets of $C(\Bbb R_+ \times \Bbb R_+)$.
\end{prop}

We remark that the restriction $p\in [1,8+\epsilon]$ is only necessary to obtain tightness in the Holder space. Using more elegant arguments, this may be extended to $p_0 \ge 6$ (see [AKQ14a, Appendix B]). The one-point convergence result will only require two moments.

\begin{proof}
Note that the functions $Z_k$ defined in \eqref{dof} satisfy the following recursion:
\begin{align*}
    Z_k(n+1,x) = \mathfrak p_1^{n+1}(x,x+1) &Z_k(n,x+1) + \mathfrak p_1^{n+1}(x,x-1) Z_k(n,x-1) \\ &+ k^{-1/4}\omega^k_{(n+1)x} Z_k(n,x).
\end{align*}
Iterating this equation $n$ times and applying the semigroup property will give a Duhamel-form (mild) equation for $Z_k$, namely \begin{equation}\label{du}Z_k(n,x) = \sum_{y \ge 0} \mathfrak p_n^n(x,y) z_0^k(y) +k^{-1/4}\sum_{i=1}^{n}\sum_{y \ge 0} \mathfrak p_{n-i}^n(x,y)Z_k(i,y)\omega_{iy}^k.
\end{equation}
Define the martingale $M_r(x,n,k):= k^{-1/4}\sum_{i=0}^{r-1} \sum_{y \ge 0} \mathfrak p_{n-i}^n(x,y)Z_k(i,y)\omega^k_{(i+1)y}.$ This is a martingale in the $r$-variable (for fixed $x,n,k$), with respect to the filtration $\mathcal F^{k}_r:= \sigma(\{\omega^k_{ij}\}_{0\leq i \leq r;j\ge 0}).$ This is because $Z_k(i,y)$ is $\mathcal F^{k}_r$-measurable, and $\mathcal F^k_r$ is independent of the mean-zero random variables $\omega^k_{(r+1)y}$ with $y \geq 0$. Applying Burkholder-Davis-Gundy to $M_r(x,n,k)$ shows that 
\begin{align}\notag
    \|M_r(x,n,k)\|_p^2 &\leq C\bigg\|k^{-1/2} \sum_{i=0}^{r-1} \bigg[ \sum_{y\ge 0} \mathfrak p_{n-i}^n(x,y)Z_k(i,y) \omega^k_{(i+1)y}\bigg]^2 \bigg\|_{p/2} \\ \label{duk}&\leq Ck^{-1/2} \sum_{i=0}^{r-1} \bigg\|\sum_{y \ge 0} \mathfrak p_{n-i}^n(x,y)Z_k(i,y) \omega^k_{(i+1)y}\bigg\|_p^2.
\end{align}
Next, we notice that since the $\omega^k_{(i+1)y}$ are independent of $Z_k(i,y)$, another application of Burkholder-Davis-Gundy (or in this case, its more elementary version for independent sums, the Marcinkiewicz-Zygmund inequality) shows that \begin{equation}\label{ducks}\bigg\|\sum_{y \ge 0} \mathfrak p_{n-i}^n(x,y)Z_k(i,y) \omega^k_{(i+1)y}\bigg\|_p^2 \leq C\sum_{y \ge 0} \mathfrak p_{n-i}^n(x,y)^2 \|Z_k(i,y)\|_p^2 \|\omega^k_{(i+1)y}\|_p^2,
\end{equation}
Since $p \le p_0$ and the $p_0^{th}$ moments of $\omega^k_{iy}$ are bounded independently of $k,i,y$ it follows that $\|\omega^k_{(i+1)y}\|_p^2$ may be absorbed into the constant. Combining \eqref{du},\eqref{duk},\eqref{ducks}, one finds that 
\begin{equation}\label{budd}
    \|Z_k(n,x)\|_p^2 \leq C\bigg(\sum_{y \ge 0} \mathfrak p_n^n(x,y)\|z_0^k(y)\|_p\bigg)^2 + Ck^{-1/2}\sum_{i=0}^{n-1} \sum_{y \ge 0} \mathfrak p_{i}^n(x,y)^2 \|Z_k(n-i,y)\|_p^2.
\end{equation}
Now, we note that $\|z_0^k(y)\|_p \leq e^{ak^{-1/2}y}$ by \eqref{pee}. Hence, $\sum_y \mathfrak p_n^n(x,y)\|z_0^k(y)\|_p$ may be bounded above by $Ce^{ak^{-1/2}x+Ka^2k^{-1}n}$, by Proposition \ref{macky}. After this, we set $x_0:=x$ and $i_0:=0$ and we iterate \eqref{budd}. Then we get \begin{align} \notag \|Z_k(n,x)\|_p^2 &\leq C \sum_{r=0}^{n} k^{-r/2} \sum_{\substack{0\le i_1<...<i_r <n\\(x_1,...,x_r)\in\Bbb Z_{\ge 0}}} \prod_{j=1}^r \mathfrak p_{n-i_{j-1}}^{n-i_j}(x_{i-1},x_i)^2 \cdot e^{ak^{-1/2}x_r + Ka^2n/k} \\ \label{dap}&\stackrel{\text{Lemma }\ref{killme}}{\leq} Ce^{ak^{-1/2}x+Ka^2n/k}\sum_{r=0}^n C^k k^{-r/2} n^{r/2}/(r/2)! \notag \\ &\leq Ce^{ak^{-1/2}x +Bn/k},
\end{align} where $B$ is a large constant. Now replace $x$ by $n^{1/2}X$, $n$ by $nT$, and $k$ by $n$. This will give $\|\mathscr Z_n(T,X)\|_p^2 \leq Ce^{aX+BT}$. But $e^{aX+BT}$ can be bounded from above on any compact set, proving \eqref{unif}.
\\
\\
Now we will prove \eqref{spatial}. By applying Burkholder-Davis-Gundy (twice) in the same way which was used in proving \eqref{budd}, one sees that 
\begin{align}\notag \|Z_k(n,x)-Z_k(n,y)\|_p^2 \leq C&\bigg\|\sum_{w \ge 0}\big( \mathfrak p_n^n(x,w)-\mathfrak p_n^n(y,w)\big)z_0^k(w)\bigg\|_p^2\\\label{fum}&+ Ck^{-1/2}\sum_{i=0}^{n-1} \sum_{w \ge 0} \big(\mathfrak p_{n-i}^n(x,w)-\mathfrak p_{n-i}^n(y,w)\big)^2 \|Z_k(i,w)\|_p^2.
\end{align}
We will bound the first term using the coupling lemma. Specifically, let $P$ (and its expectation operator $E$) denote a coupling of $\mathbf E_x^n$ and $\mathbf E_y^n$ as in Proposition \ref{coup}, and let $(S^x,S^y)$ be the associated coordinate process. Recall from Proposition \ref{nun} that $\Bbb E[(z_0^k(x)-z_0^k(y))^4] \leq Ck^{-1}|x-y|^2e^{ak^{-1/2}(x+y)}$ for some constants $C,a$ independent of $n,x,y$. Then by independence of $z_0^k$ and $S$, one may apply Minkowski and Jensen to commute the respective expectations and obtain
\begin{align*}
    \bigg\|\sum_{w \ge 0}\big( \mathfrak p_n^n(x,w)-\mathfrak p_n^n(y,w)\big)z_0^k(w)\bigg\|_p^2 &= \big\| \mathbf E_x^n[z_0^k(S_n)] - \mathbf E_y^n[z_0^k(S_n)]\big\|_p^2 \\ = \big\| E[z_0^k(S_n^x) - z_0^k(S_n^y)]\big\|_p^2 &\leq E \big[ \big\|z_0^k(S_n^x) - z_0^k(S_n^y)\|_p^2 \big] \\ \leq C E \big[ k^{-1/2} |S_n^x-S_n^y| e^{ak^{-1/2}(S_n^x+S_n^y)}] &\le Ck^{-1/2}|x-y| \mathbf E_x^n[e^{2ak^{-1/2}S_n}]^{1/2} \mathbf E_y^n[e^{2ak^{-1/2}S_n}]^{1/2} \\ \stackrel{\text{Prop. } \ref{macky}}{\le} & Ck^{-1/2}|x-y| e^{ak^{-1/2}(x+y)},
\end{align*}
where we noted that $e^c+e^d\leq 2e^{c+d}.$ Next, we geometrically interpolate (i.e., $c \wedge d \leq c^{\theta}d^{1-\theta}$ for $\theta \in [0,1]$) between the bound of Proposition \ref{fetiz} and that of \eqref{spat1} (with $p=2$ for both). This will yield the following for all $\alpha \ge 0$:
\begin{equation}\label{geo}
    \sum_{z \ge 0} \big( \mathfrak p_n^N(x,z) - \mathfrak p_n^N(y,z) \big)^{2} e^{\alpha z} \leq C e^{\alpha (x+y)+K\alpha ^2n} (n^{-\frac12-\frac12\theta}+\alpha^{\theta} n^{-\frac12}) |x-y|^{\theta}
\end{equation}
Using these bounds and using equation \eqref{fum} in macroscopic coordinates, we will obtain:
\begin{align}\notag
    &\;\;\;\;\;\|Z_k(n,x) -  Z_k(n,y)\|_p^2 \\\notag &\le Ck^{-1/2}|x-y| e^{ak^{-1/2}(x+y)}+Ck^{-1/2}\sum_{i=0}^{n-1} \sum_{w \ge 0} \big(\mathfrak p_{n-i}^n(x,w) -\mathfrak p_{n-i}^n(y,w)\big)^2 \|Z_k(i,w)\|_p^2 \\\notag &\stackrel{\eqref{dap}}{\le} Ck^{-1/2}|x-y| e^{ak^{-1/2}(x+y)}+Ck^{-1/2}\sum_{i=1}^{n} \sum_{w \ge 0} \big(\mathfrak p_{n-i}^n(x,w) -\mathfrak p_{n-i}^n(y,w)\big)^2 Ce^{ak^{-1/2}w + Bi/k} \\\notag &\stackrel{\eqref{geo}}{\leq} Ck^{-1/2}|x-y| e^{ak^{-1/2}(x+y)}\\ \notag &\;\;\;\;\;\;\;\;\;\;\;\;\;\;\;\;\;\;\;\;\;\;+Ck^{-1/2}\sum_{i=1}^{n} e^{ak^{-1/2}(x+y)}\big[(n-i)^{-\frac12 -\frac12 \theta} +a^{\theta}k^{-\theta/2} (n-i)^{-1/2}\big]|x-y|^{\theta} e^{Bn/k}\\ \label{darts} &\leq Ck^{-1/2}|x-y| e^{ak^{-1/2}(x+y)}+C(n^{\frac12-\frac12\theta}+k^{-\theta/2} n^{1/2})k^{-1/2}e^{ak^{-1/2}(x+y)} |x-y|^{\theta}e^{Bn/k}.
\end{align}
In the last line, we used the bound $\sum_i (n-i)^{-\frac12-\frac12\theta} \leq Cn^{\frac12-\frac12\theta}$ for $\theta<1$. Now we convert to macroscopic coordinates ($k \to n$; $n \to nT$; $x \to n^{1/2}X;y\to n^{1/2}Y$), to get 
$$\|\mathscr Z_n(T,X)-\mathscr Z_n(T,Y)\|_p^2 \leq Ce^{2a(X+Y)} \big(|X-Y| + (T^{\frac12-\frac12\theta}+T^{1/2})|X-Y|^{\theta}\big)e^{BT}.$$ On any compact set $|X-Y|$ may be bounded by $C|X-Y|^{\theta}$ (since $\theta<1$).Similarly, we can also absorb $(1+T^{\frac12-\frac12\theta}+T^{1/2})e^{2a(X+Y)+BT}$ into the constant, proving \eqref{spatial}.
\\
\\
Now we will prove \eqref{temporal}. Let $m\le n$. For this, one writes $$Z_k(n,x) = \sum_{y \ge 0} \mathfrak p_{n-m}^{n}(x,y) Z_k(m,y)+k^{-1/4}\sum_{i=1}^{n-m} \sum_{y \ge 0} \mathfrak p^{n}_{n-m-i}(x,y) Z_k(i+m,y)\omega^k_{(i+m)y}.$$ Again imitating the proof of \eqref{budd} and using the fact that $\mathfrak p_{n-m}^n(x,\cdot)$ is a probability measure (then applying Jensen), one sees 
\begin{align*}\|Z_k(n,x)-Z_k(m,x)\|_p^2 \le C&\sum_{y \ge 0} \mathfrak p^n_{n-m}(x,y) \|Z_k(m,y)-Z_k(m,x)\|_p^2 \\ &+Ck^{-1/2} \sum_{i=1}^{n-m}\sum_{y \ge 0} \mathfrak p_{n-m-i}^n(x,y)^2 \|Z_k(i+m,y)\|_p^2.
\end{align*}
Let us call the sums on the right side $S_1(m,n,k,x),S_2(m,n,k,x)$, respectively. We bound these separately. We first compute that 
\begin{align*}
    &\sum_{y \ge 0} \mathfrak p_n^N(x,y) |x-y|^{\theta} e^{a(x+y)} = \mathbf E_x^N[|S_n-x|^{\theta}e^{a(S_n+x)}] \\ &\le \mathbf E_x^N[|S_n-x|^{2\theta}]^{1/2} \mathbf E_x^N[e^{2a(S_n+x)}]^{1/2} \leq Cn^{\theta/2} e^{2ax+Ka^2n},
\end{align*}
where the last inequality follows from Propositions \ref{pmoments} and \ref{macky}. Using this and \eqref{darts} we see that \begin{align*}S_1 &\leq C \sum_{y \ge 0}\mathfrak p_{n-m}^n(x,y) \cdot \big[ k^{-1/2}|x-y| e^{ak^{-1/2}(x+y)}\\&\;\;\;\;\;\;\;\;\;\;\;\;\;\;\;\;\;\;\;\;\;\;\;\;+C(m^{\frac12-\frac12\theta}+k^{-\theta/2} m^{1/2})k^{-1/2}e^{ak^{-1/2}(x+y)} |x-y|^{\theta}e^{Bm/k}\big] \\ &\leq Ck^{-1/2}(n-m)^{1/2}e^{2ak^{-1/2}x+Ka^2n/k} \\&\;\;\;\;\;\;\;\;\;\;\;\;\;\;\;\;\;\;\;\;\;\;\;\;+ C(m^{\frac12-\frac12\theta}+k^{-\theta/2} m^{1/2})k^{-1/2}e^{Bm/k} (n-m)^{\theta/2} e^{2ak^{-1/2}x+Ka^2n/k}.
\end{align*}
Next, to bound $S_2$, we are going to use \eqref{dap} with Proposition \eqref{macky} and we obtain
\begin{align*}
    S_2 &\le Ck^{-1/4} \sum_{i=1}^{n-m}\sum_{y \ge 0} \mathfrak p_{n-m-i}^n(x,y)^2 e^{ak^{-1/2}y+B(i+m)/k} \\&\leq Ck^{-1/2} \sum_{i=1}^{n-m} (n-m-i)^{-1/2} e^{ak^{-1/2}x +Bn/k} \\ &\le Ck^{-1/2}(n-m)^{1/2} e^{ak^{-1/2}x+Bn/k}.
\end{align*}
Combining the bounds for $S_1,S_2$ and then converting to macroscopic coordinates $(n \to nT; m\to nS; k\to n; x \to n^{-1/2}X)$ will yield the following bound:
$$\|\mathscr Z_n(T,X) -\mathscr Z_n(S,X)\|_p^2 \leq C(|T-S|^{1/2}+ (S^{\frac12-\frac12\theta} +S^{1/2}) |T-S|^{\theta/2}\big) e^{2aX+B'T}, $$ where $B'$ is a large constant depending on $a^2$ and $B$. Since $|T-S|^{1/2}\leq C|T-S|^{\theta/2}$ on compact sets and since $(1+S^{\frac12-\frac12\theta} +S^{1/2})e^{2aX+B'T}$ may be bounded from above on compact sets, this finishes the proof of \eqref{temporal}.
\\
\\
Now we need to argue tightness from these estimates. This is a direct corollary of the Kolmogorov continuity criterion (two-parameter version), Prokhorov's theorem, and the Arzela-Ascoli Theorem.
\end{proof}

Now that we proved tightness, we only need to obtain convergence of finite-dimensional marginals of $\mathcal Z_n$ to those of SPDE \eqref{spde}. Thanks to the Cramer-Wold device (and linearity of integration with respect to space-time white noise) this will not be any more difficult than just proving convergence of \textit{one-point} marginals. This can actually be done by using the convergence result (Proposition \ref{gence}) together with the machinery developed in the papers [AKQ14a, CSZ17a].
\\
\\
Specifically, we will use [CSZ17a, Theorem 2.3], which in turn was inspired by the results of [AKQ14a, Section 4]. We state this result in a version which is adapted to our own context. Throughout, we will fix $T>0$ and we will denote $\Delta_k(T):= \{(t_1,...,t_k):0<t_1<...<t_k<Tn,t_i \in \Bbb R\}$. Also denote by $\Delta_k^n(T):= \{(\frac{t_1}n,...,\frac{t_k}n):0<t_1<...<t_k<Tn,t_i \in \Bbb Z\}$, and let $(\Bbb R^d)_n:= (n^{-1/2}\Bbb Z)^d.$ Then define $$\mathcal L^n_k:= \Delta^n_k(T) \times (\Bbb R^k)_n,$$ and we equip $\mathcal L^n_k$ with $\sigma$-finite the measure which assigns mass $n^{-3/2} = n^{-1}\cdot n^{-1/2}$ to each distinct space-time point $(\frac{t}{n},\frac{x}{\sqrt{n}}).$ We denote by $L^2(\mathcal L^n_k)$ the $L^2$ space associated to this measure.

\begin{thm}[CSZ17a, Theorem 2.3]\label{convo} For each $n \in \Bbb N$, let $\{\omega^n_{i,j}\}_{i,j\ge 0}$ be a family of random weights with mean zero and $var(\omega^n_{i,j}) = \sigma^2 +o(1)$ (as $n \to \infty$). Let $\{F_k^n\}_{n,k \in \Bbb N}$ be a family of functions, defined on $\mathcal L^n_k$. Suppose that $F_k:\Delta_k(T) \times \Bbb R^k\to \Bbb R$ be a family of continuous functions such that that $\|F_k^n-F_k\|_{L^2(\mathcal L^n_k)} \to 0$ as $n \to \infty$, for every $k \in \Bbb N$. Furthermore, assume that $$\sup_n\sum_{k \ge 0} \|F_k^n\|^2_{L^2(\mathcal L^n_k)}<\infty.$$ Then define random variables $$X_n:= \sum_{k \ge 0} n^{-3k/4} \sum_{(\vec{t},\vec{x})\in \mathcal L^n_k} F^n_k(\vec t, \vec x) \omega_{i_1x_1}\cdots \omega_{i_kx_k}.$$ Then $X_n$ converges in distribution as $n \to \infty$ to the random variable $$\int_{\Delta_k(T)} \int_{\Bbb R_+^k} F_k(t_1,...,t_k;x_1,...,x_k)\xi(dx_1dt_1)\cdots \xi(dx_ndt_n),$$ where $\xi$ is a space-time white noise on $\Bbb R_+\times \Bbb R$.
\end{thm}

With this in place, we are now ready to prove Theorem \ref{thm2}.

\begin{proof}[Proof of Theorem \ref{thm2}] Using the discussion at the end of Section 2, we know that $z_0^n(n^{1/2}X)$ converges to a geometric Brownian motion with drift, specifically $e^{B_X-(A+1/2)X}.$ We exploit Skorohod's lemma to couple all of the $z_0^n$ onto the same probability space in such a way so that this convergence occurs almost surely.
\\
\\
Fix $x,t>0$. In our case, we set $$F^n_k(t_1,...,t_k;x_1,...,x_k):= \sum_{x_k \in n^{-1/2}\Bbb Z_{\ge 0}}z_0^n(n^{1/2}x_k)\prod_{i=1}^n \mathscr P_n(t_j-t_{j-1},T-t_{j-1}; x_{k-1};x_k),$$ $$F_k(t_1,...,t_k;x_1,...,x_k):= \int_{\Bbb R_+}e^{B_{x_k}-(A+1/2)x_k}\prod_{i=1}^n \mathscr P(t_j-t_{j-1},T-t_{j-1}; x_{k-1};x_k)\;dx_k,$$ where $\mathscr P_n$ was defined in Proposition \ref{gence} and where $(x_0,t_0):=(x,t)$. The condition that $$\sup_n\sum_{k \ge 0} \|F_k^n\|^2_{L^2(\mathcal L^n_k)}<\infty,$$ follows quite simply from Lemma \ref{killme}. Also the condition that $\|F_k^n-F_k\|_{L^2(\mathcal L^n_k)} \to 0$ as $n \to \infty,$ follows by inducting on the last statement in Proposition \ref{gence}.
\\
\\
By Theorem \ref{convo}, we conclude that the one-point marginals of $\mathscr Z_n$ converge to those of the solution of \eqref{spde}. The proof for multi-point is similar, but one defines a new family $\tilde F_k^n$ by taking linear combinations of the $F_k^n$ which are defined above, then one applies the Cramer-Wold device to make the conclusion.
\\
\\
The only thing which has not been explained is the normalization $\big(2\Phi\big(\frac{X+n^{-1/2}}{\sqrt{T}}\big)-1\big)^{-1}$ which appears in Theorem \ref{thm2}. This may be viewed as a simple consequence of the fact that (by the local central limit theorem), the asymptotic mass of the measures $\mu_{n^{1/2}X}^{nT}$ appearing in Theorem \ref{rw} is equal to $2\Phi\big(\frac{X+n^{-1/2}}{\sqrt{T}}\big)-1+o(n^{-1/2}).$
\end{proof}

\begin{appendix}

\section{A priori estimates and concentration of measure}

The purpose of this appendix is to gather estimates for the simple symmetric random conditioned to stay positive. The results are somewhat standard and the literature on such measures is extensive \cite{Ig74, Bol76, Car05, CC08, DIM77} etc., but we will only give a brief exposition of those selected estimates which apply to our result in the nearest-neighbor case (which we could not find in the above references). For completeness, we provide elementary proofs which are specialized to our particular case of nearest-neighbor jumps, but some of the results below have generalizations which can be found in those references. 
\\
\\
The main goal of this appendix will be to prove a powerful concentration inequality for the positive random walk measures $\mathbf P_x^n$ defined above, more specifically, we will show that $$\mathbf P_x^n\big(\sup_k |S_k-x|>u\big) \leq Ce^{-cu^2/n},$$ where $C,c$ are independent of $n,x$. This will in turn allow us to prove various $L^p$ moment bounds (to be used later in Section 5) and derive a Donsker principle. The methods used in proving these results will be coupling arguments and martingale techniques, many of which will be very useful in and of themselves. More specifically, the main key will be to notice that for fixed $n \in \Bbb N$, the process $$M_k^n:= \frac{S_k+1}{\psi(S_k,n-k)},\;\;\;\;\;\;\;\;\;\;\;\; 0\leq k \leq n,$$ is a $\mathbf{P}_x^n$-martingale with respect to the $k$-variable. Furthermore, we will see that it has bounded increments. First we state a few preliminary lemmas.

\begin{lem}\label{ohgod}
	Let $b \geq 0$. There exists a constant $C=C(b)>0$ such that for all $n \ge 0$ and all $x,y,z \geq 0$ one has $$p_n^{(1/2)}(x,y) \leq C\bigg[\frac1{\sqrt{n+1}}\wedge \frac{x+1}{n+1}\bigg]e^{-bn^{-1/2}|x-y|}.\;\; $$
	$$\big|p_n^{(1/2)}(x,y)-p_n^{(1/2)}(x,z)\big| \leq C\bigg[\frac{1}{n+1} \wedge \frac{x+1}{(n+1)^{3/2}}\bigg]|z-y|e^{-bn^{-1/2}\big(|x-y|\wedge|x-z|\big)}. $$ 
\end{lem}

\begin{proof}
	The proof given here is inspired by the methods of [DT16, Appendix A]. 
	\\
	\\
	Let us start with the first bound. Note that it suffices to prove the bound for $y \ge x$. Indeed, if $x \ge y$, then by symmetry one has that $p_n^{(1/2)}(x,y) = p_n^{(1/2)}(y,x) \leq C\frac{y+1}{n+1} e^{-b|x-y|n^{-1/2}} \leq C\frac{x+1}{n+1}e^{-b|x-y|n^{-1/2}}$.
	\\
	\\
	Let $p_n(x)$ denote the standard discrete heat kernel on $\Bbb Z$. One first notes that for $z \in \Bbb C$ one has that $\sum_{x \in \Bbb Z} p_n(x) z^x = 2^{-n}(z+z^{-1})^n.$ Letting $C$ denote the unit circle oriented counterclockwise, Cauchy's integral formula says $$p_n(x) = \frac1{2\pi i } \oint_C z^{-x-1} 2^{-n}(z+z^{-1})^ndz.$$ Since the integrand is analytic away from the origin, one may deform the contour without changing the value. Specifically, we will expand the radius of the circle to $e^{bn^{-1/2}}$. Parametrizing this as $z = e^{bn^{-1/2}}e^{it}$, one finds that \begin{equation}\label{12}p_n(x) = \frac{1}{2\pi} \int_{-\pi}^{\pi} e^{-xbn^{-1/2}} e^{-itx} \bigg[ \cosh(b/\sqrt{n})\cos t + i \sinh(b/\sqrt{n})\sin t\bigg]^ndt.\end{equation} Now, by Taylor expanding $\sinh$ and $\cosh$, the key observation is that 
	\begin{align*}
	f_n(t):&=\bigg| \cosh(b/\sqrt{n})\cos t + i \sinh(b/\sqrt{n})\sin t\bigg|^n \\ &= \bigg[ \cosh^2(b/\sqrt{n})\cos^2 t + \sinh^2(b/\sqrt{n}) \sin^2 t \bigg]^{n/2} \\ &= \bigg[ \big(1+b^2/(2n) +O(n^{-2})\big) \cos^2t + \big( b^2/n + O(n^{-2}) \big)\sin^2 t\bigg]^{n/2} \\ &\le \bigg[ \cos^2 t + \frac{b^2}{n} +O(n^{-2}) \bigg]^{n/2},
	\end{align*}
	where the $O(n^{-2})$ terms denote quantities which are uniformly bounded above by $e^b b^4 n^{-2})$. Since $n^{-2}$ decays faster than $n^{-1}$, this means that the last expression is bounded above by $\big[ \cos^2 t +Cn^{-1} \big]^{n/2}$, where $C=C(b)$. Now, for $t \in [-\pi/2, \pi/2]$, it holds that \begin{equation}\label{boss}\big[ \cos^2 t +Cn^{-1} \big]^{n/2} \leq \begin{cases}
	\big[ 1+C/n \big]^n \leq e^C,& |t| \leq C/\sqrt{n} \\ 2^{-n}, & |t| \geq C/\sqrt{n} \end{cases}.\end{equation}
	Indeed, these bounds follow from the observation that $\cos t$ looks like $1-t^2/2$ near $t=0$, and therefore the quantity on the left side decays exponentially fast (as $n \to \infty$) uniformly outside of a $Cn^{-1/2}$-window of the origin ($C$ may need to be large).
	\\
	\\
	Because the left side of \eqref{boss} is an upper bound for $f_n(t)$, it easily follows from \eqref{boss} that $\int_{-\pi/2}^{\pi/2}f_n(t) dt \leq Cn^{-1/2}$ and $\int_{-\pi/2}^{\pi/2} |t|f_n(t) \leq Cn^{-1}$. With this in mind, we compute via \eqref{12} that \begin{equation}p_n(y-x)-p_n(y+x+2) \leq \frac1{2\pi} \int_{-\pi}^{\pi} \big| e^{-(y-x)(bn^{-1/2} + it)} - e^{-(y+x+2)(bn^{-1/2} + it)}\big| f_n(t)dt.\label{veg} \end{equation} Thanks to the absolute value and the trigonometric nature of the integrand, we may replace the integral over $[-\pi,\pi]$ with twice the integral over $[-\pi/2,\pi/2]$. Furthermore, using $|e^{is}-1| \leq |s|$ one computes that 
	\begin{align}\notag\big| e^{-(y-x)(bn^{-1/2} + it)} &- e^{-(y+x+2)(bn^{-1/2} + it)}\big|  = \big| e^{-(y-x)bn^{-1/2}} e^{-(y-x)it} \big( 1- e^{-2(x+1)(bn^{-1/2}+it)} \big) \big| \\ \notag & = e^{-(y-x) bn^{-1/2}} \big| 1- e^{-2(x+1)(bn^{-1/2}+it)} \big| \\ \notag &\leq e^{-(y-x) bn^{-1/2}} \big| 1- e^{-2(x+1)bn^{-1/2}} \big| + e^{-(y+x+2)bn^{-1/2}} \big| 1-e^{-2(x+1)it} \big| \\ \label{vog} &\leq 2be^{-(y-x)bn^{-1/2}} \frac{x+1}{n^{1/2}} + 2e^{-(y+x+2)bn^{-1/2}} (x+1)|t|. \end{align}
	Now we note that $e^{-(y+x+2)bn^{-1/2}} \leq e^{-(y-x)bn^{-1/2}}$ since $x,y \ge 0$. Combining \eqref{veg} and \eqref{vog}, together with the fact that $\int_{-\pi/2}^{\pi/2} f_n(t) dt \leq Cn^{-1/2}$ and $\int_{-\pi/2}^{\pi/2} |t|f_n(t) dt \leq Cn^{-1}$, proves that for $y \geq x$ one has $p_n^{(1/2)}(x,y) \leq C\frac{x+1}{n+1} e^{-bn^{-1/2}(y-x)},$ as desired. In order to obtain the other bound $p_n^{(1/2)}(x,y) \le \frac{1}{\sqrt{n+1}}e^{-bn^{-1/2}|x-y|}$, one replaces \eqref{vog} with the easier bound $$\big| e^{-(y-x)(bn^{-1/2} + it)} - e^{-(y+x+2)(bn^{-1/2} + it)}\big| \leq 2e^{-bn^{1/2}|x-y|},$$ where we used the triangle inequality and the fact that $|e^{is}|\le 1$. Then one uses \eqref{veg} and the fact that $\int f_n \leq Cn^{-1/2}$. This completes the proof of the first bound.
	\\
	\\
	Now we prove the second bound stated in the proposition. For this, one uses the same arguments, but one needs to replace \eqref{vog} with the approptiate bound. Specifically, we need to consider $$ e^{-(y-x)(bn^{-1/2} + it)} - e^{-(y+x+2)(bn^{-1/2} + it)} - \big( e^{-(z-x)(bn^{-1/2} + it)} - e^{-(z+x+2)(bn^{-1/2} + it)}\big).$$ We write this as $$g(y-x) - g(y+x+2) - g(z-x) +g(z+x+2) = \int_{y-x}^{y+x+2} \int_{v}^{v+z-y} g''(u)dudv,$$ where $g(u) = e^{-u(bn^{-1/2}+it)}.$ So one computes $g''(u) = (bn^{-1/2}+it)^2g(u).$ Now, for $u$ in the relevant range, it is clear that $|g(u)| = e^{-bn^{-1/2}u} \leq e^{-bn^{-1/2} (|z-x| \wedge |y-x|)}.$ Hence $$|g''(u)| = |bn^{-1/2}+it|^2 |g(u)| \leq \big( 2b^2n^{-1} + 2t^2 \big) e^{-bn^{-1/2}(|z-x|\wedge |y-x|)},$$ where we used $|p+q|^2 \leq 2|p|^2+2|q|^2$. Combining the previous two expressions, we find that $$\big|g(y-x) - g(y+x+2) - g(z-x) +g(z+x+2)\big| $$$$\leq 2(x+1)|z-y|\big( 2b^2n^{-1} + 2t^2 \big) e^{-bn^{-1/2}(|z-x|\wedge |y-x|)}.$$ Now multiplying by $f_n(t)$ and integrating over $[-\pi,\pi]$, we finally obtain $$|p_n^{(1/2)}(x,y) - p_n^{(1/2)}(x,z) | \leq C(b)(x+1)|z-y|e^{-bn^{-1/2}(|z-x|\wedge |y-x|)}\int_{-\pi}^{\pi} (n^{-1}+t^2)f_n(t)dt$$$$\leq C(x+1)|y-z|(n+1)^{-3/2},$$ where we use the bound \eqref{boss} for $f_n(t)$ in the last inequality. This already proves one part of the second bound, namely $|p_n^{(1/2)}(x,y) -p_n^{(1/2)}(x,z)|\leq \frac{C|y-z|(x+1)}{(n+1)^{3/2}}e^{-bn^{-1/2}(|z-x|\wedge |y-x|)}$. For the other bound $p_n^{(1/2)}(x,y) \leq \frac{C|y-z|}{n+1}e^{-bn^{-1/2}(|z-x|\wedge |y-x|)}$, we simply note that $$|g(y-x) - g(y+x+2) - g(z-x) +g(z+x+2)|$$$$ \leq |g(y-x)-g(z-x)|+|g(y+x+2)-g(z+x+2)|$$ $$\leq \int_{y-x}^{z-x} |g'(u)|du+ \int_{y+x+2}^{z+x+2}|g'(u)|du,$$ and then we apply similar arguments as before, noting $|g'(u)| \leq |bn^{-1/2}+it||g(u)|$.
\end{proof}

\begin{lem}\label{mass}
	Let $\psi(x,N)$ be as in Definition \ref{fuk}. Then there exists a constant $C>0$ such that for all $x,N \geq 0$ one has $$ \frac{x+1}{x+1+C\sqrt{N}}\leq \psi(x,N) \leq 1 \wedge \bigg( \frac{C(x+1)}{\sqrt{N}}\bigg).$$ Furthermore, for each $x \geq 0$ one has that $$\lim_{N \to \infty} \sqrt{N} \psi(x,N) = (x+1) \sqrt{2/\pi}.$$
\end{lem}

Note that this already proves Theorem \ref{rw}(3). Furthermore, note that the upper and lower bound on $\psi$ is strong enough to give an upper and lower envelope on $\psi$, i.e., $$ C^{-1}\frac{x+1}{x+1+\sqrt{N}}\leq \psi(x,N) \leq C \frac{x+1}{x+1+\sqrt{N}}.$$This is because $1\wedge w \leq \frac{2w}{1+w}.$ We now proceed to the proof.

\begin{proof} First we prove the upper bound. Let $p_N$ denote the standard heat kernel on the whole line $\Bbb Z$. Since $p_N$ is symmetric and sums to $1$, it holds that  \begin{align*}\psi(x,N)\stackrel{\text{def}}{=} \sum_{y \geq 0}\big(p_N(x-y)-p_N(x+y+2)\big) =p_N(x+1)+p_N(0) +2\sum_{1\leq u\leq x} p_N(u).\end{align*} Now, we use the simple bound $p_N \leq CN^{-1/2}$ to see that the right side of the last expression is bounded above by $2C(x+1)N^{-1/2}.$ On the other hand, it is obvious that $\psi(x,N) \leq 1$ for all $x,N$. So, we obtained the desired upper bound.
\\
\\
Next, we prove the lower bound. We consider two different cases: $x \leq \sqrt{N}$ and $x > \sqrt{N}$. 
\\
\\
First we consider the case $x > \sqrt{N}$. One may apply Hoeffding's inequality for the simple random walk to deduce that \begin{align*}\psi(x,N) &=p_N(x+1)+p_N(0)+2\sum_{1\le u\le x} p_N(u)   \geq \sum_{-x \leq u \leq x} p_N(u)  \geq 1-2e^{-2(x+1)^2/N} .\end{align*} Now set $q:= \frac{2(x+1)^2}{N}$. Then $q \geq 2$, so $2(q+2)\leq 2e^q$, and thus $\frac{1}{ 1-2e^{-q}} \leq 1+\frac2q.$ This means that $\psi(x,N)^{-1} \leq 1+ \frac{2N}{(x+1)^2}$. But since $x+1 \geq \sqrt{N}$, it follows that $\frac{N}{(x+1)^2} \leq \frac{\sqrt{N}}{x+1}$. Hence we obtain $\psi(x,N) \geq \frac{x+1}{x+1+2\sqrt{N}}$, whenever $x>\sqrt{N}$.
\\
\\
Now we consider the case $ x \leq \sqrt{N}$. For $u \leq x$, the local central limit theorem tells us that $p_N(u) \geq \frac{c}{\sqrt{N}} e^{-2u^2/n} \geq \frac{c}{\sqrt{N}} e^{-2},$ and hence $$\psi(x,N) =p_N(x+1)+p_N(0)+2\sum_{1\le u\le x} p_N(u)  \geq \sum_{0\leq u \leq x} p_N(u) \geq \frac{ce^{-2}}{\sqrt{N}}(x+1).$$ Now one simply notes that $\frac{ce^{-2}}{\sqrt{N}} \geq \frac{1}{x+1+c^{-1}e^2\sqrt{N}}$. This completes the proof of the lower bound.
	\\
	\\
	Finally, we prove the last statement about the limit. For this, let us write $$\psi(x,N) =p_N(x+1)+p_N(0)+2\sum_{1\le u\le x} p_N(u) $$ The local limit theorem tells us that for each $u$, the quantity $\sqrt{N}p_N(u)$ oscillates back and forth between $\sqrt{2/\pi}$ and zero (depending on the parity of $N$) as $N$ becomes large. This already implies that $N^{1/2}$ times the right side converges to $\big( 1 +x \big) \sqrt{2/\pi}.$
\end{proof}

\begin{lem}\label{ball}
	Let $(a_x)_{x \ge 0}$ be a sequence of non-negative numbers such that $$a_x \leq a_{x+1} ,\;\;\;\;\;\;\;a_{x+2}-a_{x+1} \leq a_{x+1}-a_{x}, \;\;\;\;\;\;\; \text{ for all }x \ge 0. $$ Then for all $x \ge 0$ and $k \ge 0$, one has that $$ \frac{a_x}{a_{x+k}} \leq \frac{a_{x+1}}{a_{x+k+1}}.$$
\end{lem}

\begin{proof}
	It suffices to prove the claim when $k=1$, because then one has that $$ \frac{a_x}{a_{x+k}} = \prod_{j=0}^{k-1} \frac{a_{x+j}}{a_{x+j+1}} \leq \prod_{j=0}^{k-1} \frac{a_{x+j+1}}{a_{x+j+2}} = \frac{a_{x+1}}{a_{x+k+1}}.$$ To prove the claim for $k=1$, one uses the mean value theorem to extract $u \in [a_x,a_{x+1}]$ and $v \in [a_{x+1},a_{x+2}]$ such that $$\log a_{x+1} -\log a_x = \frac{1}{u} (a_{x+1}-a_x) ,\;\;\;\;\;\;\; \log a_{x+2} - \log a_{x+1} = \frac1v (a_{x+2}-a_{x+1}).$$ Then clearly $\frac1v \leq \frac1u$, and by hypothesis, it is also true that $a_{x+2}-a_{x+1} \leq a_{x+1}-a_{x}.$ So we conclude that $\log a_{x+2} - \log a_{x+1} \leq \log a_{x+1} -\log a_x$.
\end{proof}

\begin{lem}[Monotonicity]\label{mono} Fix $n \in \Bbb N$. Then $\psi(x,n)$ is an increasing function of $x$. Thus, $\mathfrak p_1^n(x,x+1)\ge 1/2 \ge \mathfrak p_1^n(x,x-1)$ for all $x,n \ge 0$. Furthermore, $\mathfrak p_1^n(x,x+1)$ is a decreasing function of $x$, and $\mathfrak p_1^n(x,x-1)$ is an increasing function of $x$.
\end{lem}

\begin{proof}
	As in the proof of Lemma \ref{mass}, we write $$\psi(x,n) = p_n(0)+p_n(x+1) + 2 \sum_{y=1}^x p_n(y).$$ Consequently, it holds that \begin{equation}
	\label{bib}\psi(x+1,n)-\psi(x,n) = p_n(x+1)+p_n(x+2),\end{equation} and the right side is clearly non-negative, which proves the first statement. For the second statement, we just note that $$\mathfrak p_1^n(x,x+1) = \frac{\psi(x+1,n-1)}{2\psi(x,n)} \ge \frac{\psi(x-1,n-1)}{2\psi(x,n)} = \mathfrak p_1^n(x,x-1).$$ To prove the final statement, we note that $p_n$ is a non-increasing function of $|x|$, and thus the right side of \eqref{bib} is also a non-increasing function of $x$. Thus we may apply Lemma \ref{ball} with $k=2$ and $a_x = \psi(x,n)$, to conclude that $\frac{\psi(x,n)}{\psi(x+2,n)}$ is an increasing function of $x$. Now we write $$ \frac{1}{\mathfrak p_1^n(x,x+1)} =2 \frac{\psi(x,n)}{\psi(x+1,n-1)} = 1 + \frac{\psi(x-1,n-1)}{\psi(x+1,n-1)},$$ where we use the relation $\psi(x,n) = \frac12\psi(x+1,n-1)+\frac12 \psi(x-1,n-1). $ By the discussion of the previous paragraph, the right side is an increasing function of $x$, and so $\mathfrak p_1^n(x,x+1)$ is a decreasing function of $x$. Finally, this implies that $\mathfrak p_1^n(x,x-1)=1-\mathfrak p_1^n(x,x+1)$ is an increasing function of $x$.
\end{proof}

\begin{prop}[Coupling Lemma for Positive Walks]\label{coup} Fix $n \in \Bbb N$ and $x \geq 0$. There exists a coupling $\mathbf Q_{x,x+1}^n$ of the measures $\mathbf P_x^n$ and $\mathbf P_{x+1}^n$ which is supported on pairs $(s,s')$ of paths such that $|s_{i+1}-s_i|=1$ for all $i\leq n$. In other words, the positive walks started from $x$ and $x+1$ may be coupled in such a way so that their distance from each other is never greater than $1$.
\\
\\
More generally, for fixed $n \in \Bbb N$, the measures $\{\mathbf P_x^n\}_{x \ge 0}$ may all be coupled together in such a way that the coordinate processes associated to neighboring values of $x$ are never more than distance $1$ from each other.
\end{prop}

The reason why this result is non-trivial is because of the lack of spatial and temporal homogeneity of these positive walks. The analogous result for the simple symmetric random walk on $\Bbb Z$ is completely trivial.

\begin{proof}
	Let $\{U_i\}_{i=1}^n$ be a sequence of iid uniform$[0,1]$ random variables. We make an inductive construction as follows. Let $S_0=x$ and $S_0'=x+1$.
	\\
	\\
	Suppose that $S_0,...,S_k$ and $S_0',...,S_k'$ have been constructed in such a way that $|S_i-S_i'|=1$ for all $k$. If $S_k'=S_k+1$, we define $$(S_{k+1},S_{k+1}'):= \begin{cases} (S_k-1,S_k'-1), & U_{k+1} > \mathfrak p_1^{n-k}(S_k,S_k+1) \\ (S_k+1,S_k'-1), & \mathfrak p_1^{n-k}(S_k,S_k+1) > U_{k+1} > \mathfrak p_1^{n-k}(S_k',S_k'+1) \\ (S_k+1,S_k'+1), & U_{k+1} < \mathfrak p_1^{n-k}(S_k',S_k'+1)\end{cases}.$$
	We know by lemma \ref{mono} that one of these cases must hold. Similarly, if $S_k'=S_k-1$, then we define $$(S_{k+1},S_{k+1}'):= \begin{cases} (S_k-1,S_k'-1), & U_{k+1} > \mathfrak p_1^{n-k}(S_k',S_k'+1) \\ (S_k-1,S_k'+1), & \mathfrak p_1^{n-k}(S_k',S_k'+1) > U_{k+1} > \mathfrak p_1^{n-k}(S_k,S_k+1) \\ (S_k+1,S_k'+1), & U_{k+1} < \mathfrak p_1^{n-k}(S_k,S_k+1)\end{cases}.$$
	Lemma \ref{mono} again shows that one of these cases must hold. This completes the inductive step.
	\\
	\\
	A close look at this construction reveals that for $x_1,...,x_n \ge 0$ one has \begin{align*}P(S_1=x_1,S_2=x_2,...,S_n=x_n) &= \mathfrak p_1^n(x,x_1)\prod_{j=1}^{n-1} \mathfrak p_1^{n-j}(x_j,x_{j+1}), \\ P(S_1'=x_1,S_2'=x_2,...,S_n'=x_n) &= \mathfrak p_1(x+1,x_1)\prod_{j=0}^{n-1} \mathfrak p_1^{n-j}(x_j,x_{j+1}). \end{align*}
	By Proposition \ref{pt}, $S$ is distributed as $\mathbf P_x^n$ and $S'$ is distributed as $\mathbf P_{x+1}^n$.
	
The proof of the more general statement is very similar. One simply uses a uniform coupling together with the monotonicity lemma, and the argument is a straightforward generalization of the one given above (for just two values of $x$).
\end{proof}

\begin{prop}[Martingales for Positive Walks]\label{mart} Fix $x,n,k\geq 0,$ with $k \leq n$. Let $S$ be distributed according to $\mathbf P_x^n$. For $i \leq k$ define a function $f(x,i) := \mathbf E_x^{n-i}[S_{k-i}]$. Then the process $$M_i=M_i^{(k,n)}:= f(S_i,i) ,\;\;\;\;\;\;\;\;0\leq i \leq k$$ is a martingale with respect to the natural filtration of $S$. Furthermore, it has bounded increments $$|M_{i+1}-M_i| \leq 2,\;\;\;\;\;\;\;\;0\leq i \leq k-1.$$ In the special case when $k=n$, one has the explicit form $f(x,i) = -1+\frac{x+1}{\psi(x,n-i)}$
\end{prop}

\begin{proof}
	Letting $\mathcal F_k$ denote the natural filtration of $S$, it is a simple consequence of the Markov property that $f(S_i,i) = \mathbf E_x^n[S_k|\mathcal F_i]$, which shows that $M$ is a martingale in the $i$-variable for fixed $x,n,k$.
	\\
	\\
	To prove that it has bounded increments, first note that \begin{align*}f(x,i)-f(x+1,i) = \mathbf E_x^{n-i}[S_{k-i}] -\mathbf E_{x+1}^{n-i}[S_{k-i}].\end{align*} By the coupling lemma (Proposition \ref{coup}), this is bounded in absolute value by $1$. Consequently, one finds that \begin{align*}&|f(x\pm 1,k+1)-f(x,k)| = \bigg| f(x\pm 1,k+1) - \sum_{y \in \{x-1,x+1\}} \mathfrak p_1^{n-k}(x,y)f(k+1,y) \bigg| \\ &\leq \sum_{y \in \{x-1,x+1\}} \mathfrak p_1^{n-k}(x,y) \big| f(x\pm 1,k+1) -f(y,k+1)\big|  = \mathfrak p_1^{n-k}(x,x\mp 1)\cdot 2 \leq 2, \end{align*}
	which clearly implies the desired result.
\end{proof}

We are almost ready to prove our concentration result, we just need one more lemma.

\begin{lem}\label{sqrtgr}
	There exists a constant $C>0$ such that for all $x \geq 0$ and all $n \geq k \geq 1$ one has that $$\mathbf E_x^n[S_k] \leq x+Ck^{1/2}.$$
\end{lem}

\begin{proof} We consider two cases, $k>n/2$ and $k<n/2$.
\\
\\
\textit{Case 1.} $k>n/2$. First, we claim that $\mathbf E_x^n[S_k] \leq \mathbf E_x^n[S_n].$ In fact, it is even true that $S$ forms a $\mathbf P_x^n$-submartingale and thus $\mathbf E_x^n[S_k]$ is an increasing function of $k$ for ever $n$. This follows immediately from Lemma \ref{mono} after noticing that $\mathbf E_x^n[S_{k+1} | \mathcal F_k] = S_k + (2\mathfrak p_1^{n-k}(S_k,S_k+1)-1) \ge S_k$. Now, from the preceding proposition, we know that $M_k:= \frac{S_k+1}{\psi(S_k,n-k)}$ forms a martingale. Thus, we see that $$\mathbf E_0^n[S_n+1] = \mathbf E_x^n[M_0] = \frac{x+1}{\psi(x,n)} \leq x+1+Cn^{1/2},$$ where we applied the lower bound of Lemma \ref{mass} in the final bound. Since $k>n/2$, we see that $n^{1/2} \leq 2^{1/2} k^{1/2}$, which gives the desired bound in this case.
\\
\\
\textit{Case 2.} $k\leq n/2$. First we use the coupling lemma (Proposition \ref{coup}) to see that $\mathbf E_x^n[S_k] \leq 1+\mathbf E_{x-1}^n[S_k].$ Iterating this $x$ times shows that $$\mathbf E_x^n[S_k] \leq x+ \mathbf E_0^n[S_k].$$ Thus we only need to show that $\mathbf E_0^n[S_k] \leq Ck^{1/2}$. To prove this, let us write $\mathbf E_0^n[S_k] = \sum_{y \geq 0} \mathfrak p_k^n(0,y) y.$ Now we write $\mathfrak p_k^n(0,y) = p_k^{(1/2)}(0,y) \frac{\psi(y,n-k)}{\psi(0,n)}$. By Lemma \ref{mass} we know $\frac1{\psi(0,n)} \leq C\sqrt{n}$. Furthermore, we also know from the same lemma that $\psi(y,n-k)$ is bounded above by $1 \wedge (Cy(n-k)^{-1/2}),$ which is in turn bounded above by $1\wedge (Cyn^{-1/2})$ since $k \le n/2$. Moreover, we also know from Lemma \ref{ohgod} that $p_k^{(1/2)}(0,y) \leq \frac{C}{k+1}e^{-y/\sqrt{k}}$. Thus, we find that \begin{equation}\label{sack}\mathbf E_0^n[S_k] \leq \frac{C}{k+1}\bigg[\sum_{0\le y \leq \sqrt{n}} e^{-y/\sqrt{k}} n^{1/2} (n^{-1/2} y^2) + \sum_{y \geq \sqrt{n}} e^{-y/\sqrt{k}} (n^{1/2} y) \bigg].\end{equation}
Let us call the two sums inside the square brackets on the right side as $J_1$ and $J_2$, respectively. 
\\
\\
First we bound $J_1$. Now, we use the bound $\sum_{r \ge 0} r^2 \alpha^r \leq \frac{2}{(1-\alpha)^3}$ (valid for $\alpha <1$) and we see that $$J_1 \leq C \sum_{y \ge 0} y^2 e^{-y/\sqrt{k}} \leq \frac{C}{(1-e^{-1/\sqrt{k}})^3} \leq Ck^{3/2}.$$ In the last bound, we used the elementary bound $(1-e^{-q})^{-1} \le 1+q^{-1}$ (which in turn implies $(1-e^{-q})^{-3} \leq 2^3(1+q^{-3}))$ with $q = k^{-1/2}$.
\\
\\
Next, we bound $J_2$. Using the bound $\sum_{r \ge s} r\alpha^r \leq C\big[\frac{\alpha^s}{(1-\alpha)^2}+\frac{s\alpha^s}{1-\alpha}\big],$ we see that $$J_2= n^{1/2} \sum_{y \geq \sqrt{n}} e^{-y/\sqrt{k}}y \leq n^{1/2} \bigg[\frac{ e^{-\sqrt{n/k}}}{(1-e^{-1/\sqrt{k}})^2} + \frac{n^{1/2}e^{-\sqrt{n/k}}}{1-e^{-1/\sqrt{k}}}\bigg].$$ Now $$n^{1/2}e^{-\sqrt{n/k}} = k^{1/2} (n/k)^{1/2} e^{-\sqrt{n/k}} \leq k^{1/2} \sup_{u>0} ue^{-u} = Ck^{1/2}.$$ Similarly, one finds that $ne^{-\sqrt{n/k}} \leq Ck$. We also note that $(1-e^{-q})^{-1} \leq 1+q^{-1}$, and thus $(1-e^{-q})^{-2} \leq 2+2q^{-2}$. Taking $q=k^{-1/2}$ and then combining the last few expressions, one finally gets $J_2 \leq Ck^{3/2}$.
\\
\\
Combining the bounds of $J_1$ and $J_2$ with \eqref{sack}, we obtain the desired bound.
\end{proof}

Finally we have our concentration theorem, the main result of this appendix.

\begin{thm}[Concentration]\label{conc} As before, let $S = (S_k)_{0\le k \le n}$ denote the canonical process associated to $\mathbf P_x^n$. Then there exist $C,c>0$ such that for every $x \ge 0$, every $0 \le k \le n$, and every $u>0$ one has that $$\mathbf P_x^n\big(\sup_{0 \le i\le k} |S_i-x| >u \big) \le Ce^{-cu^2/k}.$$ In other words, the path measure $\mathbf P_x^n$ concentrates on scales of order $\sqrt n$.
\end{thm}

The idea of the proof is to ``squeeze" the path $S$ in between two martingales $T$ and $M$ which stay reasonably close to $S$, and then apply well-known concentration inequalities for bounded-increment martingales. The Gaussian decay constant $c$ will be obtained as $1/32$, however this is not sharp (though it will suffice for our purposes). It would be interesting to see what the optimal constant is. We conjecture it to be $1/2$, as is the case for the simple random walk.

\begin{proof}
	Throughout this proof, $x,n,$ and $k$ will be \textbf{fixed}. Let us write $$\mathbf P_x^n \big(\sup_{0\leq i \leq k} |S_i-x|>u \big)  = \mathbf P_x^n \big( \sup_{0\leq i \leq k} S_i>x+u \big)+\mathbf P_x^n \big( \inf_{0\leq i \leq k} S_i<x-u \big).$$
	Let us call the terms on the right side as $p_1,p_2$ respectively.
	\\
	\\
	First we bound $p_2$. Recall from Lemma \ref{mono} that $\mathfrak p_1^n(x,x+1) \geq 1/2 \geq \mathfrak p_1^n(x,x-1)$ for all $n,x \ge 0$. Using the same type of coupling argument as in the proof of Proposition \ref{coup}, this means that one may couple $\mathbf P_x^n$ with the law $P_x$ of a simple symmetric random walk $T=(T_i)_{i=0}^k$ of length $k$ started from $x$, in such a way that $S_i \geq T_i$ for all $i$ (or more precisely, such that $S$ takes an upward step whenever $T$ does). Then we have $$p_2 \leq P_x\big( \inf_{0\leq i \leq k} (T_i - x) < u \big)=P_x\big( \sup_{0\leq i \leq k} (T_i - x) > u \big).$$ Now, $e^{\lambda (T_i-x)}$ is a positive submartingale; thus  Doob's inequality says \begin{align}
	\notag &
	P_x\big( \sup_{0\leq i \leq k} (T_i - x) > u \big) \leq Ce^{-\lambda u} E_x[e^{\lambda (T_k-x)}]\\ &\label{dob1}=Ce^{-\lambda u} \cosh(\lambda)^k \leq Ce^{-\lambda u + \frac12 \lambda^2k} = e^{-u^2/2k},\end{align} where we set $\lambda:= u/k$ in the final equality. This proves the bound for $p_2$.
	\\
	\\
	 Now we will bound $p_1$. Letting $M=(M_i^{(n,k)})_{i=0}^k$ denote the martingale from the Proposition \ref{mart}, it is clear that $S_k = M_k$. Furthermore, $M_0 = f(x,0) =\mathbf E_x^n[S_k]  \leq Ck^{1/2}+x$ by Proposition \ref{sqrtgr}. Since the increments of $M$ are bounded above by 2, we may apply Azuma's concentration inequality to see that \begin{align*}\mathbf P_x^n(S_k>x+u) &= \mathbf P_x^n \big(  M_k>x+u \big) \leq \mathbf P_x^n \big( M_k-M_0 > u-Ck^{1/2}\big) \\ &\leq e^{-(u-Ck^{1/2})^2/8k} \leq Ce^{-u^2/16k}.\end{align*}
	In the last inequality, we used the fact that $(u-Ck^{1/2})^2 \geq \frac12 u^2 - C^2k $. This, in turn, is because $(a+b)^2 \leq 2(a^2+b^2)$. Combining the last expression with the bound \eqref{dob1} for $p_2$ shows that \begin{equation}\label{hgb}\mathbf P_x^n(|S_k-x|>u) \leq Ce^{-u^2/16k}.\end{equation} Next, we claim that for any $\lambda>0$, the process $(e^{\lambda S_i})_{i =0}^n$ is a $\mathbf P_x^n$-submartingale. To prove this, fix $i,n,x$ and set $q:= \mathfrak p_1^n(S_i,S_i+1)$. By Lemma \ref{mono} we know $q\ge 1/2$. Thus by convexity of $x \mapsto e^{\lambda x}$ we see that \begin{align}\label{subgale}
	\mathbf E_x^n [ e^{\lambda S_{i+1}} | \mathcal F_i] &\geq e^{\lambda \mathbf E_x^n[ S_{i+1}|\mathcal F_i]} =e^{\lambda\big(q(S_i+1)+(1-q)(S_i-1)\big)} = e^{\lambda S_i} e^{2q-1} \geq e^{\lambda S_i}.
	\end{align}
	Thus, we may apply Doob's inequality to see that $$p_1 \leq Ce^{-\lambda(x+u) } \mathbf E_x^n[e^{\lambda S_k}] = Ce^{-\lambda(x+u)} \bigg( 1+\int_0^{\infty} \lambda e^{\lambda y} \mathbf P_x^n(S_k>y)du \bigg). $$ Now we split the integral as $\int_0^x$ plus $\int_x^{\infty}$. We use the crude bound $\mathbf P_x^n(S_k>y)\leq 1$ for the integral over $[0,x]$, and we use the bound \eqref{hgb} for the other. This gives $$p_1 \leq Ce^{-\lambda u} + Ce^{-\lambda u} \int_x^{\infty} \lambda e^{\lambda (y-x)-(y-x)^2/16k}dy\leq C(e^{-\lambda u} + \lambda k^{1/2} e^{4\lambda^2k - \lambda u }).$$ Setting $\lambda = \frac{u}{8k}$ gives a bound of $C(e^{-u^2/8k} + uk^{-1/2} e^{-u^2/16k}).$ Now one simply notes that $r \leq Ce^{r^2/32}$, so that $uk^{-1/2} \leq Ce^{u^2/32k}$. This gives the desired bound on $p_1$, where the constant appearing in the theorem statement is $c:=1/32$.
\end{proof}

We now give a slightly generalized version of the concentration theorem.

\begin{cor}\label{conc2} In the same setting as the previous theorem, there exist $C,c>0$ such that for every $x \ge 0$, every $0 \le m\le k \le n$, and every $u>0$ one has that $$\mathbf P_x^n\big(\sup_{m \le i\le k} |S_i-S_m| >u \big) \le Ce^{-cu^2/(k-m)}.$$ Here, $C,c$ are the same as in the previous theorem.
\end{cor}

\begin{proof} Define $$g(k,n,x,u):= \mathbf P_x^n\big(\sup_{0 \le i\le k} |S_i-x| >u \big).$$
By the Markov property (conditioning on the first $m$ steps), we have that $$\mathbf P_x^n\big(\sup_{m \le i\le k} |S_i-S_m| >u \big) = \mathbf E_x^{n} \bigg[ g(k-m,n-m,S_m,u) \bigg].$$ But Theorem \ref{conc} tells us that $g(k,n,x,u) \leq Ce^{-cu^2/k}$ independently of $x,n$.
\end{proof}

We now derive an easy consequence of this concentration result, which (by Arzela-Ascoli) is enough to imply tightness of the associated measures if one rescales by a factor of $\sqrt{n}$.

\begin{cor}\label{pmoments}
Let $p> 0$. There exists a constant $C=C_p>0$ such that for every $x\geq 0$ and every $0\le k \le m\leq n$, one has $$\mathbf E_x^n \big[ |S_k-S_m|^p \big] \leq C|k-m|^{p/2}.$$ 
\end{cor}

\begin{proof}
Let us write $$\mathbf E_x^n \big[ |S_k-S_m|^p \big] = \int_0^{\infty} pu^{p-1} \mathbf P_x^n ( |S_k-S_m|>u)du.$$ By Corollary \ref{conc2}, this is bounded above by $$C\int_0^{\infty} pu^{p-1}e^{-cu^2/(k-m)}du = Cp(k-m)^{1/2}\int_0^{\infty} v^{p-1} e^{-cv^2}dv= C_p (k-m)^{1/2}, $$ where we made a substitution $y = (k-m)^{-1/2}u$ in the first equality.
\end{proof}

Using this lemma, we can actually recover the results of \cite{Ig74} quite easily, but only for this nearest-neighbor case. Indeed, fix $X,T\ge 0$. For each $x,N \ge 0$, let $(S^{x,N}_n)_{n=0}^{N}$ be distributed according to $\mathbf P_x^N$. Then we claim that the processes $(N^{-1/2}S_{Nt}^{N^{1/2}X,NT})_{t \in [0,T]}$ converge in law (with respect to the uniform topology on $C[0,T]$, as $N \to \infty$) to $\mathbf W_X^T$.
\\
\\
To see this, first note that convergence of finite-dimensional distributions is clear from inducting on the $L^1$ convergence of pdf's in Theorem \ref{gence}. Furthermore, tightness of the laws of the rescaled processes follows immediately from Kolmogorov's continuity criterion, together with Arzela-Ascoli and Corollary \ref{pmoments}. Then the claim follows immediately, since any limit point on $C[0,T]$ of the rescaled laws must have the same finite-dimensional marginals as $\mathbf W_X^T$.
\\
\\
In fact it is also possible to recover the result of \cite{BJD06} (only in the nearest-neighbor case) very easily from Corollary \ref{pmoments}. The point is to realize that $\lim_{N \to \infty} \mathfrak p_n^N(x,y) = p_n^{(1/2)}(x,y)\frac{y+1}{x+1}$. This is quite easily shown to converge to the density of the three-dimensional Bessel process, and from this one can actually obtain convergence of finite-dimensional marginals quite easily. But the bound in Corollary \ref{pmoments} was independent of the terminal time, which implies tightness in the Holder space, completing the proof.
\\
\\
These invariance principles will not actually be needed in the main body of the paper, but they illustrate the power of the concentration theorem. The point is that it gives an extremely strong quantitative bound on the fluctuations of a typical path, and this will be repeatedly illustrated by its use in Sections 3 and 5.

\end{appendix}

\end{document}